%% file: paper.tex
\documentclass{article}
\usepackage{amsmath,amssymb}
\usepackage{enumerate,a4}
\usepackage{xcolor,colortbl}
\usepackage{hyperref,color}
\usepackage{graphicx, here}
\usepackage{multicol}
\usepackage[font=footnotesize,labelfont=bf]{caption}
\usepackage{booktabs}
\usepackage{multirow}
\usepackage{tabularx}
\usepackage{siunitx}
\usepackage{tikz}
\usepackage{amsthm}
\usepackage{amsmath}
\usepackage{adjustbox}
\usepackage{comment}
\usepackage{arydshln}
\usepackage{verbatim}
\usepackage{afterpage}
\usepackage{algpseudocode}
\usepackage{array}
\usepackage{arydshln}
\usepackage[numbers,sort&compress]{natbib}

 \usepackage{pgfplots}
 \usepgfplotslibrary{groupplots}

\usepackage[ruled,vlined,linesnumbered]{algorithm2e}
\SetKwProg{Fn}{Function}{:}{}

\usepackage{graphicx}

\theoremstyle{plain}
\newtheorem{theorem}{Theorem}
\newtheorem{proposition}[theorem]{Proposition}
\newtheorem{lemma}[theorem]{Lemma}

\theoremstyle{definition}
\newtheorem{definition}[theorem]{Definition}

\newtheoremstyle{remarkstyle}
  {3pt}   
  {3pt}   
  {\normalfont} 
  {}      
  {\itshape} 
  {.}     
  {.5em}  
  {\thmname{#1}~\thmnumber{#2}\thmnote{~(#3)}}

\theoremstyle{remarkstyle}
\newtheorem{remark}[theorem]{Remark}

\DeclareMathOperator*{\argmin}{arg\,min}

\title{Improving Upper Bounds for the Maximum Clique Problem using Reduction Rules}
\author{
	{Alja\v z Krpan} \thanks{Faculty of mathematics and physics, University of Ljubljana, Slovenia {\tt krpan.aljaz@gmail.com}} \thanks{Rudolfovo, Science and Technology Centre Novo mesto, Slovenia}
	\and {Janez Povh} \thanks{ Rudolfovo, Science and Technology Centre Novo mesto, Slovenia {\tt janez.povh@rudolfovo.eu}}
    }
	
\date{}

\begin{document}

\newcommand{\Arrow}[1]{
    \parbox{.15cm}{\tikz{\draw[->](0,0)--(.15cm,0);}}
}

\maketitle

\begin{abstract}
    \noindent

    We study the interaction between reduction rules and upper-bound functions for the Maximum Clique Problem (MCP). We show how MCP upper-bound functions can strengthen classical core and truss reductions by replacing local size conditions with upper-bound tests. This leads to the \((k,\omega^u)\)-core, the \((k,\omega^u)\)-truss, and the more general \((k,d,\omega^u)\)-truss, where the parameter \(d\) controls the trade-off between stronger reductions and additional computational cost. For each of these notions, we prove clique-preservation properties, correctness of the corresponding peeling algorithm, and running-time bounds. Based on these reductions, we introduce a general framework for improving upper-bound values for MCP. We give two concrete instantiations of the framework: one that uses only the combined truss and core reductions, and one that combines the truss and core reductions with repeated applications of structions. Computational experiments on 73 benchmark graphs show that the proposed reductions can substantially improve several standard upper-bound functions and that combining multiple reduction methods can be beneficial in practice. In particular, the combination of structions, truss and core reductions with a DSatur-based bound often reached SDP-level upper-bound values faster than direct SDP computation; on the tested graphs with edge density below \(0.7\), it did so in every case. Using the truss and core reduction with the Lov\'asz theta upper-bound function, we also improve the previously best certified integer upper-bound values for three difficult DIMACS instances whose exact clique numbers are not known. In particular, we improve upper-bound values for graph
    \texttt{C500.9} from 83 to 73, for graph \texttt{C1000.9} from 122 to 115, and for graph \texttt{C2000.9} from 177 to 168.
    \\[1mm]
    \noindent\textbf{Keywords:} Maximum Clique Problem; upper bounds; reductions; core decomposition; truss decomposition; structions; DIMACS benchmarks.\\[1mm]
    \noindent\textbf{MSC 2020:} Primary 05C85; Secondary 05C69, 90C27, 68R10, 68W40.
\end{abstract}

\section{Introduction}\label{intro}

\subsection{Motivation}
The \emph{Maximum Clique Problem} (MCP) is among the most extensively studied problems in graph theory and combinatorial optimization. Given a graph \(G\), the MCP asks to determine the clique number \(\omega(G)\), that is, the size of a largest subset of pairwise adjacent vertices. In the constructive variant, one also asks to output such a clique. Since MCP is NP-hard~\cite{Karp1972}, it has received sustained attention in both theoretical and applied research. Practical instances arise in numerous areas, including bioinformatics (e.g., protein--protein interaction prediction)~\cite{yang2014ppi}, social network analysis (e.g., community detection)~\cite{goswami2022community}, and structural biology or drug design (e.g., local structure alignment)~\cite{lee2016glosa}.

Since MCP is NP-hard, practical algorithms often rely on techniques that reduce the search space. Two important tools for this purpose are reduction rules and upper-bound functions. Reduction rules simplify the input graph while preserving the clique sizes relevant to the problem, and are widely used in preprocessing, exact solvers, and local-search algorithms~\cite{SANSEGUNDO201681,Walteros2020CliqeEasy,Dahlum2016AcceleratingLocalSearch}. On the other hand, upper-bound functions certify that a graph cannot contain a clique larger than a certain value, without necessarily finding a maximum clique. A common example of this is pruning in branch-and-bound algorithms.

In this paper, we study the interaction between reductions and upper-bound functions. On one hand, upper-bound functions can make reductions stronger. For example, the usual \(k\)-core and \(k\)-truss reductions use neighborhood sizes as simple upper bounds when deciding which vertices or edges to keep. We replace these simple estimates by general MCP upper-bound functions, which can make the reductions more selective. On the other hand, reductions can improve upper-bound values. If a reduction preserves all cliques of a target size, then computing an upper bound on the reduced graph can certify that the original graph has no clique of that size. Since the reduced graph is usually smaller or structurally simpler, the upper bound may also be tighter or easier to compute. Together, these observations show that reductions and upper-bound functions can interact in useful ways. In this paper, we use this connection, on one hand, to strengthen reduction rules and, on the other hand, to design reduction-based procedures for improving upper-bound values.

\subsection{Our contribution}
In this paper, we study the relationship between reduction rules and upper-bound functions for MCP. We analyze how tighter upper-bound values can produce stronger reductions and how reduction rules can, in turn, improve existing upper-bound values. Based on this relationship, we propose a general framework for improving MCP upper-bound values. We support our theoretical results with an extensive computational study.

The main contributions of this work are summarized as follows:

\begin{itemize}
    \item \textbf{Strengthening reduction rules using upper-bound functions:}
          We show how existing upper-bound functions can be used inside the \(k\)-core and \(k\)-truss reductions. For cores, this fits into the known framework of generalized cores~\cite{Batagelj2011GeneralizedCores}; the contribution here is to use this viewpoint with MCP upper-bound functions, giving the \((k,\omega^u)\)-core, and to connect it to clique-preserving reductions used for upper-bound value improvement. For trusses, we transfer the same idea from vertices to edges, obtaining the \((k,\omega^u)\)-truss, which can provide stronger reductions. We further extend the \((k,\omega^u)\)-truss to the \((k,d,\omega^u)\)-truss. The additional parameter \(d\) allows us to trade additional computational time for stronger reductions. For these three notions, we prove clique-preservation properties, correctness of the corresponding peeling algorithms, and running-time bounds.

    \item \textbf{Improving MCP upper-bound values using reduction rules:}
          We propose a general algorithmic framework for improving upper-bound values for MCP using reduction rules. This framework is an important contribution because it gives a systematic way to identify and construct new upper-bound improvement algorithms from suitable reductions and upper-bound functions. We demonstrate this by giving two concrete instantiations, one based on truss and core reductions and one combining truss and core reductions with a known reduction called structions. By comparing the two instantiations, we observe that combining different types of reductions can be beneficial in practice.

    \item \textbf{Algorithmic development and empirical evaluation:}
          We implement the proposed reductions and upper-bound improvement framework as practical algorithms. We conduct an extensive computational study on 73 benchmark graphs, showing that the proposed reductions can substantially improve existing upper-bound values. Across all benchmark graphs, this instantiation often reaches SDP-level bounds faster than computing the corresponding SDP bound directly; on graphs with density below \(0.7\), it does so in all tested cases. The strongest experimental outcome is that, using the truss and core instantiation with the Lov\'asz theta upper-bound function, we improved the previously best known upper-bound values for three difficult DIMACS instances whose optimal clique numbers are not known.

\end{itemize}

\subsection{Notation}

In this paper, we let \(\mathcal{G}\) denote the class of all finite, simple, undirected graphs. We use the convention \(\mathbb{N}_0=\{0,1,2,\ldots\}\).

Throughout the paper, \(G=(V(G),E(G))\) denotes an element of \(\mathcal{G}\), where \(|V(G)|=n\) and \(|E(G)|=m\). When the graph is clear from the context, we simply write \(G=(V,E)\). We write \(\overline{uv}\) for the edge joining vertices \(u\) and \(v\). For a vertex \(v\in V(G)\), its neighborhood in \(G\) is
\[
    N_G(v):=\{u\in V(G) \mid \overline{uv}\in E(G)\}.
\]
When \(G\) is clear from the context, we write \(N(v)\) instead of \(N_G(v)\).

We write \(H\subseteq G\) if \(H\) is a subgraph of \(G\). For \(S\subseteq V(G)\), we denote by \(G[S]\) the subgraph of \(G\) induced by \(S\). The graph obtained by deleting a vertex \(v\) is denoted by \(G\setminus v\), and the graph obtained by deleting an edge \(\overline{uv}\) is denoted by \(G\setminus \overline{uv}\). The complement of \(G\) is denoted by \(\overline{G}\).

We assume that every graph considered in this paper has vertices labeled by distinct natural numbers, and we use the natural order of these labels as the fixed vertex ordering. This assumption is without loss of generality, since any finite graph can be labeled in this way. Subgraphs inherit the labels and hence the ordering. If \(v\in V(G)\), we define the forward neighborhood of \(v\) by
\[
    N_G^+(v):=\{u\in N_G(v)\mid v<u\}.
\]
When \(G\) is clear from the context, we write \(N^+(v)\) instead of \(N_G^+(v)\). If a different ordering is used, this will be stated explicitly.

Throughout the paper, \textsc{pick\_and\_remove}\((Q)\) denotes an operation that removes and returns an element from a given set \(Q\). In the analysis of our algorithms, we assume that this operation takes \(\mathcal{O}(1)\) time, which can be achieved, for example, by using a stack or a queue.

\subsection{Outline}
The rest of the paper is organized as follows. In Section~\ref{sec:related_work}, we discuss selected related work on reduction techniques, upper-bound functions, and methods for improving upper-bound values. In Section~\ref{sec:upper_bounds}, we recall the MCP upper-bound functions used later in the paper. Sections~\ref{sec:core} and~\ref{sec:truss} introduce the \((k,\omega^u)\)-core, the \((k,\omega^u)\)-truss, and the \((k,d,\omega^u)\)-truss, together with results on existence, uniqueness, clique preservation, and correctness and complexity of the corresponding algorithms. The longer correctness and complexity proofs for the \((k,d,\omega^u)\)-truss are given separately in Section~\ref{sec:truss-correctness-complexity}. In Section~\ref{sec:practical_reductions}, we describe the practical reduction procedures used in the experiments, including the combined truss and core reduction and the repeated use of structions. Section~\ref{sec:improving_bounds} presents the general framework for improving MCP upper-bound values using reductions, together with two concrete instantiations. In Section~\ref{sec:comp_results}, we evaluate these instantiations on several upper-bound functions and benchmark graphs, and we also report improved upper-bound values for three DIMACS instances whose exact clique numbers are not known. Finally, Section~\ref{sec:conclusions} summarizes the main findings and discusses directions for future work.

\section{Related Work}\label{sec:related_work}
This section is not intended to be a comprehensive survey. Instead, we discuss selected works on reduction techniques, upper-bound functions, and methods for improving upper-bound values that are closely related to the topics considered in this paper. For a broader survey of recent approaches to MCP, we refer the reader to~\cite{marino2024shortreviewnovelapproaches}.

\subsection{Reduction techniques}

Reduction techniques for MCP are typically studied alongside those for the Maximum Independent Set (MIS) and Minimum Vertex Cover (MinVC) problems, since standard complement-based transformations yield polynomial-time equivalences between these three problems. A comprehensive survey of reduction techniques for the maximum weighted independent set problem is provided in~\cite{grosmann2024comprehensive}. Through the standard complement transformation between MCP and MIS, many of these reductions are also relevant to MCP. Although that survey focuses on weighted variants, its results directly apply to the unweighted setting by assigning weight~1 to every vertex.

The comprehensive survey on MCP contains an extensive list of reduction and kernelization techniques for MCP and related problems~\cite[Section~2.2]{marino2024shortreviewnovelapproaches}. One work from this line of research that we want to highlight is~\cite{Hespe2017scalable}, which develops a scalable preprocessing framework for the Maximum Independent Set problem by combining several reduction rules. The efficiency of their approach comes not only from parallelization, but also from implementation techniques such as dependency checking and reduction tracking, which avoid unnecessary reduction tests and stop kernelization when further reductions become less productive. A different application-oriented use of preprocessing appears in~\cite{math10050697}, where MCP is used to solve a scheduling problem and a large number of simple preprocessing techniques is applied to the resulting MCP instances before running the solver. Reduction rules were also used dynamically inside solving procedures, rather than only as preprocessing, for example in branch-and-reduce algorithms for the related Minimum Vertex Cover problem~\cite{akiba2016branchreduce}.

Two concepts that are often used as reduction rules are the \(k\)-core and the \(k\)-truss. The \(k\)-core was introduced in~\cite{Seidman1983network} and has since become a standard tool in network analysis. For broader discussions of \(k\)-core decompositions, we refer the reader to~\cite{Kong2019kcoreTheories,Sariyuce2016kcore,Malliaros2020CoreDecomposition}. In MCP, degeneracy and core-based bounds are used to reduce the search space in exact algorithms~\cite{SANSEGUNDO201681,Walteros2020CliqeEasy}. Generalized cores, based on vertex property functions, were introduced in~\cite{Batagelj2011GeneralizedCores}. The \((k,\omega^{u})\)-core used in this paper can be viewed as an MCP-specific instance of this framework when \(\omega^{u}\) is increasing.

The \(k\)-truss is a triangle-based cohesive subgraph notion. A closely related notion, called \(k\)-dense subgraphs, was introduced by Saito and Yamada~\cite{Saito2006_kDense}. The \(k\)-truss model was later introduced by Cohen~\cite{Cohen2008trusses}. Similar to \(k\)-core, the \(k\)-truss is also used as a reduction tool in MCP algorithms, for example in~\cite{Lu2017Massive}. Efficient truss-decomposition algorithms and implementations have been studied in~\cite{Wang2012TrussDecomposition} and \cite{Kabir2017ParallelTruss}. A general clique-based framework is given by nucleus decompositions~\cite{Sariyuce2017Nucleus}. In this framework, lower-order cliques are required to be contained in sufficiently many higher-order cliques, and both the \(k\)-core and the \(k\)-truss appear as special cases. Other generalizations of truss decompositions also exist, for example generalizations based on higher-order neighborhoods~\cite{Chen2022HigherOrderTruss}.

Our \((k,\omega^{u})\)-truss and \((k,d,\omega^{u})\)-truss are different from nucleus decompositions and from these other truss generalizations. Our \((k,\omega^{u})\)-truss and \((k,d,\omega^{u})\)-truss differ from existing truss generalizations because they incorporate an increasing upper-bound function. Together with the corresponding core-based notions, they yield reductions for improving MCP upper bounds.
\subsection{Improving upper-bound values}

Improved upper-bound values for MCP are one of the main contributions of our paper. The work most closely related to ours is the study~\cite{GENDRON2008Sequential} from 2008, where the Sequential Elimination Algorithm (SEA) was introduced. Given any upper-bound function, SEA iteratively removes vertices according to upper bounds computed on their neighborhoods and updates a global upper-bound value for \(\omega(G)\). In~\cite{Curzi2012RepeatedSequential}, the authors build directly on this framework and propose a repeated version of SEA, in which the original sequential scheme is applied multiple times on suitable subgraphs, yielding empirically stronger upper-bound values. A closely related formulation was later given in~\cite{szabo2021discarding} under the name DISCARDING. It uses the same neighborhood-based elimination principle to construct a vertex deletion order and derive a global upper bound, with particular emphasis on coloring-based bounds and parallel implementation.

Other related work mainly focuses on strengthening particular upper-bound functions by exploiting their specific structural properties. For example, the coloring upper-bound function has been strengthened by infra-chromatic upper bounds~\cite{SANSEGUNDO2015Infra, SANSEGUNDO2016InfraImproved}, which can yield tighter upper-bound values. Another line of research aims to use upper-bound information more efficiently within branch-and-bound solvers by reusing previously computed information or by strengthening the pruning step; examples include incremental MaxSAT reasoning~\cite{Jiang2016EfficientPreprocessing} and the incremental upper bound IncUB proposed in~\cite{Li2018Inc}.

For bounds themselves, a wide variety of results for the clique number and the independence number has been proposed in the literature. For overviews of different families of such bounds, we refer the reader to~\cite{Brimkov2024BoundsCertainClasses, larson2012alphabounds, Margenov2019ApplicationsUltrascale, Willis2011Bounds}.

\section{Upper-bound functions for the Maximum Clique Problem}\label{sec:upper_bounds}

In this section, we first formally define the Maximum Clique Problem (MCP). We then introduce several upper-bound functions for its optimal value, which will be used later in the paper.

\begin{definition}
    Let \(G=(V,E)\in\mathcal{G}\). A set \(S\subseteq V\) is a clique in \(G\) if every two distinct vertices in \(S\) are adjacent, that is,
    \[
        \overline{uv}\in E \qquad \text{for all distinct } u,v\in S.
    \]
    The clique number of \(G\) is defined as
    \[
        \omega(G) := \max_{\substack{S\subseteq V \\ S \text{ is a clique}}} |S|.
    \]
    The Maximum Clique Problem asks for a clique \(S\subseteq V\) such that \(|S|=\omega(G)\).
\end{definition}

We now define upper-bound functions and upper-bound values for the clique number.
\begin{definition}
    \label{def:upper_bound}
    A function \(\omega^u:\mathcal{G}\rightarrow\mathbb{N}_0\) is an MCP upper-bound function if, for every graph \(G=(V,E)\in\mathcal{G}\),
    \[
        \omega(G) \leq \omega^u(G) \leq |V|.
    \]
    Any integer \(U\) satisfying \(\omega(G)\leq U\leq |V|\) is called an upper-bound value on \(\omega(G)\).
\end{definition}

The restriction \(\omega^u(G)\in\mathbb{N}_0\) is assumed without loss of generality. If an upper-bound function \(\omega^u_0\) returns real values, we can replace it by \(\lfloor \omega^u_0(G)\rfloor\), since the clique number \(\omega(G)\) is an integer. Similarly, the condition \(\omega^u(G)\leq |V|\) is assumed without loss of generality, since \(|V|\) is the trivial upper-bound value for the clique number. Therefore, if an upper-bound function \(\omega^u_0\) does not satisfy these conditions, we can replace it in constant time by
\[
    \omega^u(G) := \min\{\lfloor\omega^u_0(G)\rfloor, |V|\}.
\]
The resulting upper-bound function is never weaker than \(\omega^u_0\).

In the generalized core and truss definitions introduced later, uniqueness of the largest subgraph is not automatic for an arbitrary upper-bound function. To ensure that these definitions are well-defined, we require the upper-bound function to be monotone with respect to taking subgraphs.

\begin{definition}
    Let \(\omega^u:\mathcal{G}\rightarrow\mathbb{N}_0\) be an MCP upper-bound function. We say that \(\omega^u\) is increasing if, for all graphs \(G,G'\in\mathcal{G}\),
    \[
        G\subseteq G' \implies \omega^u(G) \leq \omega^u(G').
    \]
\end{definition}

We also use an edge-sensitivity property to distinguish upper-bound functions whose value can change under edge deletion; this will allow us to state more precise running-time bounds for the corresponding reduction algorithms.

\begin{definition}
    Let \(\omega^u\) be an MCP upper-bound function. We say that \(\omega^u\) is edge-sensitive if there exists a graph \(G=(V,E)\in\mathcal{G}\) and an edge \(\overline{uv}\in E\) such that
    \[
        \omega^u(G) \neq \omega^u(G\setminus \overline{uv}).
    \]
\end{definition}

In the following, we list the MCP upper-bound functions that will be used in this paper. They are ordered approximately by their computational cost.

\begin{enumerate}
    \item \textbf{Trivial:} The trivial upper-bound function returns \(n=|V|\), since \(\omega(G)\leq n\). If \(n=|V|\) is already stored, this value is computed in \(\mathcal{O}(1)\) time; otherwise, computing it requires \(\mathcal{O}(n)\) time.

    \item \textbf{Density:} Since every clique of size \(\omega(G)\) contains \(\binom{\omega(G)}{2}\) edges, we have \(\binom{\omega(G)}{2}\leq m\). Hence,
          \[
              \omega(G)\leq \left\lfloor \frac{1+\sqrt{1+8m}}{2} \right\rfloor.
          \]
          If \(m=|E|\) is already stored, this upper-bound value is computed in \(\mathcal{O}(1)\) time; otherwise, computing it requires \(\mathcal{O}(n+m)\) time.

    \item \textbf{Degree:} Every vertex in a maximum clique has at least \(\omega(G)-1\) neighbors. Therefore,
          \[
              \omega(G)\leq \Delta(G)+1.
          \]
          If the maximum degree \(\Delta(G)\) is not already stored, this upper-bound value can be computed in \(\mathcal{O}(n+m)\) time.

    \item \textbf{DSatur coloring:} Any proper coloring of \(G\) with \(c\) colors gives an upper-bound value \(c\), since all vertices of a clique must receive distinct colors. We compute such a coloring using the DSatur heuristic introduced in~\cite{brelaz1979new}. A straightforward implementation has time complexity \(\mathcal{O}(n^2)\). Although priority-queue implementations may be preferable on sparse graphs, we use the straightforward implementation in our experiments, since our most difficult instances are dense.

    \item \textbf{Lov\'asz theta:} The Lov\'asz theta function gives an upper-bound function \(\vartheta(G)\) for the independence number \(\alpha(G)\); see~\cite{lovasz1979shannon}. For MCP, we therefore obtain
          \(\omega(G)=\alpha(\overline{G})\leq \lfloor \vartheta(\overline{G}) \rfloor\).
          Computing this upper-bound value requires solving a semidefinite programming problem and is therefore computationally expensive. Let \(\overline{m}=|E(\overline{G})|\). Using standard primal-dual interior-point methods, an \(\varepsilon\)-accurate solution can be computed in polynomial time; for example, short-step path-following methods have complexity bounds of the order
          \(\mathcal{O}((\overline{m}^3+\overline{m}^2n^2+\overline{m}n^3)\sqrt{n}\log(1/\varepsilon))\)
          floating-point operations; see, for example,~\cite{monteiro1997primal}. Since \(\overline{m}=\mathcal{O}(n^2)\) in the worst case, this gives the worst-case bound
          \(\mathcal{O}(n^{6.5}\log(1/\varepsilon))\).

    \item \textbf{Vector coloring:} Another SDP-based upper-bound function for the MCP arises from the vector chromatic number \(\chi_v(G)\), introduced by Karger, Motwani, and Sudan~\cite{Karger1998ApproximateGraphColoring}. It satisfies
          \(
          \omega(G)\leq \chi_v(G)\leq \vartheta(\overline{G}),
          \)
          where the second inequality follows from the fact that the strict vector chromatic number coincides with the Lov\'asz theta function of the complement graph~\cite[Theorem 8.2]{Karger1998ApproximateGraphColoring}. The first inequality is immediate, since any clique of size \(k\) requires \(k\) mutually constrained vectors in any vector \(k\)-coloring of \(G\). Computing \(\chi_v(G)\) also requires solving a semidefinite programming problem. Here the linear constraints are associated with the edges of \(G\), so let \(m=|E(G)|\). Using the same class of interior-point methods as for the Lov\'asz theta function, an \(\varepsilon\)-accurate solution can be computed in
          \(\mathcal{O}((m^3+m^2n^2+mn^3)\sqrt{n}\log(1/\varepsilon))\)
          floating-point operations. Since \(m=\mathcal{O}(n^2)\) in the worst case, this again gives the worst-case bound
          \(\mathcal{O}(n^{6.5}\log(1/\varepsilon))\).

          Thus, vector coloring can give an upper-bound value for \(\omega(G)\) at least as tight as the Lov\'asz theta value. In practice, however, the relative cost depends on graph density: the Lov\'asz theta bound for MCP is computed on \(\overline{G}\), while vector coloring is computed on \(G\). Hence, Lov\'asz theta is often more feasible on dense graphs, whereas vector coloring can be preferable on sparse graphs.

\end{enumerate}

All upper-bound functions considered in this paper are edge-sensitive, with the exception of the trivial upper-bound function.

It is immediate that the trivial, density, and degree upper-bound functions are increasing. The upper-bound function \(\omega^u(G)=\lfloor \vartheta(\overline{G}) \rfloor\) is increasing with respect to taking subgraphs. Indeed, Lovász observes in~\cite[Section~11.1]{lovasz2019graphs} that if \(G'\) is an induced subgraph of \(G\), then \(\vartheta(G')\leq \vartheta(G)\), while if \(G'\) is a spanning subgraph of \(G\), then \(\vartheta(G')\geq \vartheta(G)\). Let \(H\subseteq G\) and let \(S=V(H)\). Since \(\overline{G[S]}\) is a spanning
subgraph of \(\overline{H}\), we have
\(
\vartheta(\overline{H})\leq \vartheta(\overline{G[S]}).
\)
Moreover, \(\overline{G[S]}\) is an induced subgraph of \(\overline{G}\), and hence
\(
\vartheta(\overline{G[S]})\leq \vartheta(\overline{G}).
\)
Therefore,
\[
    \omega^u(H)
    =
    \lfloor\vartheta(\overline{H})\rfloor
    \leq
    \lfloor\vartheta(\overline{G[S]})\rfloor
    \leq
    \lfloor\vartheta(\overline{G})\rfloor
    =
    \omega^u(G).
\]

The upper-bound function \(\omega^u(G)=\lfloor \chi_{v}(G) \rfloor\) is also increasing with respect
to taking subgraphs. Indeed, in the SDP formulation of vector coloring, the matrix variables are indexed by vertices, while the linear constraints are imposed only on entries corresponding to adjacent pairs of vertices. Thus, if \(H\subseteq G\), the formulation for \(H\) is obtained by keeping only the rows and columns indexed by \(V(H)\), together with the edge constraints corresponding to \(E(H)\). Equivalently, any feasible solution for \(G\) restricts to a principal submatrix feasible for \(H\); this principal submatrix remains positive semidefinite, and all edge constraints of \(H\) are inherited from \(G\). Hence \(\chi_v(H)\leq \chi_v(G)\), and therefore \(\omega^u(H)\leq \omega^u(G)\).

In contrast, the DSatur upper-bound function is not guaranteed to be increasing. Because DSatur is a heuristic, adding edges may change the coloring order and can lead to a coloring with fewer colors, and therefore to a smaller upper-bound value.

\section{The \(k\)-core and \((k,\omega^u)\)-core}
\label{sec:core}
In this section, we first recall the classical \(k\)-core reduction and its connection to the Maximum Clique Problem. We then introduce a stronger variant, called the \((k,\omega^u)\)-core, where the degree condition is replaced by a general upper-bound function for the clique number of the neighborhood.

\subsection{The \(k\)-core}

The concept of a \(k\)-core is well known and was originally introduced in~\cite{Seidman1983network}. Its connection to MCP follows from a simple necessary condition: every vertex in a clique of size \(k\) has at least \(k-1\) neighbors among the other vertices of the clique.

\begin{definition}
    Let \(G = (V,E)\) be a graph and let \(k \geq 0\) be an integer.
    A \(k\)-core of \(G\) is a largest induced subgraph \(G' = G[V']\)
    such that \(|N_{G'}(v)| \geq k\) for every vertex \(v \in V'\).
\end{definition}

Existence is immediate, since \(G\) is finite and the empty induced subgraph satisfies the condition vacuously. Uniqueness follows from the following proposition:

\begin{proposition}
    \label{prop:k-core_unique}
    A \(k\)-core of a graph is unique.
\end{proposition}

\begin{proof}
    Let \(C_1 = G[V_1]\) and \(C_2 = G[V_2]\) be two \(k\)-cores of \(G\).
    Then \(G[V_1 \cup V_2]\) also has minimum degree at least \(k\), since each
    vertex keeps all its neighbors from either \(C_1\) or \(C_2\). Because
    \(C_1\) and \(C_2\) have maximum order among such induced subgraphs, we have
    \(|V_1 \cup V_2| = |V_1| = |V_2|\). Hence \(V_1 = V_2\), and therefore
    \(C_1 = C_2\).
\end{proof}

By Proposition~\ref{prop:k-core_unique}, we may refer to a \(k\)-core of \(G\) as the \(k\)-core of \(G\). The \(k\)-core can be computed by the standard peeling procedure shown in Algorithm~\ref{alg:k-core}, which repeatedly removes vertices whose current degree is smaller than \(k\).

\begin{algorithm}[ht]
    \caption{\(k\)-core computation}
    \label{alg:k-core}

    \DontPrintSemicolon
    \small
    \SetAlgoLined
    \LinesNumbered
    \SetAlgoNlRelativeSize{-1}
    \SetKwInOut{Input}{Input}
    \SetKwInOut{Output}{Output}
    \SetKwInOut{Global}{Global}

    \SetKwProg{Proc}{ }{}{ }

    \Input{Graph \(G=(V,E)\), integer \(k \geq 0\)}
    \Output{The \(k\)-core of \(G\)}

    \BlankLine
    \BlankLine

    \Proc{\textsc{core}\((G, k)\)}{
        \(Q_V \gets V\)\;
        \While{\(Q_V \neq \emptyset\)}{
            \(v \leftarrow \textsc{pick\_and\_remove}(Q_V)\)\;
            \If{\(|N_G(v)| < k\)}{
                \(Q_V \gets Q_V \cup N_G(v)\)\;
                \(G \gets G \setminus v\)\;
            }
        }
        \Return \(G\)\;
    }
\end{algorithm}

\begin{proposition}
    \label{prop:k-core_corr_comp}
    The function \textsc{core}\((G, k)\) computes the \(k\)-core of the graph \(G\) in time \(\mathcal{O}(n+m)\).
\end{proposition}
\begin{proof}
    The algorithm removes exactly those vertices whose degree in the current graph is smaller than \(k\). Whenever a vertex is removed, only its neighbors can have their degree decreased, and therefore these vertices are inserted back into \(Q_V\). Hence, when the algorithm terminates, every remaining vertex has degree at least \(k\).

    It remains to show that no vertex belonging to a \(k\)-core is removed. Let \(L\) be the largest induced subgraph with minimum degree at least \(k\), and suppose that the algorithm removes a vertex of \(L\). Let \(w\) be the first such vertex, and let \(H\) be the current graph right before the moment when \(w\) is removed. Since no vertex of \(L\) was removed before \(w\), we have \(N_L(w) \subseteq N_H(w)\). Thus \(|N_H(w)| \geq |N_L(w)| \geq k\), which contradicts the rule by which \(w\) is removed from \(H\). Therefore all vertices of \(L\) remain in \(G'\). Since every vertex of \(G'\) has degree at least \(k\) in \(G'\), and \(L\) is a largest induced subgraph with this property, \(G'\) cannot contain any vertex outside \(L\). Hence \(G'=L\).

    For the running time, assume that the current degree of each vertex is maintained throughout the algorithm; initially, \(Q_V\) contains \(n\) vertices. Whenever a vertex \(v\) is removed, all vertices in \(N(v)\) are inserted into \(Q_V\). Thus the total number of insertions caused by removals is at most \(\sum_{v \in V} |N(v)| = 2m\). We charge each such insertion to the later iteration in which the same vertex is processed by \textsc{pick\_and\_remove}. Hence the cost of inserting a vertex into \(Q_V\) and later processing it is counted together as \(\mathcal{O}(1)\). Therefore the total running time is \(\mathcal{O}(n+m)\).

\end{proof}

The following proposition gives the reduction property that makes the \(k\)-core useful for MCP:

\begin{proposition}
    \label{prop:k-core}
    Let \(G = (V,E)\) be a graph, let \(K \subseteq V\) be a set of size \(k\geq 1\), and let \(G' = \textsc{core}(G,k-1)\). Then:
    \[K \text{ is a clique in } G \Longleftrightarrow K \text{ is a clique in } G'.\]

\end{proposition}
\begin{proof}
    The implication \((\Leftarrow)\) follows directly from \(G' \subseteq G\). For \((\Rightarrow)\), suppose that \(K\) is a clique in \(G\), but \(K\) is not contained in \(G'\). Let \(v \in K\) be the first vertex of \(K\) removed by the algorithm. At the moment when \(v\) is removed, all other vertices of \(K\) are still present. Hence \(v\) has at least \(k-1\) neighbors in the current graph, contradicting the rule that only vertices with fewer than \(k-1\) neighbors are removed. Therefore \(K \subseteq V(G')\), and since \(G'\) is an induced subgraph of \(G\), \(K\) is a clique in \(G'\).
\end{proof}

By computing the \((k-1)\)-core of a graph, we can reduce the graph while preserving all cliques of size \(k\). Consequently, all cliques of size at least \(k\) are also preserved.

\subsection{The \((k, \omega^u)\)-core}
We first consider an MCP-specific instance of generalized cores~\cite{Batagelj2011GeneralizedCores}, obtained by replacing the usual degree condition with an upper-bound condition on the neighborhood. We observe that the condition \(|N(v)|<k\) uses the size of the neighborhood \(N(v)\) as a trivial upper-bound value for the clique number of the induced subgraph \(G[N(v)]\). Indeed, if a clique of size \(k+1\) contains a vertex \(v\), then the remaining \(k\) vertices must form a clique inside \(G[N(v)]\).

This observation suggests a natural improvement: instead of using \(|N(v)|\), we can apply an MCP upper-bound function, as defined in Definition~\ref{def:upper_bound}, to the neighborhood subgraph \(G[N(v)]\), denoted by \(\omega^u(G[N(v)])\). In this way, we may obtain a more selective core and therefore a stronger reduction. For simplicity, when the graph \(G\) is clear from the context, we write \(\omega^u(N(v))\), instead of \(\omega^u(G[N_G(v)])\). This idea leads to the definition of the \((k,\omega^u)\)-core.

\begin{definition}
    Let \(G = (V,E)\) be a graph, let \(k \geq 0\) be an integer, and let \(\omega^u\) be an MCP upper-bound function. A \((k,\omega^u)\)-core of \(G\) is a largest induced subgraph \(G' = G[V']\) such that \(\omega^u(G'[N_{G'}(v)]) \geq k\) for every vertex \(v \in V'\).
\end{definition}

Thus, when \(\omega^u\) is increasing, the \((k,\omega^u)\)-core is an instance of the generalized-core framework~\cite{Batagelj2011GeneralizedCores}. We nevertheless state the definition and prove the needed properties in our setting, both for completeness and because similar ideas are later used to generalize the \(k\)-truss.

As before, existence is immediate, since \(G\) is finite and the empty induced subgraph satisfies the condition vacuously. If \(\omega^u\) is increasing, uniqueness follows from the following proposition:

\begin{proposition}
    \label{prop:k,omega-core_unique}
    If \(\omega^u\) is increasing, then a \((k,\omega^u)\)-core of a graph is unique.
\end{proposition}
\begin{proof}
    Let \(C_1 = G[V_1]\) and \(C_2 = G[V_2]\) be two \((k,\omega^u)\)-cores of \(G\). Consider \(C = G[V_1 \cup V_2]\). For every \(v \in V_1\), we have \(C_1[N_{C_1}(v)] \subseteq C[N_C(v)]\), and since \(\omega^u\) is increasing, \(\omega^u(C[N_C(v)]) \geq \omega^u(C_1[N_{C_1}(v)]) \geq k\). The same argument holds for every \(v \in V_2\) as well. Hence \(C\) satisfies the \((k,\omega^u)\)-core condition. Since \(C_1\) and \(C_2\) have maximum order among induced subgraphs satisfying this condition, \(|V_1 \cup V_2| = |V_1| = |V_2|\). Therefore \(V_1 = V_2\), and so \(C_1 = C_2\).
\end{proof}

When \(\omega^u\) is increasing, Proposition~\ref{prop:k,omega-core_unique} allows us to refer to a \((k,\omega^u)\)-core of \(G\) as the \((k,\omega^u)\)-core of \(G\). The \((k,\omega^u)\)-core can be computed by the peeling procedure shown in Algorithm~\ref{alg:k,omega-core}, where the degree test \(|N_G(v)| < k\) from Algorithm~\ref{alg:k-core} is replaced by the upper-bound value test \(\omega^u(N_G(v)) < k\).

\begin{algorithm}[ht]
    \caption{\((k,\omega^u)\)-core computation}
    \label{alg:k,omega-core}

    \DontPrintSemicolon
    \small
    \SetAlgoLined
    \LinesNumbered
    \SetAlgoNlRelativeSize{-1}
    \SetKwInOut{Input}{Input}
    \SetKwInOut{Output}{Output}
    \SetKwInOut{Global}{Global}

    \SetKwProg{Proc}{ }{}{ }

    \Input{Graph \(G=(V,E)\), integer \(k \geq 0\), MCP upper-bound function \(\omega^u\)}
    \Output{\((k,\omega^u)\)-core of \(G\), if \(\omega^u\) is increasing}

    \BlankLine
    \BlankLine

    \Proc{\textsc{core}\((G, k, \omega^u)\)}{
        \(Q_V \gets V\)\;
        \While{\(Q_V \neq \emptyset\)}{
            \(v \leftarrow \textsc{pick\_and\_remove}(Q_V)\)\;
            \If{\(\omega^u(N_G(v)) < k\)}{
                \(Q_V \gets Q_V \cup N_G(v)\)\;
                \(G \gets G \setminus v\)\;
            }
        }
        \Return \(G\)\;
    }
\end{algorithm}

\begin{proposition}
    \label{prop:k,omega-core_corr_comp}
    Let \(G' = \textsc{core}(G,k,\omega^u)\). Then the following holds:
    \begin{itemize}
        \item If \(\omega^u\) is increasing, then \(G'\) is the \((k,\omega^u)\)-core of \(G\).
        \item The running time is \(\mathcal{O}(n+wm)\), where \(w\) is the amortized time for one computation of \(\omega^u\).
    \end{itemize}
\end{proposition}

\begin{proof}
    The algorithm removes exactly those vertices \(v\) for which \(\omega^u(N_G(v)) < k\) in the current graph. Whenever a vertex is removed, only its neighbors can have their neighborhoods changed, and therefore these vertices are inserted back into \(Q_V\). Hence, when the algorithm terminates, every remaining vertex \(v\) satisfies \(\omega^u(N_G(v)) \geq k\).

    It remains to show that, if \(\omega^u\) is increasing, no vertex belonging to a \((k,\omega^u)\)-core is removed. Let \(L\) be the largest induced subgraph satisfying the \((k,\omega^u)\)-core condition, and suppose that the algorithm removes a vertex of \(L\). Let \(w\) be the first such vertex, and let \(H\) be the current graph right before the moment when \(w\) is removed. Since no vertex of \(L\) was removed before \(w\), we have \(L[N_L(w)] \subseteq H[N_H(w)]\). Since \(\omega^u\) is increasing, \(\omega^u(N_H(w)) \geq \omega^u(N_L(w)) \geq k\), which contradicts the rule by which \(w\) is removed from \(H\). Therefore all vertices of \(L\) remain in \(G'\). Since \(G'\) satisfies \(\omega^u(G'[N_{G'}(v)])\geq k\) for every vertex \(v\in V(G')\), and \(L\) is a largest induced subgraph with this property, \(G'\) cannot contain any vertex outside \(L\). Hence \(G'=L\).

    For the running-time analysis, the case \(k=0\) is immediate, since no vertex is removed. Assume therefore that \(k>0\). Isolated vertices can be detected and removed in \(\mathcal{O}(n+m)\) time. After this preprocessing, \(n=\mathcal{O}(m)\). Initially, \(Q_V\) contains \(n\) vertices. Whenever a vertex \(v\) is removed, all vertices in \(N_G(v)\) are inserted into \(Q_V\). Thus the total number of insertions caused by removals is at most \(\sum_{v \in V} |N_G(v)| = 2m\). We charge each such insertion to the later iteration in which the same vertex is processed by \textsc{pick\_and\_remove}. Hence the cost of inserting a vertex into \(Q_V\), later processing it, and computing \(\omega^u\) is counted together as \(\mathcal{O}(w)\). The edge-based part therefore takes \(\mathcal{O}(wm)\) time. Since \(w=\Omega(1)\), adding the preprocessing term gives the total running time \(\mathcal{O}(n+wm)\).
\end{proof}

The following proposition gives the reduction property that makes the \((k,\omega^u)\)-core useful for MCP:

\begin{proposition}
    \label{prop:k,omega-core}
    Let \(G = (V,E)\) be a graph, let \(\omega^u\) be an MCP upper-bound function, let \(K \subseteq V\) be a set of size \(k\geq 1\), and let \(G' = \textsc{core}(G,k-1,\omega^u)\). Then:

    \[
        K \text{ is a clique in } G \Longleftrightarrow K \text{ is a clique in } G'.
    \]
\end{proposition}
\begin{proof}
    The implication \((\Leftarrow)\) follows directly from \(G' \subseteq G\). For \((\Rightarrow)\), suppose that \(K\) is a clique in \(G\), but \(K\) is not contained in \(G'\). Let \(v \in K\) be the first vertex of \(K\) removed by the algorithm, and let \(H\) be the current graph immediately before \(v\) is removed. At this moment, all other vertices of \(K\) are still present. Hence \(N_H(v)\) contains a clique of size \(k-1\), and therefore \(\omega^u(N_H(v)) \geq k-1\). This contradicts the rule that only vertices with \(\omega^u(N_H(v)) < k-1\) are removed. Therefore \(K \subseteq V(G')\), and since \(G'\) is an induced subgraph of \(G\), \(K\) is a clique in \(G'\).
\end{proof}
By computing the \((k-1,\omega^u)\)-core of a graph, we can reduce the graph while preserving all cliques of size \(k\). Consequently, all cliques of size at least \(k\) are also preserved.

\begin{remark}
    \label{rem:core-preservation-without-increasing}
    Proposition~\ref{prop:k,omega-core} does not assume that \(\omega^u\) is increasing. The increasing property is needed only to prove that Algorithm~\ref{alg:k,omega-core} returns the largest induced subgraph satisfying the \((k,\omega^u)\)-core condition. When the algorithm is used only as a reduction for MCP, it is sufficient that all cliques of size \(k\) are preserved. Hence any valid MCP upper-bound function can be used for this purpose, even if it is not increasing.
\end{remark}

\section{The \((k,\omega^u)\)-truss and \((k, d,\omega^u)\)-truss}
\label{sec:truss}
In this section, we introduce the \emph{\((k,d,\omega^u)\)-truss}, a natural generalization of the \((k,\omega^u)\)-truss. This structure serves as an important building block for the algorithms presented in the rest of the paper.

We start with the concept of the \emph{\(k\)-truss}, which was first introduced in~\cite{Saito2006_kDense} and later formalized in~\cite{Cohen2008trusses}. Closely related notions also appear under the names \(k\)-community~\cite{verma2015MCPsparse} and triangle-core~\cite{Rossi2014FastTriangleCore}. It is defined in a similar way to the \(k\)-core:

\begin{definition}
    Let \(G = (V,E)\) be a graph and let \(k \geq 0\) be an integer. A \(k\)-truss of \(G\) is a largest spanning subgraph \(G' = (V,E')\), where \(E' \subseteq E\), such that \(|N_{G'}(u) \cap N_{G'}(v)| \geq k\) for every edge \(\overline{uv} \in E'\).
\end{definition}
This parameterization differs from the standard one in the truss literature, where a \(k\)-truss usually requires each edge to be contained in at least \(k-2\) triangles. We use the shifted parameterization in order to keep the notation consistent with the \(k\)-core definition, where the parameter denotes the required local support.

As with the \(k\)-core, we can use \(\omega^u\) to obtain a more general definition:

\begin{definition}
    Let \(G = (V,E)\) be a graph, let \(k \geq 0\) be an integer, and let \(\omega^u\) be an MCP upper-bound function. A \((k,\omega^u)\)-truss of \(G\) is a largest spanning subgraph \(G' = (V,E')\), where \(E' \subseteq E\), such that \(\omega^u(G'[N_{G'}(u) \cap N_{G'}(v)]) \geq k\) for every edge \(\overline{uv} \in E'\).
\end{definition}

Instead of treating this case separately, we now introduce a more general notion and prove the main results for that general setting. The above \((k,\omega^u)\)-truss will then be recovered as the special case \(d=0\).

\subsection{The \((k,d,\omega^u)\)-truss}
Our second major contribution in this paper is a generalization of this notion, which we call the \((k,d,\omega^u)\)-truss. The idea is to test whether an edge can still participate in a clique of the target size by looking for a \((d+2)\)-clique containing the edge whose common neighborhood has a sufficiently large upper-bound value.

\begin{definition}
    \label{def:C}
    Let \(G = (V,E)\) be a graph, let \(\overline{uv} \in E\), let \(k,d \in \mathbb{N}_0\) with \(0 \leq d \leq k\), and let \(\omega^u\) be an MCP upper-bound function. For a \((d+2)\)-clique \(K=\{u,v,w_1,\ldots,w_d\}\), define
    \[
        W_G(K) := \bigcap_{x\in K} N_G(x).
    \]
    We define \(C_{G,k,d,\omega^u}(u,v)\) as the following condition:
    \begin{align*}
        C_{G,k,d,\omega^u}(u,v) :=\;
         & \exists (d+2)\text{-clique } K=\{u,v,w_1,\ldots,w_d\}:\;\; \omega^u(G[W_G(K)]) \geq k-d.
    \end{align*}
    If \(C_{G,k,d,\omega^u}(u,v)\) holds, then any such clique \(K\) is called a witnessing clique for the edge \(\overline{uv}\), and \(W_G(K)\) is called the witnessing common neighborhood of \(K\).
\end{definition}

Thus, \(C_{G,k,d,\omega^u}(u,v)\) holds if the edge \(\overline{uv}\) can be extended by \(d\) vertices, and the remaining common neighborhood still has upper-bound value at least \(k-d\).

\begin{definition}
    Let \(G = (V,E)\) be a graph, let \(k,d \in \mathbb{N}_0\) with \(0 \leq d \leq k\), and let \(\omega^u\) be an MCP upper-bound function. A \((k,d,\omega^u)\)-truss of \(G\) is a largest spanning subgraph \(G' = (V,E')\), where \(E' \subseteq E\), such that \(C_{G',k,d,\omega^u}(u,v)\) holds for every edge \(\overline{uv} \in E'\).
\end{definition}
By setting \(d=0\), the definition reduces to the \((k,\omega^u)\)-truss. The parameter \(d\) gives a trade-off between running time and the strength of the reduction: larger values of \(d\) may give stronger reductions, but require more computational time.

As in the core case, existence is immediate, since \(G\) is finite and the empty spanning subgraph satisfies the condition vacuously. If \(\omega^u\) is increasing, uniqueness follows from the following proposition:

\begin{proposition}
    If \(\omega^u\) is increasing, then a \((k,d,\omega^u)\)-truss of a graph is unique.
\end{proposition}

\begin{proof}
    Let \(T_1 = (V,E_1)\) and \(T_2 = (V,E_2)\) be two \((k,d,\omega^u)\)-trusses of \(G\), and let \(T = (V,E_1 \cup E_2)\). We show that \(T\) also satisfies the \((k,d,\omega^u)\)-truss condition. Let \(\overline{uv} \in E_1 \cup E_2\). If \(\overline{uv} \in E_1\), then \(C_{T_1,k,d,\omega^u}(u,v)\) holds. Hence there exists a witnessing clique \(K\) for \(\overline{uv}\) in \(T_1\), so \(\omega^u(T_1[W_{T_1}(K)]) \geq k-d\). Since \(T_1 \subseteq T\), the same clique \(K\) is present in \(T\), and \(W_{T_1}(K) \subseteq W_T(K)\). Because \(\omega^u\) is increasing, \(C_{T,k,d,\omega^u}(u,v)\) also holds. The case \(\overline{uv} \in E_2\) is analogous. Therefore \(T\) satisfies the \((k,d,\omega^u)\)-truss condition. Since \(T_1\) and \(T_2\) have maximum size among spanning subgraphs satisfying this condition, \(|E_1 \cup E_2| = |E_1| = |E_2|\). Hence \(E_1 = E_2\), and therefore \(T_1 = T_2\).
\end{proof}

Thus, when \(\omega^u\) is increasing, we may refer to the unique \((k,d,\omega^u)\)-truss of \(G\) as the \((k,d,\omega^u)\)-truss of \(G\).

The computation follows the same peeling principle as for the \(k\)-core: edges that do not satisfy the condition are removed, and only edges whose condition may have changed are inserted back into the queue. Algorithm~\ref{alg:k,d,omega-truss} computes the \((k,d,\omega^u)\)-truss. The correctness and running-time analysis of the algorithm are given in the following proposition:
\begin{proposition}
    \label{prop:k-d-omega-truss-comp}
    Let \(G' = \textsc{truss}(G,k,d,\omega^u)\), let \(n=|V(G)|\), let \(m=|E(G)|\), let \(w\) be the amortized time for one computation of \(\omega^u\), and let \(w_s\) be the space used for one computation of \(\omega^u\). Then the following holds:
    \begin{itemize}
        \item If \(\omega^u\) is increasing, then \(G'\) is the \((k,d,\omega^u)\)-truss of \(G\).
        \item Assuming space complexity \(\mathcal{O}(n+m+w_s)\)
              \begin{itemize}
                  \item If \(d>0\) or \(\omega^u\) is edge-sensitive, then the running time is
                        \(\mathcal{O}\!\left(n+\left(\sqrt m+w\right)m^{\frac{d+4}{2}}\right)\).
                  \item If \(d=0\) and \(\omega^u\) is not edge-sensitive, then the running time is
                        \(\mathcal{O}\!\left(n+\left(\sqrt m+w\right)m^{\frac{3}{2}}\right)\).
              \end{itemize}
        \item Assuming space complexity \(\mathcal{O}(n+m^\frac{d+3}{2}+w_s)\)
              \begin{itemize}
                  \item If \(d>0\) or \(\omega^u\) is edge-sensitive, then the running time is
                        \(\mathcal{O}\!\left(n+wm^{\frac{d+4}{2}}\right)\).
                  \item If \(d=0\) and \(\omega^u\) is not edge-sensitive, then the running time is
                        \(\mathcal{O}\!\left(n+wm^{\frac{3}{2}}\right)\).
              \end{itemize}
    \end{itemize}
\end{proposition}
\begin{proof}
    The proof is given in Section~\ref{sec:truss-correctness-complexity}.
\end{proof}

\begin{algorithm}[ht]
    \caption{\((k,d,\omega^u)\)-truss computation}
    \label{alg:k,d,omega-truss}

    \DontPrintSemicolon
    \small
    \SetAlgoLined
    \LinesNumbered
    \SetAlgoNlRelativeSize{-1}
    \SetKwInOut{Input}{Input}
    \SetKwInOut{Output}{Output}
    \SetKw{KwConst}{const}
    \SetKw{KwGlobal}{global}

    \SetKwProg{Proc}{ }{}{ }

    \Input{Graph \(G=(V,E)\), integers \(k,d\) with \(0 \leq d \leq k\), MCP upper-bound function \(\omega^u\)}
    \Output{\((k,d,\omega^u)\)-truss of \(G\), if \(\omega^u\) is increasing}

    \BlankLine
    \KwConst \KwGlobal \(k,d,\omega^u\)\;

    \BlankLine
    \BlankLine

    \Proc{\(\textsc{cond}(v_1,v_2,G)\)}{
    \Proc{\textsc{c-rec\((R,W,B)\)}}{
        \lIf{\(|R| = d\)}{\Return \(\omega^u(G[W]) \geq k-d\)}
        \Return \(\bigvee_{v \in B} \textsc{c-rec}(R \cup \{v\},\; W \cap N_G(v), \;B \cap N_G^+(v))\)\;
        }
        \BlankLine

        \(S \gets N_G(v_1) \cap N_G(v_2)\)\;

        \Return \textsc{c-rec\((\emptyset,\; S, S)\)}\;
    }

    \BlankLine

    \Proc{\textsc{truss}\((G, k, d, \omega^u)\)}{
        Set \KwConst \KwGlobal variables \(k,d,\omega^u\)\;
        \(Q_E \gets E\)\;
        \While{\(Q_E \neq \emptyset\)}{
            \(\overline{v_1v_2} \leftarrow \textsc{pick\_and\_remove}(Q_E)\)\;
            \If{\(\neg \textsc{cond}(v_1,v_2,G)\)}{
                \(S \gets N_G(v_1)\cap N_G(v_2)\)\;
                \(Q_E \gets Q_E \cup \{\overline{wv_1},\overline{wv_2} \;|\; w \in S\}\)\;
                \If{\(d > 0\) or \(\omega^u\) is edge-sensitive}{
                    \(Q_E \gets Q_E \cup \{\overline{w_1w_2} \;|\; w_1 \in S,\; w_2 \in S \cap N_G(w_1)\}\)\;
                }
                \(G \gets G \setminus \overline{v_1v_2}\)\;
            }
        }
        \Return \(G\)\;
    }
\end{algorithm}

Let us compare our \((k,0,\omega^u)\)-truss algorithm with the specialized algorithm for computing the \(k\)-truss from~\cite{Wang2012TrussDecomposition}, in the case where \(\omega^u(G):=|V(G)|\). Recall that this trivial upper-bound function is not edge-sensitive, since deleting an edge does not change the number of vertices. For this comparison, we use the higher-space version of our algorithm. In this setting, \(\omega^u\) can be evaluated in amortized \(\mathcal{O}(1)\) time, since the size of the set \(W\) can be maintained throughout the algorithm. Therefore, for \(d=0\), our algorithm runs in time \(\mathcal{O}(n+m^{3/2})\), matching the asymptotic running time of the specialized \(k\)-truss algorithm from~\cite{Wang2012TrussDecomposition}.

The specialized algorithm has space complexity \(\mathcal{O}(m+n)\). In general, our higher-space version uses \(\mathcal{O}(n+m^{3/2}+w_s)\) space, because it precomputes and maintains the common neighborhood \(N(u)\cap N(v)\) for every edge \(\overline{uv}\). However, in the special case \(\omega^u(G):=|V(G)|\), it is enough to maintain only the value \(|N(u)\cap N(v)|\) for every edge \(\overline{uv}\). Since \(w_s = \mathcal{O}(1)\), in this special case, the space complexity can also be reduced to \(\mathcal{O}(n+m)\).

The following proposition gives the reduction property that makes the \((k,d,\omega^u)\)-truss useful for MCP:

\begin{proposition}
    \label{prop:(k,d,omega)_truss}
    Let \(G=(V,E)\) be a graph, let \(\omega^u\) be an MCP upper-bound function, let \(K\subseteq V\) be a set of size \(k\), let \(0\leq d\leq k-2\), and let \(G'=\textsc{truss}(G,k-2,d,\omega^u)\). Then:
    \[
        K \text{ is a clique in } G \Longleftrightarrow K \text{ is a clique in } G'.
    \]
\end{proposition}

\begin{proof}
    The implication \((\Leftarrow)\) follows directly from \(G' \subseteq G\). For \((\Rightarrow)\), suppose that \(K\) is a clique in \(G\), but \(K\) is not a clique in \(G'\). Since \(G'\) is a spanning subgraph of \(G\), some edge of \(K\) was removed by the algorithm. Let \(\overline{uv}\) be the first edge of \(K\) removed, and let \(H\) be the current graph right before \(\overline{uv}\) is removed. At this moment, all other edges of \(K\) are still present. Choose any \(d\) vertices \(w_1,\ldots,w_d\) from \(K\setminus\{u,v\}\), and let \(K'=\{u,v,w_1,\ldots,w_d\}\). Then \(K'\) is a \((d+2)\)-clique in \(H\). The remaining \(k-d-2\) vertices of \(K\) form a clique in \(H[W_H(K')]\), and therefore \(\omega^u(H[W_H(K')]) \geq k-d-2\). Hence \(C_{H,k-2,d,\omega^u}(u,v)\) holds, contradicting the rule by which \(\overline{uv}\) is removed. Therefore all edges of \(K\) remain in \(G'\), and so \(K\) is a clique in \(G'\).
\end{proof}

By applying \textsc{truss}\((G,k-2,d,\omega^u)\), we can reduce the graph while preserving all cliques of size \(k\). Consequently, all cliques of size at least \(k\) are also preserved.

For the same reason as in Remark~\ref{rem:core-preservation-without-increasing}, Proposition~\ref{prop:(k,d,omega)_truss} does not require the assumption that \(\omega^u\) is increasing.

\section{Correctness and complexity of the \((k,d,\omega^u)\)-truss}
\label{sec:truss-correctness-complexity}
In this section, we prove the correctness and time-complexity bounds for Algorithm~\ref{alg:k,d,omega-truss}. We give these proofs in a separate section because they are long and technical. Readers who are not interested in these details may skip this section, since it does not contain material that is crucial for understanding the later parts of the paper.

Before proving correctness and the time-complexity bounds, we need the following lemmas.
\begin{lemma}
    \label{lem:number_of_k_cliques}
    Let \(C_k\) denote the number of \(k\)-cliques in a graph \(G\) with \(m\) edges. For every integer \(k \geq 2\), we have \(C_k = \mathcal{O}(m^{k/2})\).
\end{lemma}
\begin{proof}
    In~\cite[Claim~3]{eden2018approximating} it is shown that
    \[
        C_k \leq m\binom{\sqrt{m}}{k-2}.
    \]
    Since \(\binom{x}{r} \leq x^r\) for \(x\geq 0\) and \(r\geq 0\), we get
    \[
        C_k \leq m\binom{\sqrt{m}}{k-2}
        \leq m(\sqrt{m})^{k-2}
        =
        m^{k/2}.
    \]
    Therefore \(C_k=\mathcal{O}(m^{k/2})\).
\end{proof}

\begin{lemma}
    \label{lem:common_neighborhood_intersections}
    Let \(G=(V,E)\) be a graph with \(m\) edges. For every edge \(\overline{uv}\in E\), define
    \[
        a_{\overline{uv}}:=\min\{|N_G(u)|,|N_G(v)|\}.
    \]
    Then
    \[
        \sum_{\overline{uv}\in E} a_{\overline{uv}}=\mathcal{O}(m^{3/2})
        \qquad\text{and}\qquad
        \sum_{\overline{uv}\in E} a_{\overline{uv}}^2=\mathcal{O}(m^2).
    \]
\end{lemma}

\begin{proof}
    The first bound follows from the standard degree-splitting argument used in~\cite{chiba1985arboricity}. We include the argument for completeness, since the second bound is a simple variation of the same proof.

    Orient every edge \(\overline{uv}\in E\) from the endpoint with smaller neighborhood to the endpoint with larger neighborhood, breaking ties arbitrarily. Let \(d^+(v)\) be the number of edges oriented out of \(v\). Then
    \[
        \sum_{\overline{uv}\in E} a_{\overline{uv}}
        =
        \sum_{v\in V} |N_G(v)|d^+(v).
    \]
    Let \(t=\sqrt{2m}\). If \(|N_G(v)|\leq t\), then \(|N_G(v)|d^+(v)\leq t d^+(v)\). Summing over all such vertices gives at most \(tm=\mathcal{O}(m^{3/2})\). If \(|N_G(v)|>t\), then every outgoing neighbor of \(v\) has degree at least \(|N_G(v)|\), so \(d^+(v)\leq 2m/|N_G(v)|\). Hence \(|N_G(v)|d^+(v)\leq 2m\). Since there are at most \(2m/t=\mathcal{O}(\sqrt m)\) vertices with degree larger than \(t\), their total contribution is \(\mathcal{O}(m^{3/2})\). This proves the first bound.

    For the second bound, using the same orientation gives
    \[
        \sum_{\overline{uv}\in E} a_{\overline{uv}}^2
        =
        \sum_{v\in V} |N_G(v)|^2 d^+(v).
    \]
    If \(|N_G(v)|\leq t\), then \(|N_G(v)|^2d^+(v)\leq t^2d^+(v)=2m\,d^+(v)\). Summing over these vertices gives \(\mathcal{O}(m^2)\). If \(|N_G(v)|>t\), then again \(d^+(v)\leq 2m/|N_G(v)|\), and therefore
    \[
        |N_G(v)|^2 d^+(v) \leq 2m|N_G(v)|.
    \]
    Summing over all vertices gives at most \(2m\sum_{v\in V}|N_G(v)|=4m^2\). Hence the second bound also follows.
\end{proof}

\begin{lemma}
    \label{lem:weighted-clique-degree-sum}
    For every integer \(k\geq 1\), let \(\mathcal K_k\) denote the set of all \(k\)-cliques in \(G\). Then
    \begin{equation*}
        \sum_{C\in \mathcal K_k}\min_{v\in C}\deg_G(v)
        =
        \mathcal{O}\!\left(m^{(k+1)/2}\right).
    \end{equation*}
\end{lemma}

\begin{proof}
    The case \(k=1\) is immediate, since
    \(\sum_{v\in V}\deg_G(v)=2m\).

    For \(k=2\), the desired bound follows directly from the first bound of Lemma~\ref{lem:common_neighborhood_intersections}.

    Assume \(k\geq 3\). Order the vertices by nondecreasing degree, breaking ties arbitrarily, and let \(N^+(v)\) be the set of neighbors of \(v\) that appear after \(v\) in this order.

    Every \(k\)-clique has a unique first vertex \(v\) in this order. This vertex has minimum degree in the clique, and the remaining \(k-1\) vertices form a \((k-1)\)-clique in \(G[N^+(v)]\). Therefore
    \begin{equation*}
        \sum_{K\in\mathcal K_k}\min_{u\in K}\deg_G(u)
        =
        \sum_{v\in V}\deg_G(v)\,
        \left|\{K\in\mathcal K_{k-1} \mid K \subseteq N^+(v)\}\right|.
    \end{equation*}
    Let \(C_{k-1}:=|\mathcal K_{k-1}|\). Since \(G[N^+(v)]\) is an induced subgraph of \(G\), it contains at most \(C_{k-1}\) cliques of size \(k-1\). Hence, by Lemma~\ref{lem:number_of_k_cliques} and since \(k-1\geq 2\),
    \begin{equation*}
        \sum_{K\in\mathcal K_k}\min_{u\in K}\deg_G(u)
        \leq
        \sum_{v\in V}\deg_G(v)C_{k-1}
        =
        2mC_{k-1}
        =
        \mathcal{O}\!\left(m\cdot m^{(k-1)/2}\right)
        =
        \mathcal{O}\!\left(m^{(k+1)/2}\right).
    \end{equation*}
\end{proof}
Now we can start proving Proposition~\ref{prop:k-d-omega-truss-comp}.
\subsection{Correctness}

We first prove the correctness statement from Proposition~\ref{prop:k-d-omega-truss-comp}: if \(\omega^u\) is increasing, then the graph \(G'\) returned by Algorithm~\ref{alg:k,d,omega-truss} is the \((k,d,\omega^u)\)-truss of \(G\).

We begin by showing that \(C_{G,k,d,\omega^u}\) is computed by \textsc{cond}. Consider a fixed edge \(\overline{uv}\). We first show that the recursion in \(\textsc{cond}(u,v,G)\) considers each \((d+2)\)-clique containing \(\overline{uv}\) exactly once. Let \(K=\{u,v,w_1,\ldots,w_d\}\) be such a clique, and assume that the vertices \(w_1,\ldots,w_d\) are ordered increasingly according to their labels. Initially, the branching set is \(B=N_G(u)\cap N_G(v)\), and therefore it contains all vertices \(w_1,\ldots,w_d\). After the recursion chooses \(w_1\), the next branching set becomes \(B\cap N_G^+(w_1)\), which still contains \(w_2,\ldots,w_d\), since \(K\) is a clique and \(w_1<w_2<\cdots<w_d\). Repeating this argument, the recursion has a branch that chooses \(w_1,\ldots,w_d\) in this order.

The branch is unique, because after choosing \(w_j\), the recursion keeps only forward neighbors of \(w_j\). Hence no later step can choose a vertex that appears before \(w_j\), and every \((d+2)\)-clique containing \(\overline{uv}\) is considered exactly once.

It remains to check that the tested set is correct. Let \(R\) denote the set of vertices selected by the recursion in \textsc{cond}. During the recursion, the set \(W\) is maintained as
\[
    W=N_G(u)\cap N_G(v)\cap \bigcap_{x\in R}N_G(x).
\]
Hence, when \(|R|=d\), the set \(W\) is exactly the witnessing common neighborhood \(W_G(\{u,v\}\cup R)\). Therefore the procedure tests whether
\[
    \omega^u(G[W_G(\{u,v\}\cup R)])\geq k-d,
\]
as required in Definition~\ref{def:C}. Consequently, \textsc{cond}\((u,v,G)\) returns true exactly when \(C_{G,k,d,\omega^u}(u,v)\) holds. Thus the algorithm removes exactly those edges for which the \((k,d,\omega^u)\)-truss condition is not satisfied in the current graph.

Initially, all edges are inserted into \(Q_E\), so every edge is checked at least once. If an edge is checked and not removed, then it has a witnessing clique. Later edge removals may destroy such a witness, and then the edge has to be checked again. We therefore have to show that every edge whose condition may change is inserted back into \(Q_E\).

Consider one iteration in which the algorithm removes the edge \(\overline{v_1v_2}\), and let \(S=N_G(v_1)\cap N_G(v_2)\) be its common neighborhood before removal. Suppose that deleting \(\overline{v_1v_2}\) can change the condition \(C_{G,k,d,\omega^u}(x,y)\) of an edge \(\overline{xy} \not\in Q_E\) from true to false. Then, before the deletion, \(\overline{xy}\) has a witnessing clique; let \(K_{\overline{xy}}\) be one such clique in the sense of Definition~\ref{def:C}. Write \(W_{\overline{xy}} := W_G(K_{\overline{xy}})\). The condition \(C_{G,k,d,\omega^u}(x,y)\) depends only on the existence of \(K_{\overline{xy}}\) and on the value of \(\omega^u\) on \(G[W_{\overline{xy}}]\). Therefore this particular witness can be invalidated only if \(\overline{v_1v_2}\) is an edge of \(G[K_{\overline{xy}}\cup W_{\overline{xy}}]\). We claim that, if the condition of an edge \(\overline{xy}\neq \overline{v_1v_2}\) can be affected by deleting \(\overline{v_1v_2}\), then \(\overline{xy}\) belongs to one of the two sets inserted into \(Q_E\) by the algorithm. Let
\[
    U_1:=\{\overline{wv_1},\overline{wv_2}\mid w\in S\},
    \qquad
    U_2:=\{\overline{w_1w_2}\mid w_1\in S,\; w_2\in S\cap N_G(w_1)\}.
\]
We show that \(\overline{xy}\in U_1\cup U_2\). We verify the claim by considering the possible positions of \(v_1\) and \(v_2\) in \(K_{\overline{xy}}\cup W_{\overline{xy}}\):
\begin{itemize}
    \item If both \(v_1\) and \(v_2\) belong to \(K_{\overline{xy}}\), then every other vertex of \(K_{\overline{xy}}\) is adjacent to both \(v_1\) and \(v_2\), and therefore belongs to \(S=N_G(v_1)\cap N_G(v_2)\). Since \(\overline{xy}\neq \overline{v_1v_2}\), either \(\overline{xy}\) is incident to one of \(v_1,v_2\), in which case \(\overline{xy}\in U_1\), or both endpoints of \(\overline{xy}\) belong to \(S\), in which case \(\overline{xy}\in U_2\). In the special case \(d=0\), the second update is not needed here, since \(K_{\overline{xy}}=\{x,y\}\), and therefore this case can occur only for the removed edge \(\overline{xy}=\overline{v_1v_2}\).

    \item Suppose that one endpoint of \(\overline{v_1v_2}\) belongs to \(K_{\overline{xy}}\) and the other belongs to \(W_{\overline{xy}}\). Without loss of generality, let \(v_1\in K_{\overline{xy}}\) and \(v_2\in W_{\overline{xy}}\). If \(\overline{xy}\) is of the form \(\overline{v_1y}\) or \(\overline{xv_1}\), then the other endpoint, \(y\) or \(x\), is adjacent to both \(v_1\) and \(v_2\), and therefore \(\overline{xy}\in U_1\). Otherwise, both \(x\) and \(y\) are different from \(v_1\). Since \(K_{\overline{xy}}\) is a clique and \(v_2\in W_{\overline{xy}}\), both \(x\) and \(y\) are adjacent to \(v_1\) and \(v_2\). Hence \(x,y\in S\), so \(\overline{xy}\in U_2\). In the special case \(d=0\), the second update is not needed here, since \(K_{\overline{xy}}=\{x,y\}\), so only the possibility where \(\overline{xy}\) contains \(v_1\) can occur, and this is covered by the first update.

    \item If both \(v_1\) and \(v_2\) belong to \(W_{\overline{xy}}\), then both are adjacent to every vertex of \(K_{\overline{xy}}\), in particular to \(x\) and \(y\). Hence \(x,y\in S\), so \(\overline{xy}\in U_2\). In the special case where \(\omega^u\) is not edge-sensitive, the second update is not needed here, since deleting an edge inside \(G[W_{\overline{xy}}]\) cannot change the value of \(\omega^u(G[W_{\overline{xy}}])\).
\end{itemize}

Thus every potentially affected edge different from \(\overline{v_1v_2}\) is inserted into \(Q_E\) by the first update or by the second update.

It remains to show that, if \(\omega^u\) is increasing, no edge belonging to the \((k,d,\omega^u)\)-truss is removed. Let \(L\) be a largest spanning subgraph of \(G\) such that \(C_{L,k,d,\omega^u}(u,v)\) holds for every edge \(\overline{uv}\in E(L)\), and suppose that the algorithm removes an edge of \(L\). Let \(\overline{uv}\) be the first such edge, and let \(H\) be the current graph right before the moment when \(\overline{uv}\) is removed. Since no edge of \(L\) was removed before \(\overline{uv}\), we have \(L \subseteq H\). Let \(K_{\overline{uv}}\) be a witnessing clique for \(\overline{uv}\) in \(L\), and let \(W_L(K_{\overline{uv}})\) be its witnessing common neighborhood. Since \(L \subseteq H\), the clique \(K_{\overline{uv}}\) is also present in \(H\), and \(W_L(K_{\overline{uv}}) \subseteq W_H(K_{\overline{uv}})\). Since \(\omega^u\) is increasing, the same clique also witnesses \(C_{H,k,d,\omega^u}(u,v)\), contradicting the rule by which \(\overline{uv}\) is removed from \(H\). Therefore all edges of \(L\) remain in the output graph \(G'\). Since \(G'\) satisfies \(C_{G',k,d,\omega^u}(u,v)\) for every edge \(\overline{uv}\in E(G')\), and \(L\) is a largest spanning subgraph with this property, the output cannot contain any edge outside \(L\). Hence \(G'=L\).

\subsection{Complexity}
For the purpose of the running-time analysis, we ignore isolated vertices in the edge-based part of the algorithm. They can be detected and ignored in \(\mathcal{O}(n+m)\) time. Hence the final running-time bounds include an additional \(\mathcal{O}(n+m)\) term for this preprocessing step.

We first bound the number of processed edge occurrences. The initial queue contains \(m\) edges. We bound the number of insertions into \(Q_E\) after the initial one by summing over all edge removals, since the update rules are applied only when an edge is removed. First consider the update
\[
    Q_E \gets Q_E \cup \{\overline{wv_1},\overline{wv_2} \mid w\in N_G(v_1)\cap N_G(v_2)\}
\]
performed when the edge \(\overline{v_1v_2}\) is removed. This update inserts at most \(2|N_G(v_1)\cap N_G(v_2)|\) edges. For every \(w\in N_G(v_1)\cap N_G(v_2)\), the vertices \(\{v_1,v_2,w\}\) form a triangle in the graph before the removal of \(\overline{v_1v_2}\), and the update inserts the two remaining edges of this triangle. We charge these insertions to that triangle. Each original triangle is charged at most once, namely when its first edge is removed. Hence the total number of insertions caused by this update is \(\mathcal{O}(C_3)=\mathcal{O}(m^{3/2})\), where the last equality follows from Lemma~\ref{lem:number_of_k_cliques}. If \(d>0\) or \(\omega^u\) is edge-sensitive, the algorithm also performs the update
\[
    Q_E \gets Q_E \cup \{\overline{w_1w_2} \mid w_1\in S,\; w_2\in S\cap N_G(w_1)\},
\]
where \(S=N_G(v_1)\cap N_G(v_2)\). For a fixed removed edge \(\overline{v_1v_2}\), this update inserts all edges inside \(G[S]\). For every inserted edge \(\overline{w_1w_2}\), the set \(\{v_1,v_2,w_1,w_2\}\) forms a \(4\)-clique in the graph before the removal. By the same charging argument as above, each original \(4\)-clique is charged at most once. Hence the total number of insertions caused by this update is \(\mathcal{O}(C_4)=\mathcal{O}(m^2)\), where the last equality follows from Lemma~\ref{lem:number_of_k_cliques}. Thus the algorithm processes \(\mathcal{O}(m^2)\) edge occurrences when this second update is used, and \(\mathcal{O}(m^{3/2})\) edge occurrences when \(d=0\) and \(\omega^u\) is not edge-sensitive.

We have now bounded the total number of edge occurrences processed by the algorithm. It remains to bound the time needed to process one such occurrence, which depends on the amount of space allocated to the algorithm.

\subsubsection*{Assuming space complexity \(\mathcal{O}(n+m + w_s)\)}
We now prove the running-time bounds for the
\(\mathcal{O}(n+m+w_s)\)-space implementation:
\begin{itemize}
    \item if \(d>0\) or \(\omega^u\) is edge-sensitive, then the running time is
          \(\mathcal{O}(n+(w+\sqrt{m})m^{(d+4)/2})\);
    \item if \(d=0\) and \(\omega^u\) is not edge-sensitive, then the running time is
          \(\mathcal{O}(n+(w+\sqrt{m})m^{3/2})\).
\end{itemize}

First assume that \(d>0\) or that \(\omega^u\) is edge-sensitive. In this
case the second queue update is used, and by the preceding argument the
algorithm processes at most \(\mathcal{O}(m^2)\) edge occurrences. We now
bound the amortized cost of processing one such occurrence.

We analyse one call of \(\textsc{cond}\). We will analyse it by considering the number of nodes in the recursive tree. We have to prove two thing, the cost of all \(\omega^u\) computations, and the cost of all intersections \(W \cap N_G(v)\). Since intersections \(B \cap N_G^+(v)\) are always intersecting sets that are not larger than intersections \(W \cap N_G(v)\), we don't have to consider them since they don't increase asymptotic time complexity.

We first bound the cost of the \(\omega^u\)-computations in one call of
\textsc{cond}. Such computations are performed only at leaves of the recursion,
that is, when \(|R|=d\). For a fixed processed edge \(\overline{uv}\), each
leaf corresponds uniquely to a \(d\)-clique \(R\) in
\(G[N_G(u)\cap N_G(v)]\). Equivalently, \(\{u,v\}\cup R\) is a
\((d+2)\)-clique containing \(\overline{uv}\).

If \(d=0\), there is only one leaf, and hence the cost of the
\(\omega^u\)-computation is \(\mathcal{O}(w)\) per call.

If \(d=1\), then one call computes \(\omega^u\) once for each vertex in
\(N_G(u)\cap N_G(v)\). We bound this cost amortized over the whole algorithm.
An edge can be processed at most \(\mathcal{O}(m)\) times, since it can be
inserted into the queue only after an edge removal. Moreover,
\[
    \sum_{\overline{uv}\in E}|N_G(u)\cap N_G(v)|
    =
    3C_3
    =
    \mathcal{O}(m^{3/2})
\]
by Lemma~\ref{lem:number_of_k_cliques}. Therefore the total number of
\(\omega^u\)-computations over all calls is at most
\[
    \mathcal{O}\!\left(
    m\sum_{\overline{uv}\in E}|N_G(u)\cap N_G(v)|
    \right)
    =
    \mathcal{O}(m^{5/2}).
\]
Thus their total cost is \(\mathcal{O}(wm^{5/2})\), or
\(\mathcal{O}(wm^{1/2})\) amortized per processed edge occurrence.

Finally, suppose that \(d\geq 2\). For a fixed call of \textsc{cond}, the
number of leaves is at most the number of \(d\)-cliques in \(G\). By
Lemma~\ref{lem:number_of_k_cliques}, this number is
\(\mathcal{O}(m^{d/2})\). Hence, for every \(d\geq 0\), the amortized cost of the \(\omega^u\)-computations per processed edge occurrence is \(\mathcal{O}(wm^{d/2})\).
We next bound the cost of the intersections. First consider the initial
intersection \(N_G(v_1)\cap N_G(v_2)\) computed before the recursive call. We
compute it by scanning the smaller neighborhood and testing membership in the
other one in constant time. Thus, for an edge \(\overline{uv}\), this costs
\[
    \mathcal{O}\!\left(\min\{|N_G(u)|,|N_G(v)|\}\right).
\]
Since the graph only loses edges during the algorithm, this cost is bounded by
the same quantity in the original graph. Moreover, each edge can be processed
at most \(\mathcal{O}(m)\) times. Hence the total cost of all such initial
intersections is at most
\[
    \mathcal{O}\!\left(
    m\sum_{\overline{uv}\in E}
    \min\{|N_G(u)|,|N_G(v)|\}
    \right)
    =
    \mathcal{O}(m^{5/2}),
\]
by Lemma~\ref{lem:common_neighborhood_intersections}. Since in the present
case the algorithm processes at most \(\mathcal{O}(m^2)\) edge occurrences,
this corresponds to an \(\mathcal{O}(\sqrt m)\) amortized contribution per
processed edge occurrence.

We now bound the cost of the intersections \(W\cap N_G(v)\) computed inside the recursive calls. Consider a clique \(R\) chosen by the recursion, where \(1\leq |R|\leq d\). The intersection that creates the node corresponding to \(R\) has the form \(W'\cap N_G(x)\), where \(x\in R\) is the last added vertex. Since we are intersecting \(W'\) with \(N_G(x)\), where \(W'\subseteq N_G(y)\) for every \(y\in R\setminus\{x\}\), the cost of this intersection is \(\mathcal{O}(\min_{y\in R}\deg_G(y))\). Since each recursive node is represented by a unique clique \(R\), the total time used for these intersections is bounded by Lemma~\ref{lem:weighted-clique-degree-sum}:
\[
    \sum_{r=1}^{d}\sum_{R\in\mathcal K_r}\min_{v\in R}\deg_G(v)
    =
    \sum_{r=1}^{d}\mathcal{O}\!\left(m^{(r+1)/2}\right)
    =
    \mathcal{O}\!\left(m^{(d+1)/2}\right).
\]
The last equality follows because this is a geometric sum dominated by its last term. Thus the total cost of the initial intersection and all recursive intersections in one call of \textsc{cond} is \(\mathcal{O}(m^{(d+1)/2})\).

Combining the amortized cost of the \(\omega^u\)-computations with the amortized cost of the intersections, the amortized cost per processed edge occurrence is
\[
    \mathcal{O}\!\left(wm^{d/2}+m^{(d+1)/2}\right)
    =
    \mathcal{O}\!\left((w+\sqrt m)m^{d/2}\right).
\]
Since, in the present case, the algorithm processes \(\mathcal{O}(m^2)\) edge occurrences, together with the preprocessing step for isolated vertices, the total running time is
\[
    \mathcal{O}\!\left(n+(w+\sqrt m)m^{(d+4)/2}\right).
\]

It remains to consider the case \(d=0\) where \(\omega^u\) is not edge-sensitive. Then the recursion has no internal intersections and only one \(\omega^u\)-computation is performed per call. For an edge \(e=\overline{uv}\), let \(a_e=\min\{|N_G(u)|,|N_G(v)|\}\) and \(t_e=|N_G(u)\cap N_G(v)|\), where both quantities are taken in the original graph. The edge \(e\) is processed once initially, and later it can be inserted into \(Q_E\) only when one of the other two edges of a triangle containing \(e\) is removed. Hence \(e\) is processed at most \(\mathcal{O}(1+t_e)\) times. Each computation of \(N_G(u)\cap N_G(v)\) costs at most \(\mathcal{O}(a_e)\), since the graph only loses edges during the algorithm. Therefore the total cost of all initial intersections is
\[
    \mathcal{O}\!\left(\sum_{e\in E}(1+t_e)a_e\right)
    \leq
    \mathcal{O}\!\left(\sum_{e\in E}a_e+\sum_{e\in E}a_e^2\right)
    =
    \mathcal{O}(m^2),
\]
by Lemma~\ref{lem:common_neighborhood_intersections}, using \(t_e\leq a_e\). Since the algorithm processes \(\mathcal{O}(m^{3/2})\) edge occurrences, the total cost of the \(\omega^u\)-computations is \(\mathcal{O}(wm^{3/2})\). Together with the preprocessing step for isolated vertices, the total running time is
\[
    \mathcal{O}\!\left(n+(w+\sqrt m)m^{3/2}\right).
\]

\subsubsection*{Assuming space complexity \(\mathcal{O}(n+m^\frac{d+3}{2}+ w_s) \)}
Finally, we prove the running-time bounds for the higher-space implementation with space complexity \(\mathcal{O}(n+m^{(d+3)/2}+w_s)\):
\begin{itemize}
    \item if \(d>0\) or \(\omega^u\) is edge-sensitive, then the running time is \(\mathcal{O}(n+wm^{(d+4)/2})\);
    \item if \(d=0\) and \(\omega^u\) is not edge-sensitive, then the running time is \(\mathcal{O}(n+wm^{3/2})\).
\end{itemize}

With increased space usage, we precompute and maintain the common neighborhoods needed at the last recursion level of \textsc{cond}. For every clique \(K\) of size \(d+2\), we store
\(
W_G(K)=\bigcap_{v\in K}N_G(v).
\)
The sets are stored in a dictionary indexed by \(K\), so the corresponding common neighborhood can be obtained by an \(\mathcal{O}(1)\)-time lookup.

The space usage is bounded by the total number of stored vertex occurrences. We store each clique according to the order of the vertex labels, and therefore each clique has a unique representation in the dictionary. Thus each common neighborhood is stored only once. If \(|K|=d+2\), then every vertex \(x\in W_G(K)\) gives a \((d+3)\)-clique \(K\cup\{x\}\). Conversely, each \((d+3)\)-clique contributes to at most \(d+3\) such stored occurrences, so the total number of stored occurrences is at most \((d+3)C_{d+3}\), where \(C_k\) denotes the number of \(k\)-cliques in \(G\). From the proof of Lemma~\ref{lem:number_of_k_cliques}, we have
\(
C_k \leq m\binom{\sqrt m}{k-2}.
\)
Hence the total number of stored vertex occurrences is at most
\[
    (d+3)C_{d+3}
    \leq
    (d+3)m\binom{\sqrt m}{d+1}
    \leq
    \frac{d+3}{(d+1)!}m^{(d+3)/2}
    =
    \mathcal{O}\!\left(m^{(d+3)/2}\right).
\]
Thus the total space needed is \(\mathcal{O}(n+m^{(d+3)/2}+w_s)\).

The same bound applies to preprocessing and maintenance time. The common neighborhoods are computed once at the beginning and then updated during the peeling process. By Lemma~\ref{lem:weighted-clique-degree-sum}, the total time needed to construct all stored common neighborhoods is \(\mathcal{O}(m^{(d+3)/2})\). Since the graph only loses vertices and edges, stored common neighborhoods only lose elements. When an edge \(\overline{uv}\) is removed, the affected stored sets can be found by enumerating the relevant cliques in \(N_G(u)\cap N_G(v)\). The stored cliques that contain both \(u\) and \(v\) are deleted, and this deletion is charged to the corresponding clique, while for stored cliques containing exactly one of \(u\) and \(v\), the other endpoint is removed from the corresponding common neighborhood. Each such update can be charged to a stored vertex occurrence that disappears after the deletion. Hence each stored occurrence is updated at most once, and the total update time is also \(\mathcal{O}(m^{(d+3)/2})\).

It remains to bound the cost of the calls to \textsc{cond}. Since the common neighborhoods at the last recursion level are stored, the algorithm does not recompute the final intersections \(W\cap N_G(v)\). Thus the intersection cost is the same as in the \(\mathcal{O}(n+m+w_s)\)-space implementation, but with one recursion level less. By Lemma~\ref{lem:weighted-clique-degree-sum}, this gives an amortized intersection cost of \(\mathcal{O}(m^{d/2})\) per processed edge occurrence. The number of leaves in one call is still at most \(\mathcal{O}(m^{d/2})\), and therefore the cost of the \(\omega^u\)-computations in one call is \(\mathcal{O}(wm^{d/2})\). Hence the amortized cost per processed edge occurrence is \(\mathcal{O}((1+w)m^{d/2})=\mathcal{O}(wm^{d/2})\).

If \(d>0\) or \(\omega^u\) is edge-sensitive, then the algorithm processes \(\mathcal{O}(m^2)\) edge occurrences. Together with the preprocessing step for isolated vertices and the update cost for the stored common neighborhoods, this gives
\[
    \mathcal{O}\!\left(
    n+m^{(d+3)/2}
    +
    m^2\cdot m^{d/2}
    +
    wm^2\cdot m^{d/2}
    \right)
    =
    \mathcal{O}\!\left(n+wm^{(d+4)/2}\right),
\]
where we use that \(w=\Omega(1)\).
The \(\omega^u\)-computations now remain the dominant term, so storing common neighborhoods for additional recursion levels would not improve the stated asymptotic bound.

In the special case \(d=0\) and \(\omega^u\) is not edge-sensitive, the common neighborhoods of edges are stored. Therefore each call to \textsc{cond} only performs one lookup and one computation of \(\omega^u\). Since the algorithm processes \(\mathcal{O}(m^{3/2})\) edge occurrences, together with the preprocessing step for isolated vertices, the total running time is

\[
    \mathcal{O}\!\left(n+m^{3/2}+wm^{3/2}\right)
    =
    \mathcal{O}\!\left(n+wm^{3/2}\right).
\]

\section{Practical implementations of reduction methods}
\label{sec:practical_reductions}
In this section, we describe the practical implementations of the reduction methods used later in our framework for improving upper-bound values. We consider two reduction methods. The first combines the truss and core reductions developed above, while the second is based on repeated applications of structions. These implementations are designed for the computational experiments and will serve as the concrete reduction procedures used by the upper-bound value improvement algorithm.

\subsection{Combining \((k,\omega^u)\)-core and \((k,d,\omega^u)\)-truss}

While the \((k,\omega^u)\)-core and the \((k,d,\omega^u)\)-truss both give valid reductions when applied separately, they remove different parts of the graph: the core removes vertices, while the truss removes edges. In practice, we want to reduce both vertices and edges. We therefore combine the two procedures by alternating calls to \textsc{core} and \textsc{truss}, as shown in Algorithm~\ref{alg:truss-core}.

Note that the \textsc{cond\_mod} is modified in this implementation in order to prune the recursion earlier, before all branches reach depth \(d\). We did not use this version in Definition~\ref{def:C}, because the resulting condition would depend on the order in which vertices are processed. Therefore, we define the theoretical condition in a vertex-order-independent way, while the modified version is only an implementation choice used to speed up the algorithm in our setting. As shown below, this modification still preserves the clique-preservation property, which is the only property needed in our experiments.

Also note that both \textsc{core} and \textsc{truss} are slightly modified so that they both update \(Q_E\) and \(Q_V\) sets. This allows us to prove that the combined procedure has the same asymptotic running time as the truss procedure alone.

Another small implementation detail is the dictionary \(L\), which stores previously found witnessing cliques. When an edge is checked again, the algorithm first tests whether the stored witness is still valid. If it is, the full recursive search can be skipped. This does not change the asymptotic running time, but since checking a stored witness is cheaper than searching for one, it can provide practical speedups.

In this implementation, we use the \(\mathcal{O}(n+m+w_s)\)-space version of the truss procedure. We could also use the higher-space version if desired, since it gives a better theoretical running-time bound. However, it can be more limiting in practice because it stores many common neighborhoods. Moreover, when the computation of \(\omega^u\) costs more than \(\sqrt m\), the two bounds are essentially governed by the same cost of computing \(\omega^u\), so the practical benefit of the higher-space version is smaller. The higher-space version is included mainly for theoretical completeness and to show that, in the special case \(d=0\) with the trivial upper-bound function, our general algorithm matches the asymptotic running time of specialized state-of-the-art truss algorithms. The \(\mathcal{O}(n+m+w_s)\)-space version is therefore less demanding in memory and is the one used in our implementation. Strictly speaking, the combined implementation uses \(\mathcal{O}(n+(d+1)m+w_s)\) space, because the dictionary \(L\) stores at most one witnessing clique for each edge, and each such witness contains \(\mathcal{O}(d+2)\) vertices. This is still acceptable in practice, because \(d\) is small in the intended applications.

\begin{algorithm}[H]
    \caption{The truss and core computation}
    \label{alg:truss-core}

    \DontPrintSemicolon
    \small
    \SetAlgoLined
    \LinesNumbered
    \SetAlgoNlRelativeSize{-1}
    \SetKwInOut{Input}{Input}
    \SetKwInOut{Output}{Output}

    \SetKw{KwGlobal}{global}
    \SetKw{KwConst}{const}

    \SetKwProg{Proc}{ }{}{ }

    \Input{Graph \(G=(V,E)\), integers \(k,d\) with \(0 \leq d \leq k\), MCP upper-bound function \(\omega^u\)}
    \Output{Subgraph \(G'\) of \(G\)}
    \BlankLine
    \KwGlobal Set of vertices \(Q_V\), set of edges \(Q_E\), dictionary of witnessing cliques \(L\)\;
    \KwConst \KwGlobal \(k,d,\omega^u\)\;

    \BlankLine
    \Proc{\textsc{cond\_mod}\((v_1,v_2,G)\)}{
    \Proc{\textsc{c-rec\((R,W,B)\)}}{
        \lIf{\(\omega^u(G[W]) < k-d\)}{\Return \textbf{false}}
        \If{\(|R| = d\)}
        {
            \lIf{\(L[\overline{v_1v_2}]= \emptyset\)}{\(\forall \overline{uv} \in E(G[R\cup\{v_1,v_2\}]) : L[\overline{uv}] \gets R\cup\{v_1,v_2\}\)}
            \Return \textbf{true}\;
        }
        \Return \(\bigvee_{v \in B} \textsc{c-rec}(R \cup \{v\},\; W\cap N_G(v),\; B\cap N_G^+(v))\)\;
        }
        \BlankLine
        \(S\gets N_G(v_1)\cap N_G(v_2)\)\;
        \Return \textsc{c-rec\((\emptyset,\; S,\; S)\)}\;
    }

    \BlankLine

    \Proc{\textsc{core}\((G,k,\omega^u)\)}{
        \While{\(Q_V \neq \emptyset\)}{
            \(v \leftarrow \textsc{pick\_and\_remove}(Q_V)\)\;
            \If{\(v\in V\) \textbf{and} \(\omega^u(N_G(v)) < k\)}{
                \(Q_V \gets Q_V \cup N_G(v)\)\;
                \(Q_E \gets Q_E \cup \{\overline{w_1w_2} \;|\; w_1 \in N_G(v),\; w_2 \in N_G(v)\cap N_G(w_1)\}\)\;
                \(G \gets G \setminus v\)\;
            }
        }
    }

    \BlankLine

    \Proc{\textsc{truss}\((G,k,d,\omega^u)\)}{
        \While{\(Q_E \neq \emptyset\)}{
            \(\overline{v_1v_2} \leftarrow \textsc{pick\_and\_remove}(Q_E)\)\;
            \lIf{\(L[\overline{v_1v_2}]\) is not a witness for \(\overline{v_1v_2}\)}{\(L[\overline{v_1v_2}] \gets \emptyset\)}
            \If{\(\overline{v_1v_2} \in E\) \textbf{and} \(L[\overline{v_1v_2}] = \emptyset\) \textbf{and} \(\neg \textsc{cond\_mod}(v_1,v_2,G)\)}{
                \(S \gets N_G(v_1)\cap N_G(v_2)\)\;
                \(Q_V \gets Q_V \cup \{v_1,v_2\} \cup S\)\;
                \(Q_E \gets Q_E \cup \{\overline{wv_1},\overline{wv_2} \;|\; w \in S\}\)\;
                \If{\(d > 0\)}{
                    \(Q_E \gets Q_E \cup \{\overline{w_1w_2} \;|\; w_1 \in S,\; w_2 \in S\cap N_G(w_1)\}\)\;
                }
                \(G \gets G \setminus \overline{v_1v_2}\)\;
            }
        }
    }

    \BlankLine

    \Proc{\textsc{truss\_core}\((G,k,d,\omega^u)\)}{
        \(Q_V \gets V, \;Q_E \gets E\)\;

        \While{\(Q_V \neq \emptyset\) \textbf{or} \(Q_E \neq \emptyset\)}{
            \textsc{core}\((G,k,\omega^u)\)\;
            \lIf{$k\geq1$}{\textsc{truss}\((G,k-1,d,\omega^u)\)}
        }

        \Return \(G\)\;
    }
\end{algorithm}

\begin{proposition}
    \label{prop:truss-core-comp}
    Let \(G'=\textsc{truss\_core}(G,k,d,\omega^u)\), let \(n=|V(G)|\), let \(m=|E(G)|\), let \(w\) be the amortized time for one computation of \(\omega^u\), and let \(w_s\) be the space used for one computation of \(\omega^u\). Assume that the algorithm is implemented with \(\mathcal{O}(n+(d+1)m+w_s)\) space. Then the following running-time bounds hold:
    \begin{itemize}
        \item If \(d>0\) or \(\omega^u\) is edge-sensitive, then the running time is
              \(\mathcal{O}\!\left(n+(w+\sqrt{m})m^{\frac{d+4}{2}}\right)\).
        \item If \(d=0\) and \(\omega^u\) is not edge-sensitive, then the running time is
              \(\mathcal{O}\!\left(n+(w+\sqrt{m})m^{\frac{3}{2}}\right)\).
    \end{itemize}
\end{proposition}

\begin{proof}
    We use the \(\mathcal{O}(n+m+w_s)\)-space running-time bounds for Algorithm~\ref{alg:k,d,omega-truss} and the running-time bound for Algorithm~\ref{alg:k,omega-core}. The modified procedure \textsc{cond\_mod} only prunes branches earlier, so its cost is bounded by the cost of the original truss procedure. The dictionary \(L\) also does not change this bound: it only stores witnesses found during calls to \textsc{cond\_mod}. When such a witness is found, recording the witnessing clique can be charged to the recursive call branch that found it. Later, checking whether \(L[\overline{v_1v_2}]\) is still a valid witness is only the verification part of a full call to \textsc{cond\_mod}, since the search for a witness is skipped. Hence the overhead of \(L\) is dominated by the cost of \textsc{cond\_mod}.

    Next, consider the extra queue updates. When \textsc{core} removes a vertex \(v\), it inserts into \(Q_E\) all edges inside \(N_G(v)\), since these are exactly the edges whose common neighborhood contains \(v\), and therefore their condition may be affected by the removal of \(v\). Each such edge, together with \(v\), forms a triangle. Each triangle can be charged only once in this way, namely when its first vertex is removed. Thus the total number of these insertions is \(\mathcal{O}(C_3)=\mathcal{O}(m^{3/2})\).

    When \textsc{truss} removes an edge \(\overline{v_1v_2}\), it inserts \(v_1\), \(v_2\), and the vertices in \(N_G(v_1)\cap N_G(v_2)\) into \(Q_V\), since a vertex \(v\) can be affected by the removal of an edge only if one endpoint of the removed edge is \(v\), or if both endpoints of the removed edge belong to \(N_G(v)\). The endpoint insertions contribute \(\mathcal{O}(m)\) in total, while the common-neighbor insertions can again be charged to triangles, giving \(\mathcal{O}(C_3)=\mathcal{O}(m^{3/2})\) insertions. Therefore the additional work caused by the interaction between the two queues is dominated by the truss running time.

    Hence, if \(d>0\) or \(\omega^u\) is edge-sensitive, the running time remains
    \[
        \mathcal{O}\!\left(n+\left(w+\sqrt{m}\right)m^{\frac{d+4}{2}}\right).
    \]
    If \(d=0\) and \(\omega^u\) is not edge-sensitive, the second truss update is not used, and the same argument gives
    \[
        \mathcal{O}\!\left(n+\left(w+\sqrt{m}\right)m^{\frac{3}{2}}\right).
    \]
\end{proof}

\begin{proposition}
    \label{prop:(k,d,omega)_truss_core}
    Let \(G=(V,E)\) be a graph, let \(\omega^u\) be an MCP upper-bound function, let \(K\subseteq V\) be a set of size \(k\), let \(0\leq d\leq k-2\), and let \(G'=\textsc{truss\_core}(G,k-1,d,\omega^u)\). Then:
    \[
        K \text{ is a clique in } G \Longleftrightarrow K \text{ is a clique in } G'.
    \]
\end{proposition}

\begin{proof}
    The implication \((\Leftarrow)\) follows directly from \(G'\subseteq G\). For \((\Rightarrow)\), let \(K\) be a clique of size \(k\) in \(G\). We show that no step of \textsc{truss\_core} removes a vertex or edge of \(K\).

    The \textsc{core} steps preserve \(K\) by Proposition~\ref{prop:k,omega-core}. It remains to consider a \textsc{truss} step. Suppose that an edge \(\overline{uv}\in K\) is tested. Since \(0\leq d\leq k-2\), choose \(d\) vertices \(w_1,\ldots,w_d\) from \(K\setminus\{u,v\}\). The set \(K'=\{u,v,w_1,\ldots,w_d\}\) is a \((d+2)\)-clique, and the remaining vertices of \(K\) lie in \(W_G(K')\). Hence this branch can be extended to a clique of size \(k\).

    If the modified condition used in Algorithm~\ref{alg:truss-core} prunes a branch early, then it does so because the current candidate subgraph has an upper-bound value that is too small to complete a clique of size \(k\). Since \(\omega^u\) is an MCP upper-bound function, such a subgraph cannot contain the missing vertices of \(K\). Therefore the branch corresponding to the vertices of \(K\) is never pruned. Consequently, \textsc{cond\_mod}\((u,v,G)\)returns true for every edge \(\overline{uv}\in K\), so no \textsc{truss} step removes an edge of \(K\). Since the \textsc{core} steps also preserve all vertices of \(K\), the set \(K\) remains a clique in \(G'\).
\end{proof}

\subsection{Structions}\label{sec:structions}
So far, the only reduction method introduced in this paper are truss and core related. We now introduce another reduction, called a \emph{struction}. Structions were originally defined for the Stable Set Problem~\cite{Ebenegger1984PseudoBoolean, alexe2003StructionsRevisited}, but they can also be applied to the Maximum Clique Problem~\cite{math10050697}. Using structions serves two purposes: it shows how the framework for improving upper-bound values defined later can use multiple reduction methods, and it illustrates how combining different reductions can lead to stronger upper-bound value improvements.

A struction transforms a graph \(G\) into a graph \(G'\) such that
\[\omega(G')=\omega(G)-1.\]
Unlike the truss and core reductions, a struction does not necessarily reduce the size of the graph. In fact, the transformed graph often has more edges/vertices than the original graph.

The transformation depends on a chosen pivot vertex \(v\in V\). Different choices of the pivot may produce different transformed graphs, which we denote by \(G_{v}\). In our implementation, we choose a pivot that minimizes the number of edges in the transformed graph, that is,
\[
    v_p \in \argmin_{v\in V}|E(G_v)|.
\]
We define \(\textsc{struction}(G) := G_{v_p}\) obtained by applying the struction transformation with such a pivot \(v_p\).
This is all we need to know about structions for the purposes of this paper; for a full description of the construction, we refer to~\cite{math10050697}.

For our practical purposes, we use the repeated version, \(\textsc{reapply\_struction}(G,\textit{edge\_limit})\), defined in Algorithm~\ref{alg:reapply_struction}. If \(\textit{edge\_limit}\) is provided, structions are applied as long as the transformed graph does not exceed this edge limit. Thus, it may happen that no transformation is accepted. If \(\textit{edge\_limit}\) is not provided, we first apply one struction, set \(\textit{edge\_limit}=|E(G)|\), and then continue with the same rule. Thus, at least one transformation is applied. The algorithm returns the final transformed graph \(G'\) and the number \(h\) of accepted structions, so that \(\omega(G')=\omega(G)-h\).

\begin{algorithm}[H]
    \small
    \SetAlgoLined
    \DontPrintSemicolon
    \SetKwProg{Proc}{ }{}{ }

    \caption{Repeated application of \textsc{struction}}
    \label{alg:reapply_struction}

    \KwIn{A graph \(G\), and an edge limit \(\textit{edge\_limit}\) (default \(0\)).}
    \KwOut{A pair \((H, h)\).}

    \Proc{\textsc{reapply\_struction}\((G, \textit{edge\_limit} \gets 0)\)}{
        \lIf{$|E(G)| = 0$}{\Return $(G,0)$}
        \(H \gets G\)\;
        \(h \gets 0\)\;

        \If{\(\textit{edge\_limit} = 0\)}{
            \(\textit{edge\_limit} \gets |E(G)|\)\;
            \(H \gets \textsc{struction}(G)\)\;
            \(h \gets h + 1\)\;
        }

        \(G' \gets H\)\;
        \While{\(0 <|E(G')| \leq \textit{edge\_limit}\)}{
            \(G' \gets \textsc{struction}(G')\)\;
            \If{\(|E(G')| \leq \textit{edge\_limit}\)}{
                \(H \gets G'\)\;
                \(h \gets h + 1\)\;
            }
        }
        \Return{\((H, h)\)}\;
    }

\end{algorithm}

Note that \textsc{reapply\_struction} cannot return the empty graph if the input graph is not empty. Indeed, a struction is accepted only if the transformed graph still has at least one edge. Since each accepted struction decreases the clique number by one, the process can at worst end with a graph whose remaining vertices are isolated. Such a graph has clique number \(1\); the previous graph had clique number \(2\), and therefore contained at least one edge.

\section{Improving MCP upper-bound values using reductions}
\label{sec:improving_bounds}
In this section, we describe how graph reductions can be used to iteratively improve upper-bound values for MCP. We first introduce an abstract notion of a reduction that is suitable for this purpose, and then show how such reductions can be composed and used inside a general upper-bound value improvement algorithm.

\subsection{General upper-bound value improvement algorithm for MCP}
The basic idea behind the algorithm can already be seen from the \(k\)-core reduction. As observed in~\cite{Walteros2020CliqeEasy}, if a graph has an empty \(k\)-core, then it cannot contain a clique of size \(k+1\). Indeed, every vertex in a clique of size \(k+1\) has at least \(k\) neighbors inside the clique, and therefore such a clique would survive in the \(k\)-core. If the \(k\)-core is empty, then the trivial upper-bound value for the reduced graph is \(0\), which is smaller than \(k+1\). Hence we can conclude that \(\omega(G)<k+1\).

The same reasoning does not depend on the trivial upper-bound function. After computing the \(k\)-core, we may apply any MCP upper-bound function \(\omega^u\) to the reduced graph. If this gives \(\omega^u(k\text{-core}(G))<k+1\), then the \(k\)-core contains no clique of size \(k+1\). Since every clique of size \(k+1\) in \(G\) is preserved by the \(k\)-core reduction, it follows that \(G\) itself contains no clique of size \(k+1\).

Finally, the same principle is not specific to the \(k\)-core. We can use any reduction that preserves the clique sizes relevant to the current value of \(k\), and then apply any MCP upper-bound function to the reduced graph. If the resulting upper-bound value is below the required threshold, then the original graph cannot contain a clique of size \(k\). The following definition captures exactly this implication.

\begin{definition}
    \label{def:k-bound-reduction}
    For \(k \in \mathbb{N}_0\), the function \(R_k:\mathcal{G}\rightarrow\mathcal{G}\) is a \(k\)-bound-reduction if there exists \(r_k:\mathcal{G}\rightarrow\mathbb{N}_0\) such that, for every \(G\in\mathcal{G}\),
    \[
        \omega(R_k(G)) < r_k(G) \implies \omega(G) < k.
    \]
\end{definition}

The function \(r_k\) specifies the threshold that has to be checked after applying the reduction. In the simplest case, the reduction preserves cliques of size \(k\), and then we can take \(r_k(G)=k\). However, some reductions change the clique number in a controlled way. For example, structions reduce the clique number by the number of applied transformations, so the threshold for the reduced graph must be shifted accordingly.
The next proposition shows that both types of reductions fit into Definition~\ref{def:k-bound-reduction}.

\begin{proposition}
    \label{prop:tc_rs_k-bound-reductions}
    Let \(k\in\mathbb{N}_0\), let \(d\in\mathbb{N}_0\) with \(d\leq k-2\), let \(\omega^u\) be an MCP upper-bound function, and let \(a\in\mathbb{N}_0\). Define:
    \begin{itemize}
        \item \(R^{d,\omega^u}_k(G) := \textsc{truss\_core}(G,k-1,d,\omega^u)\),
        \item \(R^a_k(G) := H\), where \((H,h) \gets \textsc{reapply\_struction}(G,a)\).
    \end{itemize}
    Then \(R^{d,\omega^u}_k\) and \(R^a_k\) are \(k\)-bound-reductions.
\end{proposition}
\begin{proof}
    For \(R^{d,\omega^u}_k\), define \(r^{d,\omega^u}_k(G):=k\). If \(\omega(R^{d,\omega^u}_k(G))<r^{d,\omega^u}_k(G)\), then \(\omega(R^{d,\omega^u}_k(G))<k\). By Proposition~\ref{prop:(k,d,omega)_truss_core}, this implies that \(G\) contains no clique of size \(k\), and therefore \(\omega(G)<k\).

    For \(R^a_k\), let \((H,h)\gets\textsc{reapply\_struction}(G,a)\), and define \(r^a_k(G):=k-h\). If \(\omega(R^a_k(G))<r^a_k(G)\), then \(\omega(H)<k-h\). Since \(\textsc{reapply\_struction}\) applies \(h\) structions, we have \(\omega(H)=\omega(G)-h\). Hence \(\omega(G)<k\).
\end{proof}

Once we have a set of \(k\)-bound-reductions, we can apply them sequentially and obtain another \(k\)-bound-reduction. This is formalized in the following proposition:

\begin{proposition}
    Let \(k\in\mathbb{N}_0\). Let \(R_k:\mathcal{G}\rightarrow\mathcal{G}\) be a \(k\)-bound-reduction with corresponding function \(r_k:\mathcal{G}\rightarrow\mathbb{N}_0\). For every \(q\in\mathbb{N}_0\), let \(R'_q:\mathcal{G}\rightarrow\mathcal{G}\) be a \(q\)-bound-reduction with corresponding function \(r'_q:\mathcal{G}\rightarrow\mathbb{N}_0\). Define
    \[
        R^c_k(G):=(R'_{r_k(G)}\circ R_k)(G)=R'_{r_k(G)}(R_k(G)).
    \]
    Then \(R^c_k\) is a \(k\)-bound-reduction.
\end{proposition}
\begin{proof}
    Define \(r^c_k:\mathcal{G}\rightarrow\mathbb{N}_0\) by
    \[
        r^c_k(G):=r'_{r_k(G)}(R_k(G)).
    \]
    Then
    \begin{equation*}
        \begin{split}
            \omega(R^c_k(G)) < r^c_k(G)
             & \Rightarrow \omega(R'_{r_k(G)}(R_k(G))) < r'_{r_k(G)}(R_k(G)) \\
             & \Rightarrow \omega(R_k(G)) < r_k(G)                           \\
             & \Rightarrow \omega(G) < k.
        \end{split}
    \end{equation*}
    The first implication follows from the definition of \(R^c_k\) and \(r^c_k\). The second and third implications follow from the definition of a \(k\)-bound-reduction. Hence \(R^c_k\) is a \(k\)-bound-reduction.
\end{proof}
Now we can define Algorithm~\ref{alg:mcp_upper_bound_improvement} for improving an upper-bound value.

\begin{algorithm}[H]
    \small
    \caption{MCP upper-bound value improvement algorithm}
    \label{alg:mcp_upper_bound_improvement}
    \SetAlgoLined
    \DontPrintSemicolon

    \KwIn{Graph \(G\)\;
        \quad\quad\quad\: Family \((\mathcal R_k)_{k\in\mathbb N}\), where each \(\mathcal R_k\) is a set of pairs \((R_k,r_k)\) such that \(R_k\) is a\;
        \quad\quad\quad\:\quad\quad\quad\: \(k\)-bound-reduction with threshold function \(r_k\)\;
        \quad\quad\quad\: MCP upper-bound function \(\omega^u:\mathcal{G}\rightarrow\mathbb{N}_0\)\;}
    \BlankLine
    \(k \gets \omega^u(G)\)\;
    \While{$\neg${\textsc{termination\_criteria}}}{
        Choose a pair \((R_k,r_k)\in\mathcal R_k\)\;
        \If{\(\omega^u(R_k(G)) < r_k(G)\)}{
            \(k \gets k-1\)\;
        }
    }
    \Return \(k\)\;
\end{algorithm}
The algorithm initializes \(k\) with \(\omega^u(G)\), which is an integer by Definition~\ref{def:upper_bound}.

The loop then tries to decrease the current upper-bound value \(k\). In each iteration, for the current value of \(k\), we choose a pair \((R_k,r_k)\in\mathcal R_k\). Thus \(R_k\) is a \(k\)-bound-reduction and \(r_k\) is its corresponding threshold function. The choice of \((R_k,r_k)\) is left open: in practice, one usually starts with cheaper reductions and applies more expensive reductions only if necessary. The stopping rule \textsc{termination\_criteria} is implementation-dependent. It can be based, for example, on a time limit, a maximum number of reductions, a target upper-bound value, or any other criterion suitable for the given application.

After applying the reduction, we compute the upper-bound value \(\omega^u(R_k(G))\). By Definitions~\ref{def:k-bound-reduction} and~\ref{def:upper_bound}, we have
\begin{equation*}
    \begin{split}
        \omega^u(R_k(G)) < r_k(G)
         & \Rightarrow \omega(R_k(G)) < r_k(G) \\
         & \Rightarrow \omega(G) < k.
    \end{split}
\end{equation*}
Therefore, if \(\omega^u(R_k(G)) < r_k(G)\), then the current value \(k\) is not attainable as a clique size in \(G\). Since clique numbers are integers, this implies \(\omega(G)\leq k-1\), and the algorithm can update \(k\gets k-1\). If the condition fails, the current upper-bound value is not improved in that iteration, and the algorithm tries again with another selected reduction.

After termination, the algorithm returns \(k\) as the improved upper-bound value for \(\omega(G)\).

\begin{remark}
    \label{rem:k-not-certified-upper-bound}
    The value \(k\) does not have to be a certified strict upper-bound value for this argument to work. We may also test a value of \(k\) for which it is not known whether \(\omega(G)<k\). If the condition \(\omega^u(R_k(G))<r_k(G)\) holds, then by Definition~\ref{def:k-bound-reduction} we know that \(\omega(G)<k\). In other words, if \(\omega(G)\geq k\), the algorithm cannot falsely certify an improvement; in that case, the certification condition will fail.
\end{remark}

\subsection{Instantiation with truss and core}
Our first experimental variant uses only the \textsc{truss\_core} reduction, which is given in Algorithm~\ref{alg:truss_core_realization}.

\begin{algorithm}[H]
    \small
    \SetAlgoLined
    \DontPrintSemicolon

    \caption{Truss and core based tightening of an upper-bound value for \(\omega(G)\)}
    \label{alg:truss_core_realization}

    \KwIn{A graph \(G\).}
    \KwOut{A tightened upper-bound value for \(\omega(G)\).}

    \(k \gets \omega^{u}(G)\)\;

    \(d \gets 0\)\;
    \(H\gets G\)\;
    \While{$\neg$\textsc{termination\_criteria}}{
        \lIf{\(d+2 > k\)}{\Return \(k\)}
        \(H \gets\) \textsc{truss\_core}\((H, k-1, d, \omega^u)\)\;
        \eIf{\(\omega^{u}(H) < k\)}{
            \(k \gets k-1\)\;
            \(H \gets G\)
        }
        {
            \(d\gets d+1\)
        }
    }

    \Return{\(k\)}\;

\end{algorithm}
In each iteration, the algorithm applies \textsc{truss\_core} and checks whether the upper-bound value can be improved. If the value is improved, then \(k\) is decreased by one, \(H\) is reset to be again \(G\) and the same value of \(d\) is kept. Otherwise, \(k\) remains unchanged and we increase \(d\gets d+1\).

Notice that the algorithm has an additional stopping rule, which we call the \emph{structural stopping condition}: it terminates when \(d+2>k\). The reason is that the last meaningful case is \(d+2=k\). In this case, a witnessing clique from Definition~\ref{def:C} already has size \(k\). Hence, if at least one edge remains after the reduction, then some remaining edge satisfies Definition~\ref{def:C}, and its witnessing clique is a clique of size \(k\). Therefore, the upper-bound test cannot certify \(\omega(G)<k\), and larger values of \(d\) cannot improve the upper-bound value further. If the reduction removes all edges, then only isolated vertices may remain. For \(k\geq 2\), these isolated vertices are also removed by the core step, so the reduced graph becomes empty, \(k\) is decreased, and the algorithm continues normally. For \(k=1\), the structural stopping condition is satisfied trivially, and the algorithm terminates.

In terms of the abstract framework in Algorithm~\ref{alg:mcp_upper_bound_improvement}, this variant uses reductions built from
\[
    \mathrm{TC}_{k,d}(H)
    :=
    \textsc{truss\_core}(H,k-1,d,\omega^u).
\]
For a chosen value of \(d\), the reduction is the composition of all \textsc{truss\_core} reductions from the current starting value \(d'\) up to \(d\), where \(0\leq d'\leq d\):
\[
    \mathrm{TC}_{k,d}
    \circ
    \mathrm{TC}_{k,d-1}
    \circ
    \cdots
    \circ
    \mathrm{TC}_{k,d'}.
\]
Thus, for each current clique-size threshold \(k\), the set \(\mathcal R_k\) can be understood as containing these composed reductions for all \(0\leq d'\leq d\leq k-2\). The value \(d'\) is the current truss parameter when the graph is reset; in particular, when \(k\) decreases, the algorithm resets the graph but continues with the same value of \(d\), rather than starting again from \(d=0\). The corresponding threshold function from Definition~\ref{def:k-bound-reduction} is \(r_k(G)=k\), since each \(\mathrm{TC}\) step preserves all cliques of size \(k\).

\subsection{Instantiation with structions, truss and core}
The second algorithm which we will use in our experiments is defined in Algorithm~\ref{alg:struction_truss_tightening}.

\begin{algorithm}[H]
    \small
    \SetAlgoLined
    \DontPrintSemicolon

    \caption{Struction, truss and core based tightening of an upper-bound value for \(\omega(G)\)}
    \label{alg:struction_truss_tightening}

    \KwIn{A graph \(G\).}
    \KwOut{A tightened upper-bound value for \(\omega(G)\).}

    \(((H_{\mathrm{init}}, h_{\mathrm{init}}), t_{\mathrm{rs\_init}}) \gets \textsc{reapply\_struction}(G;\ \textit{edge\_limit}=\lvert E(G)\rvert)\)\;

    \(H \gets H_{\mathrm{init}},\; h \gets h_{\mathrm{init}},\; k \gets \omega^u(G),\; d \gets 0,\; t_{\mathrm{rs}} \gets t_{\mathrm{rs\_init}},\; t_{\mathrm{tc}} \gets 0\)\;

    \While{$\neg$\textsc{termination\_criteria}}{
    \(\textit{improved} \gets \textbf{false}\)\;

    \While{\(true\)}{
        \If{\(\omega^{u}(H) < k - h\)}{
            \(H \gets H_{\mathrm{init}},\; h \gets h_{\mathrm{init}},\; k \gets k-1,\; d \gets 0,\; t_{\mathrm{rs}} \gets t_{\mathrm{rs\_init}},\; t_{\mathrm{tc}} \gets 0\)\;
            \(\textit{improved} \gets \textbf{true}\)\;
        }
        \lIf{\(t_{\mathrm{tc}} > t_{\mathrm{rs}}\) \textbf{or} \textit{improved}}{\textbf{break}}
        \lIf{\(d+2 > k-h\)}{\Return \(k\)}
        \(((H,\cdot), t_{\mathrm{tc}}) \gets \textsc{truss\_core}(H,k-h-1,d,\omega^u)\)\;
        \(d \gets d+1\)\;
    }

    \If{\(\textbf{not }\textit{improved}\)}{
    \(((H,h_{\mathrm{add}}), t_{\mathrm{rs}}) \gets \textsc{reapply\_struction}(H)\)\;
    \(h \gets h + h_{\mathrm{add}}\)\;
    }
    }

    \Return{\(k\)}\;

\end{algorithm}

First let us make note that for both \textsc{reapply\_struction} and \textsc{truss\_core} functions we added additional return value which tell us the time the algorithm was executing. First thing we do is call \textsc{reapply\_struction} with \(edge\_limit\) set to number of edges in original graph. What we get out will be our initial graph which we will use at every reset. Then we have two loops. The inner loop executes as long as \textsc{truss\_core} reduction is not slower then previous \textsc{reapply\_struction}. Once it gets slower, we exit the inner loop, call \textsc{reapply\_struction} again, measure its time, and again enter the inner loop.

The algorithm uses the same structural stopping condition as Algorithm~\ref{alg:truss_core_realization}, but adjusted for the clique-size decrease caused by the accepted structions: it terminates when \(d+2>k-h\). After \(h\) structions, a clique of size \(k\) in the original graph corresponds to a clique of size \(k-h\) in the current graph. Thus, at the last meaningful value \(d+2=k-h\), if at least one edge remains after the reduction, then some remaining edge has a witnessing clique of size \(k-h\). This clique can be lifted through the \(h\) structions to a clique of size \(k\) in the original graph, so we cannot certify \(\omega(G)<k\). If no edge remains, then only isolated vertices may remain. For \(k-h\geq 2\), these isolated vertices are removed by the core step, so the reduced graph becomes empty and the upper-bound value can still be improved. For \(k-h=1\), the structural stopping condition is satisfied trivially, and the algorithm terminates.

To connect this algorithm with the abstract framework from Algorithm~\ref{alg:mcp_upper_bound_improvement}, we can view it as choosing reductions from the family \((\mathcal R_k)_{k\in\mathbb N}\). Let
\[
    \mathrm{RS}_a(X)
    :=
    H,
    \qquad
    \text{where } (H,h)=\textsc{reapply\_struction}(X,a),
\]
and let
\[
    \mathrm{TC}_{k,d}(X)
    :=
    \textsc{truss\_core}(X,k-1,d,\omega^u).
\]
Here \(k\) denotes the current clique-size threshold. The reduction \(\mathrm{RS}_a\) changes this threshold from \(k\) to \(k-h\), while \(\mathrm{TC}_{k,d}\) preserves all cliques of size \(k\).

For each current clique-size threshold \(k\), the set \(\mathcal R_k\) used by Algorithm~\ref{alg:struction_truss_tightening} can be understood through composed reductions. The algorithm alternates between applications of \(\mathrm{RS}\) and sequences of \(\mathrm{TC}\) reductions with increasing values of \(d\). The first \(\mathrm{RS}\) step uses the initial edge limit \(a_1\), while all later \(\mathrm{RS}\) steps use parameter \(0\). Thus the composed reduction has the form
\[
    \cdots
    \circ \mathrm{RS}_{0}
    \circ
    \left(\mathrm{TC}_{k_2,D_2}\circ\cdots\circ\mathrm{TC}_{k_2,0}\right)
    \circ \mathrm{RS}_{0}
    \circ
    \left(\mathrm{TC}_{k_1,D_1}\circ\cdots\circ\mathrm{TC}_{k_1,0}\right)
    \circ \mathrm{RS}_{a_1}.
\]
Let \(h_j\) be the clique-size decrease returned by the \(j\)-th application of \(\mathrm{RS}\), counting the initial call \(\mathrm{RS}_{a_1}\) as \(j=1\). For each block \(i\), which follows the first \(i\) applications of \(\mathrm{RS}\), the current clique-size threshold is \(k_i = k-\sum_{j=1}^{i} h_j\). Moreover, \(D_i\leq k_i-2\) is the largest truss parameter reached before the inner stopping rule terminates the block. Thus, the threshold function defined for this composed reduction is \(r_k(G):=k_p\).
The \(\mathrm{TC}\) steps do not change this value, since they preserve all cliques of the current target size. The algorithm may terminate at any point in this composition. A concrete example is
\[
    \cdots
    \circ \mathrm{RS}_{0}
    \circ
    \left(
    \mathrm{TC}_{k_2,2}
    \circ \mathrm{TC}_{k_2,1}
    \circ \mathrm{TC}_{k_2,0}
    \right)
    \circ \mathrm{RS}_{0}
    \circ
    \left(
    \mathrm{TC}_{k_1,1}
    \circ \mathrm{TC}_{k_1,0}
    \right)
    \circ \mathrm{RS}_{a_1}.
\]

In this instantiation, structions and the combined truss and core reductions play complementary roles. Structions decrease the clique number and often reduce graph density, which can make later truss and core reductions faster and more effective. However, structions may also increase the number of vertices or edges. In contrast, truss and core reductions directly remove vertices and edges, and can therefore keep the graph growth caused by structions under control. Alternating the two methods therefore helps control the graph size and reduces the risk that either method gets stuck in a long computation.

\section{Computational results}
\label{sec:comp_results}
In this section, we evaluate the proposed framework for improving upper-bound values experimentally. The goal is to measure how much the reductions improve the values obtained from different MCP upper-bound functions, how quickly these improvements are obtained, and how the two algorithm instantiations compare in practice. We first describe the benchmark instances, implementation details, and tested upper-bound functions. We then report the results on the selected benchmark graphs and finally compare our strongest upper-bound values with previously reported values for several difficult DIMACS instances.

\subsection{Computations on benchmark graphs}

We tested several upper-bound functions with both algorithm instantiations. For brevity, we denote the first instantiation, defined in Algorithm~\ref{alg:truss_core_realization}, by \texttt{tc}, and the second instantiation, defined in Algorithm~\ref{alg:struction_truss_tightening}, by \texttt{tc+rs}. The tested upper-bound functions were:
\begin{itemize}
    \item \textbf{\texttt{trivial}:} We use the number of vertices as the upper-bound value.

    \item \textbf{\texttt{degree+density}:} We compute both the degree and density upper-bound values, and return the smaller value. This adds almost no extra cost, since both values can be computed in the same pass over the graph.

    \item \textbf{\texttt{dsatur}:} We color the graph with the DSatur heuristic, and use the number of colors as the upper-bound value.

    \item \textbf{\texttt{sdp}:} We use a hybrid SDP upper-bound function. If the graph density is larger than \(0.5\), we use the Lov\'asz theta upper-bound function. Otherwise, we use vector coloring. The Lov\'asz theta upper-bound value for MCP is computed on the complement graph. Therefore, when the original graph is dense, its complement is sparse and the corresponding SDP has fewer constraints, making it faster to solve. For sparse graphs, vector coloring is usually faster because its corresponding SDP has fewer constraints. All SDP computations were performed using MOSEK 11.1~\cite{mosek}.
\end{itemize}

Our computational setup was as follows. The implementation was written in C++. We used all graphs with fewer than \(45000\) edges from the benchmark collection of~\cite{marino2024shortreviewnovelapproaches}. This cutoff was chosen for computational feasibility: in preliminary tests, slightly larger instances from the same collection exhausted the available memory during the SDP computation. The resulting set contains \(73\) benchmark graphs, which was sufficient for a broad comparison across graph sizes and densities. The instances used here come from the classical DIMACS benchmark set~\cite{johnson1996cliques}, the EVIL benchmark family~\cite{szabo2019benchmark}, and Sloane's coding-theory challenge instances~\cite{sloane2000independentsets}. We divided them into two groups: small graphs with at most \(15000\) edges, whose results are shown in Table~\ref{tab:small}, and medium graphs with more than \(15000\) edges, whose results are shown in Table~\ref{tab:medium}. For each graph, we ran every combination of algorithm and upper-bound function with a time limit of \(25\) minutes.

\input{progress_tables.tex}

The left part of each table reports basic graph information, including the number of vertices, the number of edges, the density, and the clique number. All clique numbers reported here are proven optimal values taken from~\cite{marino2024shortreviewnovelapproaches}. The right part of each table reports the computational results. The column \(\omega^u\) gives the initial upper-bound value computed by the corresponding function. The column named after the upper-bound function reports the progress of the value over time, namely the best upper-bound value after \(1\) second, \(1\) minute, \(5\) minutes, and \(25\) minutes. The rightmost entry in each cell gives the running time. A computation can finish before the \(25\)-minute limit if the current upper-bound value reaches the known lower bound. When a computation finishes before the time limit, the reported value is written in bold.

For the instances considered here, these lower bounds are also proven optimal values, but the algorithm itself does not require this. If the known lower bound is not optimal, then the algorithm will not certify an upper-bound value below the true optimum; it will instead terminate either when \textsc{termination\_criteria} is met or when the structural stopping condition is reached. As explained for both instantiations, the structural stopping condition is reached when the reduction finds a clique matching the current value of \(k\). For \texttt{tc}, this condition is \(d+2>k\), while for \texttt{tc+rs} it is \(d+2>k-h\). When the structural stopping condition is satisfied, the algorithm has found a clique whose size matches the current upper bound; hence, it has certified the optimum value and can return \(k\).

All computations in Tables~\ref{tab:small} and~\ref{tab:medium} were performed on a single machine with a 13th Gen Intel(R) Core(TM) i5-13420H processor. The machine has 8 physical cores, 12 logical threads, and 31~GiB of RAM. The maximum CPU frequency reported by the system is approximately 4.60~GHz. All computations were run single-threaded, including the MOSEK computations used for the SDP-based upper-bound functions. This was done to make the running times comparable across all tested methods.

Another observation from the tables is that \texttt{tc} compares favorably with the Sequential Elimination Algorithm (SEA) from~\cite{GENDRON2008Sequential}. On the benchmark instances that overlap with the SEA experiments, \texttt{tc} with the trivial upper-bound function gives strictly better values in all cases. With the DSatur upper-bound function, \texttt{tc} gives either strictly better values or the same values; in the equal cases, both methods already attain the optimal value.

This comparison is consistent with the relationship between SEA and the core-only version of our framework. More precisely, if \(\omega^u\) is increasing, the core-only version of our framework returns an upper-bound value no larger than that returned by SEA. By the core-only version, we mean Algorithm~\ref{alg:truss_core_realization} without the \textsc{truss} step, terminated when the \textsc{core} reduction no longer improves the bound. We omit the formal proof, since it would require additional notation and the result is not used elsewhere in the paper. The full \texttt{tc} variant additionally applies truss reductions, which can produce stronger bounds, as observed in our experiments.

To obtain a clearer comparison between the methods, we also plot the results from Tables~\ref{tab:small} and~\ref{tab:medium}; see Figure~\ref{plot:progress}. The \(x\)-axis shows time, while the \(y\)-axis shows the percentage gap between the current upper-bound value and the optimal value. More precisely, if \(U_t\) is the upper-bound value at time \(t\), then the plotted gap is \(100\cdot (U_t-\omega(G))/\omega(G)\). The value at time \(0\) is the initial upper-bound value. After that, the values are shown in steps of \(5\)s. If the initial upper-bound value took more than \(5\)s to compute, we use the initial value also for the later time points until the computation finishes. We omit instances for which the initial upper-bound value was not computed within \(25\) minutes.
\input{progress_plots.tex}

The initial values for the \texttt{trivial} and \texttt{degree+density} upper-bound functions are not shown in the plots, since they would make the scale too large and hide the relevant behavior of the other methods. For this reason, we created Figure~\ref{plot:minute_progress}, which analyzes only the first minute for these two upper-bound functions. The figure has the same structure as Figure~\ref{plot:progress}, but the second row of the \(x\)-axis also reports the average value of the parameter \(d\) for the \texttt{tc} runs at each time point. The plot also includes \texttt{tc+rs}; however, we do not report an average \(d\) for this algorithm, because \texttt{tc+rs} resets \(d\) several times during its execution, making this statistic less informative.

Figures~\ref{plot:progress} and~\ref{plot:minute_progress} show that the \texttt{trivial} upper-bound function quickly outperforms the \texttt{degree+\linebreak[0]density} upper-bound function, although the latter gives a better initial upper-bound value. This is due to the lower cost of computing the \texttt{trivial} upper-bound function: it takes \(\mathcal{O}(n)\) time, while the \texttt{degree+density} upper-bound function takes \(\mathcal{O}(m)\) time. In this setting, the stronger initial upper-bound value is not always worth the additional computation time. The time for the \texttt{degree+density} upper-bound function could possibly be reduced by a specialized implementation that maintains the value inside the reduction algorithm. We do not do this here, since in this work we treat each computation of an upper-bound value as a separate task.
\input{first_minute_progress_plots.tex}

As expected, \(d\) increases faster for the \texttt{trivial} upper-bound function than for the \texttt{degree+\linebreak[0]density} upper-bound function, because the \texttt{trivial} upper-bound function is cheaper to compute. We also observe that most of the improvement happens in the first few seconds; after that, the progress slows down substantially.

Returning to Figure~\ref{plot:progress}, we observe that, for all tested upper-bound functions, \texttt{tc+rs} outperforms \texttt{tc} on average. This shows the benefit of combining different reduction methods. On the smaller graphs, the \texttt{trivial} and \texttt{degree+density} upper-bound functions combined with \texttt{tc+rs} come close to \texttt{dsatur} combined with \texttt{tc}. However, this behavior does not continue on the medium graphs. On the other hand, \texttt{dsatur} combined with \texttt{tc+rs} comes close to the performance of the \texttt{sdp} method.

This motivated a second comparison: we compare \texttt{dsatur} with \texttt{tc+rs} against the initial \texttt{sdp} upper-bound value. More precisely, we count how often \texttt{dsatur} with \texttt{tc+rs} reaches the initial \texttt{sdp} upper-bound value faster than the \texttt{sdp} computation itself. The results are shown in Figure~\ref{plot:heatmap}, where instances are grouped by graph density. Below the heatmap, we also report the average speedup on instances where \texttt{dsatur} with \texttt{tc+rs} is faster, and the average slowdown on instances where the \texttt{sdp} computation is faster.

The results are promising. For all density buckets below \(0.7\), \texttt{dsatur} with \texttt{tc+rs} reaches the initial \texttt{sdp} upper-bound value faster in most cases, with the largest speedups appearing on lower-density graphs. In the density bucket \([0.7,0.8)\), it still wins in a majority of cases, but the slowdowns in that range are much larger than the speedups. For graphs with density above \(0.8\), the \texttt{sdp} upper-bound function performs better, and the slowdowns are even larger. This is expected: the \texttt{tc+rs} reduction is less effective on very dense graphs, while the Lov\'asz theta computation becomes easier because it is applied to the complement graph, which becomes sparse when the original graph is dense.
\input{heatmap_comparison.tex}

\subsection{Improving state-of-the-art upper-bound values on large instances}

We now apply our method to improve existing upper-bound values for open DIMACS instances~\cite{johnson1996cliques}, focusing on the graphs \texttt{C500.9}, \texttt{C1000.9}, and \texttt{C2000.9}.

The plain Lovász theta values reported in~\cite{gong2025roundingtheta} are
\[
    \vartheta(\overline{\texttt{C500.9}}) \approx 84.20,\qquad
    \vartheta(\overline{\texttt{C1000.9}}) \approx 123.49,\qquad
    \vartheta(\overline{\texttt{C2000.9}}) \approx 178.93.
\]
These values were obtained by solving the standard Lovász theta SDP relaxation
on the complemented DIMACS instances.

Stronger upper-bound values were reported in~\cite{battista2021inequality}, who considered a strengthened theta formulation with additional inequality constraints. Their method solves a large-scale SDP relaxation and then applies a post-processing procedure to recover a valid dual upper-bound value. For the three instances considered here, they report the certified upper-bound values
\[
    \vartheta_{+}(\overline{\texttt{C500.9}}) \approx 83.58,\qquad
    \vartheta_{+}(\overline{\texttt{C1000.9}}) \approx 122.60,\qquad
    \vartheta_{+}(\overline{\texttt{C2000.9}}) \approx 177.73.
\]
Here, \(\vartheta_{+}\) denotes the nonnegativity-strengthened theta relaxation used in~\cite{battista2021inequality}. These values improve the integer upper-bound values obtained from the plain Lovász theta values.

For \texttt{C500.9}, an even slightly stronger floating-point value was later reported in~\cite{battista2024lovaszschrijver}. They apply a compact Lovász--Schrijver lift-and-project SDP relaxation, denoted \(\mu(G,C)\), where \(C\) is the constraint set used in their compact formulation, and report
\[
    \mu(\overline{\texttt{C500.9}},C) \approx 83.51.
\]
Since the clique number is integer, any certified floating upper-bound value \(U\) implies
\[
    \omega(G) \leq \lfloor U \rfloor.
\]
Applying this observation to the above values gives
\[
    \omega(\texttt{C500.9}) \leq 83,\qquad
    \omega(\texttt{C1000.9}) \leq 122,\qquad
    \omega(\texttt{C2000.9}) \leq 177.
\]
To the best of our knowledge, these are the strongest certified integer upper-bound values previously reported for these three DIMACS instances.

In this work, we obtained new improved upper-bound values by using the \textsc{truss\_core} reduction with the SDP upper-bound function as \(\omega^u\). Since the three graphs considered here are dense, this means that the Lov\'asz theta upper-bound function was used. The obtained values are
\[
    \omega(\texttt{C500.9}) \leq 73,\qquad
    \omega(\texttt{C1000.9}) \leq 115,\qquad
    \omega(\texttt{C2000.9}) \leq 168.
\]

For \texttt{C500.9}, we allowed both \textsc{core} and \textsc{truss} with \(d=0\) to run as long as they continued to improve the value. The computation that improved the upper-bound value from \(83\) to \(73\) took \texttt{4d 9h 13m}. We also tried to improve the value to \(72\), but stopped the computation after three days, since the \(d=0\) truss step no longer removed a substantial number of edges and moving to \(d=1\) was too expensive.

For \texttt{C1000.9}, the Lov\'asz theta computations were even more expensive. Therefore, we only used the \textsc{core} part of the reduction. The computation that improved the upper-bound value from \(122\) to \(115\) took \texttt{4d 14h 45m}. We also tried to improve the value to \(114\), but stopped the computation after three days, since \textsc{core} no longer removed enough vertices to make a further improvement likely, while applying \textsc{truss} was too expensive.

The computations for \texttt{C500.9} and \texttt{C1000.9} were performed on a single compute node with two AMD EPYC 7402 24-Core processors. The node has 48 physical cores, 96 logical threads, and 125~GiB of RAM. The maximum CPU frequency reported by the system is approximately 3.35~GHz. The MOSEK computations were run as single-process shared-memory jobs, using 24 threads for \texttt{C500.9} and 48 threads for \texttt{C1000.9}. The number of threads was chosen empirically. We tested different thread counts for the Lov\'asz theta computations on \texttt{C500.9} and \texttt{C1000.9}, and used the values that gave the best trade-off before additional threads started to give diminishing returns.

For \texttt{C2000.9}, we used a more targeted approach. Computing all intermediate upper-bound values before reaching \(169\) would require too much time and memory. Instead, as explained in Remark~\ref{rem:k-not-certified-upper-bound}, we directly tested the candidate value \(k=169\). Once the reduction certified \(\omega(G)<169\), we obtained the valid upper-bound value \(\omega(G)\leq 168\).

The value \(k=169\) was chosen empirically. In our implementation, \textsc{core} processes vertices in degeneracy order. On \texttt{C500.9} and \texttt{C1000.9}, we observed that when this ordering allowed \textsc{core} to remove several dozen vertices without stalling, the process usually continued until the desired certificate was obtained. We therefore tested several candidate values for \texttt{C2000.9} and found that \(k=169\) was the smallest value for which this behavior occurred consistently, and then applied only \textsc{core} for this value of \(k\). The computation that returned the upper-bound value \(168\) took \texttt{17d 18h 8m}.

Our choice of the degeneracy order~\cite{Matula1983SmallestLast} is motivated by the standard \(k\)-core computation, where this order reduces the number of queue reinsertions to zero: vertices can be processed and removed without being reinserted into the queue. In our generalized core reduction, the removal test is based on an upper-bound function instead of only on the neighborhood size, so this zero-reinsertion property no longer follows. Nevertheless, the degeneracy order still tends to process vertices with small remaining neighborhoods early, and we therefore expect it to be more effective than an arbitrary ordering.

The computation for \texttt{C2000.9} was performed on a single compute node with two AMD EPYC 9455 48-Core processors. The node has 96 physical cores, 192 logical threads, and 751~GiB of RAM. The maximum CPU frequency reported by the system is approximately \(4.41\) GHz. The MOSEK computation was run as a single-process shared-memory job using 192 threads, which corresponds to all available logical threads on the node. This thread count was chosen empirically on \texttt{C2000.9}, where it gave the fastest computations among the tested values.

\section{Conclusions and Future work}\label{sec:conclusions}

\subsection{Conclusions}
In this paper, we studied the interaction between reduction rules and upper-bound functions for MCP. We first showed how MCP upper-bound functions can be used to strengthen classical core and truss reductions, leading to the \((k,\omega^u)\)-core and the \((k,\omega^u)\)-truss. We then further generalized the truss reduction by introducing the \((k,d,\omega^u)\)-truss, where the parameter \(d\) allows one to trade additional computational time for stronger reductions. We proved the relevant clique-preservation properties, correctness results, and running-time bounds. We then introduced a general framework for improving MCP upper-bound values by applying reductions before evaluating an upper-bound function, and provided two concrete instantiations: one using truss and core reductions, and one combining truss and core reductions with structions.

The computational results show that the proposed reductions can substantially improve several standard upper-bound functions. They also show that combining different reductions can be beneficial in practice. On sparse and medium-density benchmark graphs, the instantiation combining truss and core reductions with structions, using the coloring-based upper-bound function, often reached SDP-level bounds faster than computing the corresponding SDP bound directly. Finally, using the truss and core instantiation with the Lov\'asz theta upper-bound function, we improved the previously best known upper-bound values for three difficult DIMACS instances whose optimal clique numbers are not known.

\subsection{Future work}
The results of this paper suggest several natural extensions. We list the main directions for future work below.

\subsubsection*{Generalization of truss and core}
\begin{itemize}
    \item In this work, we used the \((k,\omega^u)\)-core and introduced the \((k,\omega^u)\)-truss and the \((k,d,\omega^u)\)-truss as upper-bound-based generalizations of the \(k\)-truss. We used these notions as reduction techniques for MCP. Since \(k\)-cores and generalized cores are also used in network analysis, including cohesion analysis, community detection, and large-scale graph processing~\cite{Seidman1983network,Batagelj2011GeneralizedCores,Kong2019kcoreTheories,Sariyuce2016kcore}, it would be interesting to study whether the \((k,\omega^u)\)-core can also be useful in such applications for suitable choices of \(\omega^u\). Similarly, since \(k\)-trusses are widely used as cohesive subgraph models, especially for finding dense communities in large graphs~\cite{Wang2012TrussDecomposition,huang2015closesttruss}, it would be interesting to study whether the two generalizations introduced here, namely the use of \(\omega^u\) and the additional parameter \(d\), can also be useful in such applications.

    \item Regarding the time complexity of the \((k,d,\omega^u)\)-truss algorithm, the case \(d=0\) with \(\omega^u\) not being edge-sensitive improves the running time by a factor of \(\mathcal{O}(\sqrt m)\), because the second, more expensive update of \(Q_E\) can be omitted. We believe that the assumption that \(\omega^u\) is not edge-sensitive is necessary for this improvement: if deleting an edge can change the value of \(\omega^u\), then \(\mathcal{O}(m^2)\) edge occurrences may be affected and the expensive update seems unavoidable. The role of the assumption \(d=0\) is less clear. For \(d>0\), after an edge is removed, the current algorithm has to inspect the edges inside its common neighborhood in order to find which witnessing cliques may have been affected. A possible improvement would be to maintain a data structure that stores witnessing cliques and avoids this search. This could improve the practical running time, but we were not able to prove a better asymptotic bound. A natural direction for future work is therefore to test whether such a data structure improves the practical running time, and then to understand its theoretical limits: either by proving an improved asymptotic bound, or by showing that the current bound cannot be improved in this way, or perhaps cannot be improved at all.

\end{itemize}

\subsubsection*{Upper-bound value improvement algorithm:}
\begin{itemize}
    \item We presented a general algorithm for improving MCP upper-bound values and gave two concrete instantiations. However, the framework is not limited to the reduction rules studied here. Many other reductions could be included, either alone or in combination with existing ones. Future work should therefore study which reduction rules complement each other well and how they can be implemented efficiently in parallel. Once an efficient implementation is developed, it should be evaluated in other state-of-the-art applications that rely on upper-bound values, such as branch-and-bound solvers.

          Another direction for future work is to improve the method that produced the new state-of-the-art upper-bound values for the three benchmark instances. In our implementation, parallelization was used only within the \texttt{sdp} computation. MOSEK supports parallel computation only within a single node and cannot distribute the computation across multiple nodes. Although a different solver could provide multi-node support, parallelizing only the computation of upper-bound values would probably still give diminishing returns. A more promising approach is therefore to parallelize the upper-bound value improvement procedure itself across multiple nodes. The running time could also be reduced by minimizing the number of expensive computations of upper-bound values, for example by warm-starting the solver or reusing relevant intermediate results.

    \item The upper-bound value improvement framework is not limited to MCP or to the reduction methods introduced above. Consider two maximization problems with instance sets \(\mathcal{P}\) and \(\mathcal{Q}\). Thus, \(I\in\mathcal{P}\) denotes an instance of the first problem. Let \(u_{\mathcal{Q}}:\mathcal{Q}\rightarrow\mathbb{R}\) be an upper-bound function for the second problem, and let \(v_{\mathcal{P}}^{\mathrm{opt}}:\mathcal{P}\rightarrow\mathbb{R}\) give the optimal value of an instance of the first problem. For each threshold \(k\), suppose that there exist a transformation \(T_k:\mathcal{P}\rightarrow\mathcal{Q}\) and a threshold function \(t_k:\mathcal{P}\rightarrow\mathbb{R}\) such that, for every \(I\in\mathcal{P}\),
          \[
              u_{\mathcal{Q}}(T_k(I))<t_k(I)
              \Rightarrow
              v_{\mathcal{P}}^{\mathrm{opt}}(I)<k.
          \]
          Then an upper-bound value for the transformed instance \(T_k(I)\) can certify that the optimal value of the original instance \(I\) is below \(k\).

          It would be interesting to investigate whether this idea can improve upper-bound values for other optimization problems and whether allowing more general transformations can produce stronger results.

    \item In our experiments, \texttt{tc+rs} with \texttt{dsatur} reached the same upper-bound value as the \texttt{sdp} computation faster on all graphs with density below \(0.7\). For graphs with density between \(0.7\) and \(0.8\), it was still faster in \(6/8\) cases. We limited the experiments to medium-sized graphs with fewer than \(45000\) edges. The reason is that the single-threaded computation of the \texttt{sdp} upper-bound function was already time-consuming on these instances, and it would become practically infeasible on much larger graphs. If the improvement algorithm is parallelized in future work, it would be useful to test whether this behavior also holds for larger graphs when compared with a multi-threaded SDP solver. If so, this method could provide a practical substitute for SDP-based upper-bound functions, which are expensive and sometimes practically infeasible to compute on large instances.

    \item In this work, the function \(\omega^u\) used inside the reductions is always one of the underlying upper-bound functions, such as coloring-based or SDP-based bounds. A possible direction for future work is to define \(\omega^u\) itself as the output of an upper-bound improvement algorithm. This would create a feedback loop in which improved upper-bound values lead to stronger reductions, and these stronger reductions may then lead to further upper-bound improvements.
\end{itemize}

\section*{Author contributions}
Conceptualization, A.K.; methodology, A.K.; software, A.K.; validation, A.K.; formal analysis, A.K.; investigation, A.K.; visualization, A.K.; writing -- original draft, A.K.; writing -- review and editing, A.K. and J.P.; supervision, J.P.; resources, J.P.; funding acquisition, J.P.

\section*{Notes on contributors}
Aljaž Krpan is a master's student at the University of Ljubljana, Faculty of Mathematics and Physics, mentored by Prof. Janez Povh. His main research interests include algorithmic optimization for discrete and combinatorial problems, mathematical modeling of real-life problems, and efficient practical implementations of algorithms for large-scale and parallel computing environments.

Janez Povh is a full professor of applied mathematics at the University of Ljubljana, Slovenia, and  managing director of Rudolfovo -- Science and Technology Centre Novo mesto, Slovenia. His research focuses on developing state-of-the-art algorithms for hard computational problems, mainly in combinatorial optimization and data science. He is primarily interested in computing exact solutions, that is, global optima, but also adapts his methods to obtain high-quality local optima. He analyzes algorithms theoretically and implements them as high-performance code designed to use modern supercomputers and quantum computers.

\section*{Funding}
This research  was co-funded by (1) the Slovenian Research and Innovation Agency ARIS through the research project Quantum Solver for Hard Binary Quadratic Problems (Grant No. J7-50186), and through the annual work programme of Rudolfovo, and by (2)
QEC4QEA project (Grant Agreement No. 101194322), funded by the European High Performance Computing Joint Undertaking (EuroHPC JU) and participating countries, including Slovenia, under the Horizon Europe research and innovation programme.

\section*{Disclosure statement}
The authors declare that they have no competing interests.

\section*{Data availability}
All code used for the computational experiments is available on the arXiv page of this paper, on GitHub {\footnotesize \url{https://github.com/Rudolfovoorg/Improving_Upper_Bounds_of_MCP_using_Reduction_Rules.git}} and on Zenodo {\footnotesize \url{https://doi.org/10.5281/zenodo.21336337}}.

\afterpage{\clearpage}

\bibliographystyle{plainnat}
\bibliography{References}

\end{document}

%% file: progress_tables.tex
\begin{table}[H]
    \centering
    {
        \fontsize{9pt}{9pt}\selectfont
        \setlength{\arrayrulewidth}{0.3pt}

        \renewcommand{\arraystretch}{0.8}
        \setlength{\extrarowheight}{2pt}

        \begin{adjustbox}{max width=\textwidth}
            \setlength{\dashlinedash}{0.6pt}
            \setlength{\dashlinegap}{1.5pt}
            \setlength{\arrayrulewidth}{0.4pt}
            \begin{tabular}{l l l : l | c: l r |c: l r: |c: l r: |c: l r}
                \hline
                Graph & & & Algorithm & \(\omega^u\) & \multicolumn{2}{:c|}{\texttt{trivial}} & \(\omega^u\) & \multicolumn{2}{:c:|}{\texttt{degree+density}} & \(\omega^u\) & \multicolumn{2}{:c:|}{\texttt{dsatur}} & \(\omega^u\) & \multicolumn{2}{:c}{\texttt{sdp}} \\ \hline
                1dc.64 & \(n:64\) & \(m:1473\) & \texttt{tc} & 64 & \textbf{10} & \textless{}1s & 54 & \textbf{10} & \textless{}1s & 12 & \textbf{10} & \textless{}1s & \textbf{10} & \textbf{10} & \textless{}1s \\
                          & \(\rho:0.73\) & \(\omega:10\) & \texttt{tc+rs} &  & \textbf{10} & \textless{}1s &  & \textbf{10} & \textless{}1s &  & \textbf{10} & \textless{}1s &  & \textbf{10} & \textless{}1s \\ \hline
                1dc.128 & \(n:128\) & \(m:6657\) & \texttt{tc} & 128 & 63 \(\Arrow{.15cm}\) 37 \(\Arrow{.15cm}\) 33 \(\Arrow{.15cm}\) 30 & 25m & 115 & 62 \(\Arrow{.15cm}\) 43 \(\Arrow{.15cm}\) 38 \(\Arrow{.15cm}\) 33 & 25m & 21 & 17 \(\Arrow{.15cm}\) \textbf{16} & 2s & \textbf{16} & \textbf{16} & \textless{}1s \\
                          & \(\rho:0.82\) & \(\omega:16\) & \texttt{tc+rs} &  & 42 \(\Arrow{.15cm}\) \textbf{16} & 14s &  & 39 \(\Arrow{.15cm}\) \textbf{16} & 20s &  & \textbf{16} & \textless{}1s &  & \textbf{16} & \textless{}1s \\ \hline
                1et.64 & \(n:64\) & \(m:1752\) & \texttt{tc} & 64 & 29 \(\Arrow{.15cm}\) 19 \(\Arrow{.15cm}\) \textbf{18} & 1m15s & 59 & 29 \(\Arrow{.15cm}\) 21 \(\Arrow{.15cm}\) \textbf{18} & 3m14s & 20 & \textbf{18} & \textless{}1s & \textbf{18} & \textbf{18} & \textless{}1s \\
                          & \(\rho:0.87\) & \(\omega:18\) & \texttt{tc+rs} &  & \textbf{18} & \textless{}1s &  & \textbf{18} & \textless{}1s &  & \textbf{18} & \textless{}1s &  & \textbf{18} & \textless{}1s \\ \hline
                1et.128 & \(n:128\) & \(m:7456\) & \texttt{tc} & 128 & 91 \(\Arrow{.15cm}\) 80 \(\Arrow{.15cm}\) 76 \(\Arrow{.15cm}\) 74 & 25m & 122 & 88 \(\Arrow{.15cm}\) 79 \(\Arrow{.15cm}\) 78 \(\Arrow{.15cm}\) 74 & 25m & 32 & 31 \(\Arrow{.15cm}\) 30 \(\Arrow{.15cm}\) 30 \(\Arrow{.15cm}\) 30 & 25m & 29 & 29 \(\Arrow{.15cm}\) \textbf{28} & 20s \\
                          & \(\rho:0.92\) & \(\omega:28\) & \texttt{tc+rs} &  & \textbf{28} & \textless{}1s &  & \textbf{28} & \textless{}1s &  & \textbf{28} & \textless{}1s &  & \textbf{28} & \textless{}1s \\ \hline
                1tc.8 & \(n:8\) & \(m:22\) & \texttt{tc} & 8 & \textbf{4} & \textless{}1s & 7 & \textbf{4} & \textless{}1s & \textbf{4} & \textbf{4} & \textless{}1s & \textbf{4} & \textbf{4} & \textless{}1s \\
                          & \(\rho:0.79\) & \(\omega:4\) & \texttt{tc+rs} &  & \textbf{4} & \textless{}1s &  & \textbf{4} & \textless{}1s &  & \textbf{4} & \textless{}1s &  & \textbf{4} & \textless{}1s \\ \hline
                1tc.16 & \(n:16\) & \(m:98\) & \texttt{tc} & 16 & \textbf{8} & \textless{}1s & 14 & \textbf{8} & \textless{}1s & \textbf{8} & \textbf{8} & \textless{}1s & \textbf{8} & \textbf{8} & \textless{}1s \\
                          & \(\rho:0.82\) & \(\omega:8\) & \texttt{tc+rs} &  & \textbf{8} & \textless{}1s &  & \textbf{8} & \textless{}1s &  & \textbf{8} & \textless{}1s &  & \textbf{8} & \textless{}1s \\ \hline
                1tc.32 & \(n:32\) & \(m:428\) & \texttt{tc} & 32 & \textbf{12} & \textless{}1s & 29 & \textbf{12} & \textless{}1s & \textbf{12} & \textbf{12} & \textless{}1s & \textbf{12} & \textbf{12} & \textless{}1s \\
                          & \(\rho:0.86\) & \(\omega:12\) & \texttt{tc+rs} &  & \textbf{12} & \textless{}1s &  & \textbf{12} & \textless{}1s &  & \textbf{12} & \textless{}1s &  & \textbf{12} & \textless{}1s \\ \hline
                1tc.64 & \(n:64\) & \(m:1824\) & \texttt{tc} & 64 & 35 \(\Arrow{.15cm}\) 29 \(\Arrow{.15cm}\) 25 \(\Arrow{.15cm}\) 21 & 25m & 60 & 37 \(\Arrow{.15cm}\) 30 \(\Arrow{.15cm}\) 27 \(\Arrow{.15cm}\) 24 & 25m & \textbf{20} & \textbf{20} & \textless{}1s & \textbf{20} & \textbf{20} & \textless{}1s \\
                          & \(\rho:0.90\) & \(\omega:20\) & \texttt{tc+rs} &  & \textbf{20} & \textless{}1s &  & \textbf{20} & \textless{}1s &  & \textbf{20} & \textless{}1s &  & \textbf{20} & \textless{}1s \\ \hline
                1tc.128 & \(n:128\) & \(m:7616\) & \texttt{tc} & 128 & 97 \(\Arrow{.15cm}\) 89 \(\Arrow{.15cm}\) 86 \(\Arrow{.15cm}\) 83 & 25m & 123 & 96 \(\Arrow{.15cm}\) 89 \(\Arrow{.15cm}\) 86 \(\Arrow{.15cm}\) 85 & 25m & \textbf{38} & \textbf{38} & \textless{}1s & \textbf{38} & \textbf{38} & \textless{}1s \\
                          & \(\rho:0.94\) & \(\omega:38\) & \texttt{tc+rs} &  & \textbf{38} & \textless{}1s &  & \textbf{38} & \textless{}1s &  & \textbf{38} & \textless{}1s &  & \textbf{38} & \textless{}1s \\ \hline
                1zc.128 & \(n:128\) & \(m:7008\) & \texttt{tc} & 128 & 83 \(\Arrow{.15cm}\) 55 \(\Arrow{.15cm}\) 51 \(\Arrow{.15cm}\) 50 & 25m & 118 & 80 \(\Arrow{.15cm}\) 60 \(\Arrow{.15cm}\) 53 \(\Arrow{.15cm}\) 50 & 25m & 28 & 23 \(\Arrow{.15cm}\) 22 \(\Arrow{.15cm}\) 22 \(\Arrow{.15cm}\) 21 & 25m & 20 & 20 \(\Arrow{.15cm}\) 19 \(\Arrow{.15cm}\) \textbf{18} & 3m41s \\
                          & \(\rho:0.86\) & \(\omega:18\) & \texttt{tc+rs} &  & 59 \(\Arrow{.15cm}\) 42 \(\Arrow{.15cm}\) 34 \(\Arrow{.15cm}\) 23 & 25m &  & 60 \(\Arrow{.15cm}\) 44 \(\Arrow{.15cm}\) 36 \(\Arrow{.15cm}\) 28 & 25m &  & 23 \(\Arrow{.15cm}\) 19 \(\Arrow{.15cm}\) \textbf{18} & 3m53s &  & 20 \(\Arrow{.15cm}\) 19 \(\Arrow{.15cm}\) \textbf{18} & 1m50s \\ \hline
                2dc.128 & \(n:128\) & \(m:2955\) & \texttt{tc} & 128 & \textbf{5} & \textless{}1s & 77 & \textbf{5} & \textless{}1s & 6 & \textbf{5} & \textless{}1s & \textbf{5} & - \(\Arrow{.15cm}\) \textbf{5} & 3s \\
                          & \(\rho:0.36\) & \(\omega:5\) & \texttt{tc+rs} &  & \textbf{5} & \textless{}1s &  & \textbf{5} & \textless{}1s &  & \textbf{5} & \textless{}1s &  & - \(\Arrow{.15cm}\) \textbf{5} & 3s \\ \hline
                brock200-1 & \(n:200\) & \(m:14834\) & \texttt{tc} & 200 & 127 \(\Arrow{.15cm}\) 52 \(\Arrow{.15cm}\) 45 \(\Arrow{.15cm}\) 40 & 25m & 166 & 107 \(\Arrow{.15cm}\) 58 \(\Arrow{.15cm}\) 52 \(\Arrow{.15cm}\) 45 & 25m & 53 & 38 \(\Arrow{.15cm}\) 28 \(\Arrow{.15cm}\) 26 \(\Arrow{.15cm}\) 25 & 25m & 27 & - \(\Arrow{.15cm}\) 26 \(\Arrow{.15cm}\) 23 \(\Arrow{.15cm}\) \textbf{21} & 20m15s \\
                          & \(\rho:0.75\) & \(\omega:21\) & \texttt{tc+rs} &  & 127 \(\Arrow{.15cm}\) 48 \(\Arrow{.15cm}\) 37 \(\Arrow{.15cm}\) 32 & 25m &  & 107 \(\Arrow{.15cm}\) 51 \(\Arrow{.15cm}\) 40 \(\Arrow{.15cm}\) 35 & 25m &  & 38 \(\Arrow{.15cm}\) 26 \(\Arrow{.15cm}\) 22 \(\Arrow{.15cm}\) \textbf{21} & 6m35s &  & - \(\Arrow{.15cm}\) 26 \(\Arrow{.15cm}\) 23 \(\Arrow{.15cm}\) \textbf{21} & 20m15s \\ \hline
                brock200-2 & \(n:200\) & \(m:9876\) & \texttt{tc} & 200 & 46 \(\Arrow{.15cm}\) \textbf{12} & 4s & 115 & 25 \(\Arrow{.15cm}\) \textbf{12} & 7s & 31 & 14 \(\Arrow{.15cm}\) \textbf{12} & 1s & 14 & - \(\Arrow{.15cm}\) 14 \(\Arrow{.15cm}\) \textbf{12} & 1m31s \\
                          & \(\rho:0.50\) & \(\omega:12\) & \texttt{tc+rs} &  & 48 \(\Arrow{.15cm}\) \textbf{12} & 5s &  & 26 \(\Arrow{.15cm}\) \textbf{12} & 9s &  & 14 \(\Arrow{.15cm}\) \textbf{12} & 1s &  & - \(\Arrow{.15cm}\) 14 \(\Arrow{.15cm}\) \textbf{12} & 1m31s \\ \hline
                brock200-3 & \(n:200\) & \(m:12048\) & \texttt{tc} & 200 & 87 \(\Arrow{.15cm}\) 17 \(\Arrow{.15cm}\) \textbf{15} & 2m18s & 135 & 61 \(\Arrow{.15cm}\) 22 \(\Arrow{.15cm}\) \textbf{15} & 4m54s & 39 & 23 \(\Arrow{.15cm}\) \textbf{15} & 14s & 18 & - \(\Arrow{.15cm}\) 18 \(\Arrow{.15cm}\) \textbf{15} & 2m14s \\
                          & \(\rho:0.61\) & \(\omega:15\) & \texttt{tc+rs} &  & 88 \(\Arrow{.15cm}\) 18 \(\Arrow{.15cm}\) \textbf{15} & 2m47s &  & 61 \(\Arrow{.15cm}\) 23 \(\Arrow{.15cm}\) 18 \(\Arrow{.15cm}\) \textbf{15} & 8m1s &  & 23 \(\Arrow{.15cm}\) \textbf{15} & 9s &  & - \(\Arrow{.15cm}\) 18 \(\Arrow{.15cm}\) \textbf{15} & 2m14s \\ \hline
                brock200-4 & \(n:200\) & \(m:13089\) & \texttt{tc} & 200 & 104 \(\Arrow{.15cm}\) 27 \(\Arrow{.15cm}\) 22 \(\Arrow{.15cm}\) \textbf{17} & 19m18s & 148 & 79 \(\Arrow{.15cm}\) 31 \(\Arrow{.15cm}\) 26 \(\Arrow{.15cm}\) 22 & 25m & 43 & 28 \(\Arrow{.15cm}\) 19 \(\Arrow{.15cm}\) \textbf{17} & 2m38s & 21 & - \(\Arrow{.15cm}\) 20 \(\Arrow{.15cm}\) \textbf{17} & 2m56s \\
                          & \(\rho:0.66\) & \(\omega:17\) & \texttt{tc+rs} &  & 104 \(\Arrow{.15cm}\) 28 \(\Arrow{.15cm}\) 20 \(\Arrow{.15cm}\) \textbf{17} & 9m22s &  & 79 \(\Arrow{.15cm}\) 32 \(\Arrow{.15cm}\) 24 \(\Arrow{.15cm}\) 18 & 25m &  & 29 \(\Arrow{.15cm}\) \textbf{17} & 28s &  & - \(\Arrow{.15cm}\) 20 \(\Arrow{.15cm}\) \textbf{17} & 2m55s \\ \hline
                C125-9 & \(n:125\) & \(m:6963\) & \texttt{tc} & 125 & 81 \(\Arrow{.15cm}\) 67 \(\Arrow{.15cm}\) 64 \(\Arrow{.15cm}\) 62 & 25m & 118 & 79 \(\Arrow{.15cm}\) 69 \(\Arrow{.15cm}\) 67 \(\Arrow{.15cm}\) 63 & 25m & 52 & 42 \(\Arrow{.15cm}\) 38 \(\Arrow{.15cm}\) 38 \(\Arrow{.15cm}\) 37 & 25m & 37 & 37 \(\Arrow{.15cm}\) \textbf{34} & 35s \\
                          & \(\rho:0.90\) & \(\omega:34\) & \texttt{tc+rs} &  & 81 \(\Arrow{.15cm}\) 53 \(\Arrow{.15cm}\) 44 \(\Arrow{.15cm}\) 36 & 25m &  & 79 \(\Arrow{.15cm}\) 53 \(\Arrow{.15cm}\) 45 \(\Arrow{.15cm}\) 39 & 25m &  & 42 \(\Arrow{.15cm}\) \textbf{34} & 48s &  & 37 \(\Arrow{.15cm}\) \textbf{34} & 35s \\ \hline
                c-fat200-1 & \(n:200\) & \(m:1534\) & \texttt{tc} & 200 & \textbf{12} & \textless{}1s & 18 & \textbf{12} & \textless{}1s & 15 & \textbf{12} & \textless{}1s & \textbf{12} & \textbf{12} & \textless{}1s \\
                          & \(\rho:0.08\) & \(\omega:12\) & \texttt{tc+rs} &  & \textbf{12} & \textless{}1s &  & \textbf{12} & \textless{}1s &  & \textbf{12} & \textless{}1s &  & \textbf{12} & \textless{}1s \\ \hline
                c-fat200-2 & \(n:200\) & \(m:3235\) & \texttt{tc} & 200 & \textbf{24} & \textless{}1s & 35 & \textbf{24} & \textless{}1s & \textbf{24} & \textbf{24} & \textless{}1s & \textbf{24} & - \(\Arrow{.15cm}\) \textbf{24} & 2s \\
                          & \(\rho:0.16\) & \(\omega:24\) & \texttt{tc+rs} &  & \textbf{24} & \textless{}1s &  & \textbf{24} & \textless{}1s &  & \textbf{24} & \textless{}1s &  & - \(\Arrow{.15cm}\) \textbf{24} & 2s \\ \hline
                c-fat200-5 & \(n:200\) & \(m:8473\) & \texttt{tc} & 200 & \textbf{58} & \textless{}1s & 87 & \textbf{58} & \textless{}1s & 84 & \textbf{58} & \textless{}1s & 60 & - \(\Arrow{.15cm}\) \textbf{58} & 20s \\
                          & \(\rho:0.43\) & \(\omega:58\) & \texttt{tc+rs} &  & \textbf{58} & \textless{}1s &  & \textbf{58} & \textless{}1s &  & \textbf{58} & \textless{}1s &  & - \(\Arrow{.15cm}\) \textbf{58} & 20s \\ \hline
                c-fat500-1 & \(n:500\) & \(m:4459\) & \texttt{tc} & 500 & \textbf{14} & \textless{}1s & 21 & \textbf{14} & \textless{}1s & \textbf{14} & \textbf{14} & \textless{}1s & \textbf{14} & - \(\Arrow{.15cm}\) \textbf{14} & 7s \\
                          & \(\rho:0.04\) & \(\omega:14\) & \texttt{tc+rs} &  & \textbf{14} & \textless{}1s &  & \textbf{14} & \textless{}1s &  & \textbf{14} & \textless{}1s &  & - \(\Arrow{.15cm}\) \textbf{14} & 7s \\ \hline
                c-fat500-2 & \(n:500\) & \(m:9139\) & \texttt{tc} & 500 & \textbf{26} & \textless{}1s & 39 & \textbf{26} & \textless{}1s & \textbf{26} & \textbf{26} & \textless{}1s & \textbf{26} & - \(\Arrow{.15cm}\) \textbf{26} & 36s \\
                          & \(\rho:0.07\) & \(\omega:26\) & \texttt{tc+rs} &  & \textbf{26} & \textless{}1s &  & \textbf{26} & \textless{}1s &  & \textbf{26} & \textless{}1s &  & - \(\Arrow{.15cm}\) \textbf{26} & 36s \\ \hline
                evil-chv12x10 & \(n:120\) & \(m:6595\) & \texttt{tc} & 120 & 84 \(\Arrow{.15cm}\) 72 \(\Arrow{.15cm}\) 67 \(\Arrow{.15cm}\) 64 & 25m & 113 & 85 \(\Arrow{.15cm}\) 75 \(\Arrow{.15cm}\) 72 \(\Arrow{.15cm}\) 68 & 25m & 35 & 30 \(\Arrow{.15cm}\) 29 \(\Arrow{.15cm}\) 29 \(\Arrow{.15cm}\) 28 & 25m & 24 & 24 \(\Arrow{.15cm}\) 22 \(\Arrow{.15cm}\) 22 \(\Arrow{.15cm}\) 22 & 25m \\
                          & \(\rho:0.92\) & \(\omega:20\) & \texttt{tc+rs} &  & \textbf{20} & \textless{}1s &  & \textbf{20} & \textless{}1s &  & \textbf{20} & \textless{}1s &  & \textbf{20} & \textless{}1s \\ \hline
                evil-myc5x24 & \(n:120\) & \(m:6904\) & \texttt{tc} & 120 & 101 \(\Arrow{.15cm}\) 97 \(\Arrow{.15cm}\) 95 \(\Arrow{.15cm}\) 94 & 25m & 118 & 101 \(\Arrow{.15cm}\) 97 \(\Arrow{.15cm}\) 96 \(\Arrow{.15cm}\) 95 & 25m & 61 & 55 \(\Arrow{.15cm}\) 54 \(\Arrow{.15cm}\) 54 \(\Arrow{.15cm}\) 54 & 25m & 52 & 52 \(\Arrow{.15cm}\) 49 \(\Arrow{.15cm}\) 49 \(\Arrow{.15cm}\) 49 & 25m \\
                          & \(\rho:0.97\) & \(\omega:48\) & \texttt{tc+rs} &  & \textbf{48} & \textless{}1s &  & \textbf{48} & \textless{}1s &  & \textbf{48} & \textless{}1s &  & \textbf{48} & \textless{}1s \\ \hline
                evil-myc11x11 & \(n:121\) & \(m:6752\) & \texttt{tc} & 121 & 87 \(\Arrow{.15cm}\) 75 \(\Arrow{.15cm}\) 72 \(\Arrow{.15cm}\) 70 & 25m & 116 & 88 \(\Arrow{.15cm}\) 79 \(\Arrow{.15cm}\) 76 \(\Arrow{.15cm}\) 73 & 25m & 40 & 33 \(\Arrow{.15cm}\) 32 \(\Arrow{.15cm}\) 32 \(\Arrow{.15cm}\) 31 & 25m & 26 & 26 \(\Arrow{.15cm}\) 25 \(\Arrow{.15cm}\) 24 \(\Arrow{.15cm}\) 24 & 25m \\
                          & \(\rho:0.93\) & \(\omega:22\) & \texttt{tc+rs} &  & \textbf{22} & \textless{}1s &  & \textbf{22} & \textless{}1s &  & \textbf{22} & \textless{}1s &  & \textbf{22} & \textless{}1s \\ \hline
                evil-s3m25x5 & \(n:125\) & \(m:6877\) & \texttt{tc} & 125 & 84 \(\Arrow{.15cm}\) 63 \(\Arrow{.15cm}\) 57 \(\Arrow{.15cm}\) 54 & 25m & 113 & 83 \(\Arrow{.15cm}\) 67 \(\Arrow{.15cm}\) 64 \(\Arrow{.15cm}\) 58 & 25m & 41 & 32 \(\Arrow{.15cm}\) 30 \(\Arrow{.15cm}\) 29 \(\Arrow{.15cm}\) 28 & 25m & 25 & 25 \(\Arrow{.15cm}\) 22 \(\Arrow{.15cm}\) 22 \(\Arrow{.15cm}\) 22 & 25m \\
                          & \(\rho:0.89\) & \(\omega:20\) & \texttt{tc+rs} &  & 86 \(\Arrow{.15cm}\) 46 \(\Arrow{.15cm}\) 29 \(\Arrow{.15cm}\) \textbf{20} & 8m11s &  & 83 \(\Arrow{.15cm}\) 42 \(\Arrow{.15cm}\) 23 \(\Arrow{.15cm}\) \textbf{20} & 5m27s &  & 32 \(\Arrow{.15cm}\) 27 \(\Arrow{.15cm}\) 23 \(\Arrow{.15cm}\) \textbf{20} & 12m13s &  & 25 \(\Arrow{.15cm}\) 22 \(\Arrow{.15cm}\) 22 \(\Arrow{.15cm}\) 22 & 25m \\ \hline
                evil-myc23x6 & \(n:138\) & \(m:8211\) & \texttt{tc} & 138 & 95 \(\Arrow{.15cm}\) 54 \(\Arrow{.15cm}\) 48 \(\Arrow{.15cm}\) 41 & 25m & 127 & 90 \(\Arrow{.15cm}\) 60 \(\Arrow{.15cm}\) 53 \(\Arrow{.15cm}\) 47 & 25m & 28 & 21 \(\Arrow{.15cm}\) 20 \(\Arrow{.15cm}\) 19 \(\Arrow{.15cm}\) 18 & 25m & 15 & 15 \(\Arrow{.15cm}\) 14 \(\Arrow{.15cm}\) 14 \(\Arrow{.15cm}\) 13 & 25m \\
                          & \(\rho:0.87\) & \(\omega:12\) & \texttt{tc+rs} &  & 95 \(\Arrow{.15cm}\) \textbf{12} & 11s &  & 90 \(\Arrow{.15cm}\) \textbf{12} & 9s &  & 21 \(\Arrow{.15cm}\) 19 \(\Arrow{.15cm}\) \textbf{12} & 3m18s &  & 15 \(\Arrow{.15cm}\) 14 \(\Arrow{.15cm}\) 14 \(\Arrow{.15cm}\) 13 & 25m \\ \hline
                evil-myc5x30 & \(n:150\) & \(m:10837\) & \texttt{tc} & 150 & 132 \(\Arrow{.15cm}\) 124 \(\Arrow{.15cm}\) 123 \(\Arrow{.15cm}\) 122 & 25m & 147 & 130 \(\Arrow{.15cm}\) 126 \(\Arrow{.15cm}\) 124 \(\Arrow{.15cm}\) 122 & 25m & 77 & 70 \(\Arrow{.15cm}\) 66 \(\Arrow{.15cm}\) 66 \(\Arrow{.15cm}\) 66 & 25m & 65 & 65 \(\Arrow{.15cm}\) 63 \(\Arrow{.15cm}\) 61 \(\Arrow{.15cm}\) 61 & 25m \\
                          & \(\rho:0.97\) & \(\omega:60\) & \texttt{tc+rs} &  & \textbf{60} & \textless{}1s &  & \textbf{60} & \textless{}1s &  & \textbf{60} & \textless{}1s &  & \textbf{60} & \textless{}1s \\ \hline
                evil-s3m25x6 & \(n:150\) & \(m:10073\) & \texttt{tc} & 150 & 120 \(\Arrow{.15cm}\) 85 \(\Arrow{.15cm}\) 82 \(\Arrow{.15cm}\) 77 & 25m & 138 & 115 \(\Arrow{.15cm}\) 90 \(\Arrow{.15cm}\) 86 \(\Arrow{.15cm}\) 83 & 25m & 50 & 42 \(\Arrow{.15cm}\) 37 \(\Arrow{.15cm}\) 37 \(\Arrow{.15cm}\) 36 & 25m & 30 & 30 \(\Arrow{.15cm}\) 28 \(\Arrow{.15cm}\) 27 \(\Arrow{.15cm}\) 27 & 25m \\
                          & \(\rho:0.90\) & \(\omega:24\) & \texttt{tc+rs} &  & 120 \(\Arrow{.15cm}\) 101 \(\Arrow{.15cm}\) 95 \(\Arrow{.15cm}\) 89 & 25m &  & 115 \(\Arrow{.15cm}\) 93 \(\Arrow{.15cm}\) 89 \(\Arrow{.15cm}\) 82 & 25m &  & 42 \(\Arrow{.15cm}\) 36 \(\Arrow{.15cm}\) 32 \(\Arrow{.15cm}\) 30 & 25m &  & 30 \(\Arrow{.15cm}\) 28 \(\Arrow{.15cm}\) 27 \(\Arrow{.15cm}\) 27 & 25m \\ \hline
                evil-myc11x14 & \(n:154\) & \(m:11080\) & \texttt{tc} & 154 & 131 \(\Arrow{.15cm}\) 106 \(\Arrow{.15cm}\) 103 \(\Arrow{.15cm}\) 100 & 25m & 149 & 127 \(\Arrow{.15cm}\) 109 \(\Arrow{.15cm}\) 107 \(\Arrow{.15cm}\) 103 & 25m & 48 & 45 \(\Arrow{.15cm}\) 41 \(\Arrow{.15cm}\) 41 \(\Arrow{.15cm}\) 41 & 25m & 33 & 33 \(\Arrow{.15cm}\) 32 \(\Arrow{.15cm}\) 32 \(\Arrow{.15cm}\) 31 & 25m \\
                          & \(\rho:0.94\) & \(\omega:28\) & \texttt{tc+rs} &  & \textbf{28} & \textless{}1s &  & \textbf{28} & \textless{}1s &  & \textbf{28} & \textless{}1s &  & \textbf{28} & \textless{}1s \\ \hline
                hamming6-2 & \(n:64\) & \(m:1824\) & \texttt{tc} & 64 & 37 \(\Arrow{.15cm}\) \textbf{32} & 20s & 58 & 39 \(\Arrow{.15cm}\) 33 \(\Arrow{.15cm}\) \textbf{32} & 1m1s & \textbf{32} & \textbf{32} & \textless{}1s & \textbf{32} & \textbf{32} & \textless{}1s \\
                          & \(\rho:0.90\) & \(\omega:32\) & \texttt{tc+rs} &  & 35 \(\Arrow{.15cm}\) \textbf{32} & 2s &  & 35 \(\Arrow{.15cm}\) \textbf{32} & 2s &  & \textbf{32} & \textless{}1s &  & \textbf{32} & \textless{}1s \\ \hline
                hamming6-4 & \(n:64\) & \(m:704\) & \texttt{tc} & 64 & \textbf{4} & \textless{}1s & 23 & \textbf{4} & \textless{}1s & 7 & \textbf{4} & \textless{}1s & 5 & \textbf{4} & \textless{}1s \\
                          & \(\rho:0.35\) & \(\omega:4\) & \texttt{tc+rs} &  & \textbf{4} & \textless{}1s &  & \textbf{4} & \textless{}1s &  & \textbf{4} & \textless{}1s &  & \textbf{4} & \textless{}1s \\ \hline
                johnson8-2-4 & \(n:28\) & \(m:210\) & \texttt{tc} & 28 & \textbf{4} & \textless{}1s & 16 & \textbf{4} & \textless{}1s & 6 & \textbf{4} & \textless{}1s & \textbf{4} & \textbf{4} & \textless{}1s \\
                          & \(\rho:0.56\) & \(\omega:4\) & \texttt{tc+rs} &  & \textbf{4} & \textless{}1s &  & \textbf{4} & \textless{}1s &  & \textbf{4} & \textless{}1s &  & \textbf{4} & \textless{}1s \\ \hline
                johnson8-4-4 & \(n:70\) & \(m:1855\) & \texttt{tc} & 70 & \textbf{14} & \textless{}1s & 54 & 15 \(\Arrow{.15cm}\) \textbf{14} & 1s & 17 & \textbf{14} & \textless{}1s & \textbf{14} & \textbf{14} & \textless{}1s \\
                          & \(\rho:0.77\) & \(\omega:14\) & \texttt{tc+rs} &  & \textbf{14} & \textless{}1s &  & \textbf{14} & \textless{}1s &  & \textbf{14} & \textless{}1s &  & \textbf{14} & \textless{}1s \\ \hline
                johnson16-2-4 & \(n:120\) & \(m:5460\) & \texttt{tc} & 120 & 39 \(\Arrow{.15cm}\) \textbf{8} & 24s & 92 & 35 \(\Arrow{.15cm}\) \textbf{8} & 57s & 14 & 11 \(\Arrow{.15cm}\) 9 \(\Arrow{.15cm}\) \textbf{8} & 1m20s & \textbf{8} & \textbf{8} & \textless{}1s \\
                          & \(\rho:0.76\) & \(\omega:8\) & \texttt{tc+rs} &  & 41 \(\Arrow{.15cm}\) \textbf{8} & 24s &  & 41 \(\Arrow{.15cm}\) \textbf{8} & 31s &  & 11 \(\Arrow{.15cm}\) 9 \(\Arrow{.15cm}\) 9 \(\Arrow{.15cm}\) \textbf{8} & 9m31s &  & \textbf{8} & \textless{}1s \\ \hline
                keller4 & \(n:171\) & \(m:9435\) & \texttt{tc} & 171 & 73 \(\Arrow{.15cm}\) 16 \(\Arrow{.15cm}\) \textbf{11} & 3m32s & 125 & 52 \(\Arrow{.15cm}\) 21 \(\Arrow{.15cm}\) 14 \(\Arrow{.15cm}\) \textbf{11} & 7m39s & 24 & 14 \(\Arrow{.15cm}\) \textbf{11} & 8s & 14 & - \(\Arrow{.15cm}\) 12 \(\Arrow{.15cm}\) \textbf{11} & 1m12s \\
                          & \(\rho:0.65\) & \(\omega:11\) & \texttt{tc+rs} &  & 73 \(\Arrow{.15cm}\) 12 \(\Arrow{.15cm}\) \textbf{11} & 1m6s &  & 53 \(\Arrow{.15cm}\) 13 \(\Arrow{.15cm}\) \textbf{11} & 1m40s &  & 14 \(\Arrow{.15cm}\) \textbf{11} & 8s &  & - \(\Arrow{.15cm}\) 12 \(\Arrow{.15cm}\) \textbf{11} & 1m12s \\ \hline
                MANN-a9 & \(n:45\) & \(m:918\) & \texttt{tc} & 45 & 27 \(\Arrow{.15cm}\) \textbf{16} & 57s & 42 & 28 \(\Arrow{.15cm}\) 21 \(\Arrow{.15cm}\) \textbf{16} & 3m27s & 19 & 17 \(\Arrow{.15cm}\) \textbf{16} & 15s & 17 & \textbf{16} & \textless{}1s \\
                          & \(\rho:0.93\) & \(\omega:16\) & \texttt{tc+rs} &  & \textbf{16} & \textless{}1s &  & \textbf{16} & \textless{}1s &  & \textbf{16} & \textless{}1s &  & \textbf{16} & \textless{}1s \\ \hline
                p-hat300-1 & \(n:300\) & \(m:10933\) & \texttt{tc} & 300 & \textbf{8} & \textless{}1s & 133 & \textbf{8} & \textless{}1s & 22 & \textbf{8} & \textless{}1s & 10 & - \(\Arrow{.15cm}\) - \(\Arrow{.15cm}\) \textbf{8} & 1m21s \\
                          & \(\rho:0.24\) & \(\omega:8\) & \texttt{tc+rs} &  & \textbf{8} & \textless{}1s &  & \textbf{8} & \textless{}1s &  & \textbf{8} & \textless{}1s &  & - \(\Arrow{.15cm}\) - \(\Arrow{.15cm}\) \textbf{8} & 1m21s \\ \hline
                san200-0-7-1 & \(n:200\) & \(m:13930\) & \texttt{tc} & 200 & 116 \(\Arrow{.15cm}\) 68 \(\Arrow{.15cm}\) 64 \(\Arrow{.15cm}\) 60 & 25m & 156 & 100 \(\Arrow{.15cm}\) 70 \(\Arrow{.15cm}\) 67 \(\Arrow{.15cm}\) 63 & 25m & 42 & \textbf{30} & \textless{}1s & \textbf{30} & - \(\Arrow{.15cm}\) \textbf{30} & 10s \\
                          & \(\rho:0.70\) & \(\omega:30\) & \texttt{tc+rs} &  & 116 \(\Arrow{.15cm}\) 62 \(\Arrow{.15cm}\) 58 \(\Arrow{.15cm}\) 58 & 25m &  & 100 \(\Arrow{.15cm}\) 65 \(\Arrow{.15cm}\) 62 \(\Arrow{.15cm}\) 58 & 25m &  & \textbf{30} & \textless{}1s &  & - \(\Arrow{.15cm}\) \textbf{30} & 10s \\ \hline
                san200-0-7-2 & \(n:200\) & \(m:13930\) & \texttt{tc} & 200 & 108 \(\Arrow{.15cm}\) 80 \(\Arrow{.15cm}\) 72 \(\Arrow{.15cm}\) 68 & 25m & 165 & 102 \(\Arrow{.15cm}\) 82 \(\Arrow{.15cm}\) 77 \(\Arrow{.15cm}\) 72 & 25m & 23 & \textbf{18} & \textless{}1s & \textbf{18} & - \(\Arrow{.15cm}\) \textbf{18} & 17s \\
                          & \(\rho:0.70\) & \(\omega:18\) & \texttt{tc+rs} &  & 109 \(\Arrow{.15cm}\) 33 \(\Arrow{.15cm}\) 26 \(\Arrow{.15cm}\) \textbf{18} & 23m6s &  & 100 \(\Arrow{.15cm}\) 52 \(\Arrow{.15cm}\) 36 \(\Arrow{.15cm}\) 21 & 25m &  & \textbf{18} & \textless{}1s &  & - \(\Arrow{.15cm}\) \textbf{18} & 17s \\ \hline
                sanr200-0-7 & \(n:200\) & \(m:13868\) & \texttt{tc} & 200 & 114 \(\Arrow{.15cm}\) 36 \(\Arrow{.15cm}\) 30 \(\Arrow{.15cm}\) 26 & 25m & 162 & 91 \(\Arrow{.15cm}\) 43 \(\Arrow{.15cm}\) 35 \(\Arrow{.15cm}\) 29 & 25m & 47 & 33 \(\Arrow{.15cm}\) 22 \(\Arrow{.15cm}\) 21 \(\Arrow{.15cm}\) 19 & 25m & 23 & - \(\Arrow{.15cm}\) 22 \(\Arrow{.15cm}\) 19 \(\Arrow{.15cm}\) \textbf{18} & 12m12s \\
                          & \(\rho:0.70\) & \(\omega:18\) & \texttt{tc+rs} &  & 114 \(\Arrow{.15cm}\) 35 \(\Arrow{.15cm}\) 26 \(\Arrow{.15cm}\) 21 & 25m &  & 92 \(\Arrow{.15cm}\) 40 \(\Arrow{.15cm}\) 29 \(\Arrow{.15cm}\) 25 & 25m &  & 33 \(\Arrow{.15cm}\) 19 \(\Arrow{.15cm}\) \textbf{18} & 1m33s &  & - \(\Arrow{.15cm}\) 22 \(\Arrow{.15cm}\) 19 \(\Arrow{.15cm}\) \textbf{18} & 12m13s \\ \hline
            \end{tabular}
        \end{adjustbox}
    }
    \caption{Upper-bound progress on small graph instances with at most \(15000\) edges.}
    \label{tab:small}
\end{table}

\begin{table}[H]
    \centering
    {
        \fontsize{9pt}{9pt}\selectfont
        \setlength{\arrayrulewidth}{0.3pt}

        \renewcommand{\arraystretch}{0.8}
        \setlength{\extrarowheight}{2pt}

        \begin{adjustbox}{max width=\textwidth}
            \setlength{\dashlinedash}{0.6pt}
            \setlength{\dashlinegap}{1.5pt}
            \setlength{\arrayrulewidth}{0.4pt}
            \begin{tabular}{l l l : l | c: l r |c: l r: |c: l r: |c: l r}
                \hline
                Graph & & & Algorithm & \(\omega^u\) & \multicolumn{2}{:c|}{\texttt{trivial}} & \(\omega^u\) & \multicolumn{2}{:c:|}{\texttt{degree+density}} & \(\omega^u\) & \multicolumn{2}{:c:|}{\texttt{dsatur}} & \(\omega^u\) & \multicolumn{2}{:c}{\texttt{sdp}} \\ \hline
                1dc.256 & \(n:256\) & \(m:28801\) & \texttt{tc} & 256 & 212 \(\Arrow{.15cm}\) 145 \(\Arrow{.15cm}\) 137 \(\Arrow{.15cm}\) 127 & 25m & 240 & 199 \(\Arrow{.15cm}\) 151 \(\Arrow{.15cm}\) 143 \(\Arrow{.15cm}\) 135 & 25m & 43 & 37 \(\Arrow{.15cm}\) 33 \(\Arrow{.15cm}\) 33 \(\Arrow{.15cm}\) 33 & 25m & \textbf{30} & - \(\Arrow{.15cm}\) \textbf{30} & 5s \\
                          & \(\rho:0.88\) & \(\omega:30\) & \texttt{tc+rs} &  & 198 \(\Arrow{.15cm}\) 138 \(\Arrow{.15cm}\) 138 \(\Arrow{.15cm}\) 135 & 25m &  & 183 \(\Arrow{.15cm}\) 141 \(\Arrow{.15cm}\) 139 \(\Arrow{.15cm}\) 139 & 25m &  & 37 \(\Arrow{.15cm}\) 33 \(\Arrow{.15cm}\) 33 \(\Arrow{.15cm}\) 33 & 25m &  & - \(\Arrow{.15cm}\) \textbf{30} & 5s \\ \hline
                1et.256 & \(n:256\) & \(m:30976\) & \texttt{tc} & 256 & 232 \(\Arrow{.15cm}\) 200 \(\Arrow{.15cm}\) 198 \(\Arrow{.15cm}\) 195 & 25m & 249 & 226 \(\Arrow{.15cm}\) 203 \(\Arrow{.15cm}\) 199 \(\Arrow{.15cm}\) 194 & 25m & 69 & 65 \(\Arrow{.15cm}\) 63 \(\Arrow{.15cm}\) 62 \(\Arrow{.15cm}\) 62 & 25m & 55 & - \(\Arrow{.15cm}\) 54 \(\Arrow{.15cm}\) 54 \(\Arrow{.15cm}\) 53 & 25m \\
                          & \(\rho:0.95\) & \(\omega:50\) & \texttt{tc+rs} &  & \textbf{50} & \textless{}1s &  & \textbf{50} & \textless{}1s &  & \textbf{50} & \textless{}1s &  & - \(\Arrow{.15cm}\) \textbf{50} & 2s \\ \hline
                1tc.256 & \(n:256\) & \(m:31328\) & \texttt{tc} & 256 & 234 \(\Arrow{.15cm}\) 208 \(\Arrow{.15cm}\) 207 \(\Arrow{.15cm}\) 204 & 25m & 250 & 229 \(\Arrow{.15cm}\) 210 \(\Arrow{.15cm}\) 207 \(\Arrow{.15cm}\) 203 & 25m & 69 & 65 \(\Arrow{.15cm}\) 64 \(\Arrow{.15cm}\) \textbf{63} & 1m52s & \textbf{63} & \textbf{63} & \textless{}1s \\
                          & \(\rho:0.96\) & \(\omega:63\) & \texttt{tc+rs} &  & \textbf{63} & \textless{}1s &  & \textbf{63} & \textless{}1s &  & \textbf{63} & \textless{}1s &  & \textbf{63} & 1s \\ \hline
                1zc.256 & \(n:256\) & \(m:29824\) & \texttt{tc} & 256 & 231 \(\Arrow{.15cm}\) 177 \(\Arrow{.15cm}\) 164 \(\Arrow{.15cm}\) 160 & 25m & 244 & 220 \(\Arrow{.15cm}\) 187 \(\Arrow{.15cm}\) 177 \(\Arrow{.15cm}\) 177 & 25m & 56 & 49 \(\Arrow{.15cm}\) 46 \(\Arrow{.15cm}\) 46 \(\Arrow{.15cm}\) 45 & 25m & 38 & - \(\Arrow{.15cm}\) 38 \(\Arrow{.15cm}\) \textbf{36} & 2m39s \\
                          & \(\rho:0.91\) & \(\omega:36\) & \texttt{tc+rs} &  & 215 \(\Arrow{.15cm}\) 156 \(\Arrow{.15cm}\) 156 \(\Arrow{.15cm}\) 156 & 25m &  & 199 \(\Arrow{.15cm}\) 157 \(\Arrow{.15cm}\) 157 \(\Arrow{.15cm}\) 157 & 25m &  & 49 \(\Arrow{.15cm}\) 46 \(\Arrow{.15cm}\) 46 \(\Arrow{.15cm}\) 43 & 25m &  & - \(\Arrow{.15cm}\) 37 \(\Arrow{.15cm}\) \textbf{36} & 1m51s \\ \hline
                2dc.256 & \(n:256\) & \(m:15457\) & \texttt{tc} & 256 & 60 \(\Arrow{.15cm}\) \textbf{7} & 6s & 176 & 38 \(\Arrow{.15cm}\) \textbf{7} & 9s & 10 & \textbf{7} & \textless{}1s & \textbf{7} & - \(\Arrow{.15cm}\) - \(\Arrow{.15cm}\) \textbf{7} & 3m59s \\
                          & \(\rho:0.47\) & \(\omega:7\) & \texttt{tc+rs} &  & 9 \(\Arrow{.15cm}\) \textbf{7} & 1s &  & \textbf{7} & \textless{}1s &  & \textbf{7} & \textless{}1s &  & - \(\Arrow{.15cm}\) - \(\Arrow{.15cm}\) \textbf{7} & 4m1s \\ \hline
                C250-9 & \(n:250\) & \(m:27984\) & \texttt{tc} & 250 & 210 \(\Arrow{.15cm}\) 150 \(\Arrow{.15cm}\) 142 \(\Arrow{.15cm}\) 136 & 25m & 237 & 198 \(\Arrow{.15cm}\) 156 \(\Arrow{.15cm}\) 149 \(\Arrow{.15cm}\) 144 & 25m & 92 & 80 \(\Arrow{.15cm}\) 70 \(\Arrow{.15cm}\) 69 \(\Arrow{.15cm}\) 69 & 25m & 56 & - \(\Arrow{.15cm}\) 55 \(\Arrow{.15cm}\) 52 \(\Arrow{.15cm}\) 51 & 25m \\
                          & \(\rho:0.90\) & \(\omega:44\) & \texttt{tc+rs} &  & 210 \(\Arrow{.15cm}\) 156 \(\Arrow{.15cm}\) 156 \(\Arrow{.15cm}\) 156 & 25m &  & 198 \(\Arrow{.15cm}\) 152 \(\Arrow{.15cm}\) 148 \(\Arrow{.15cm}\) 148 & 25m &  & 80 \(\Arrow{.15cm}\) 70 \(\Arrow{.15cm}\) 66 \(\Arrow{.15cm}\) 62 & 25m &  & - \(\Arrow{.15cm}\) 55 \(\Arrow{.15cm}\) 52 \(\Arrow{.15cm}\) 51 & 25m \\ \hline
                c-fat500-5 & \(n:500\) & \(m:23191\) & \texttt{tc} & 500 & 87 \(\Arrow{.15cm}\) \textbf{64} & 3s & 96 & 74 \(\Arrow{.15cm}\) \textbf{64} & 2s & \textbf{64} & \textbf{64} & \textless{}1s & \textbf{64} & - \(\Arrow{.15cm}\) - \(\Arrow{.15cm}\) - \(\Arrow{.15cm}\) \textbf{64} & 7m26s \\
                          & \(\rho:0.19\) & \(\omega:64\) & \texttt{tc+rs} &  & 89 \(\Arrow{.15cm}\) \textbf{64} & 3s &  & 76 \(\Arrow{.15cm}\) \textbf{64} & 2s &  & \textbf{64} & \textless{}1s &  & - \(\Arrow{.15cm}\) - \(\Arrow{.15cm}\) - \(\Arrow{.15cm}\) \textbf{64} & 7m27s \\ \hline
                evil-chv12x15 & \(n:180\) & \(m:15166\) & \texttt{tc} & 180 & 160 \(\Arrow{.15cm}\) 128 \(\Arrow{.15cm}\) 124 \(\Arrow{.15cm}\) 120 & 25m & 173 & 155 \(\Arrow{.15cm}\) 132 \(\Arrow{.15cm}\) 128 \(\Arrow{.15cm}\) 124 & 25m & 55 & 50 \(\Arrow{.15cm}\) 45 \(\Arrow{.15cm}\) 45 \(\Arrow{.15cm}\) 45 & 25m & 36 & 36 \(\Arrow{.15cm}\) 35 \(\Arrow{.15cm}\) 34 \(\Arrow{.15cm}\) 34 & 25m \\
                          & \(\rho:0.94\) & \(\omega:30\) & \texttt{tc+rs} &  & 148 \(\Arrow{.15cm}\) 103 \(\Arrow{.15cm}\) 94 \(\Arrow{.15cm}\) 83 & 25m &  & 142 \(\Arrow{.15cm}\) 101 \(\Arrow{.15cm}\) 91 \(\Arrow{.15cm}\) 78 & 25m &  & 43 \(\Arrow{.15cm}\) 38 \(\Arrow{.15cm}\) 37 \(\Arrow{.15cm}\) 34 & 25m &  & 36 \(\Arrow{.15cm}\) 33 \(\Arrow{.15cm}\) 33 \(\Arrow{.15cm}\) 33 & 25m \\ \hline
                evil-myc5x36 & \(n:180\) & \(m:15671\) & \texttt{tc} & 180 & 166 \(\Arrow{.15cm}\) 153 \(\Arrow{.15cm}\) 150 \(\Arrow{.15cm}\) 149 & 25m & 177 & 164 \(\Arrow{.15cm}\) 153 \(\Arrow{.15cm}\) 151 \(\Arrow{.15cm}\) 150 & 25m & 90 & 84 \(\Arrow{.15cm}\) 80 \(\Arrow{.15cm}\) 80 \(\Arrow{.15cm}\) 80 & 25m & 77 & 77 \(\Arrow{.15cm}\) 74 \(\Arrow{.15cm}\) 73 \(\Arrow{.15cm}\) \textbf{72} & 8m38s \\
                          & \(\rho:0.97\) & \(\omega:72\) & \texttt{tc+rs} &  & 115 \(\Arrow{.15cm}\) 73 \(\Arrow{.15cm}\) \textbf{72} & 1m6s &  & 108 \(\Arrow{.15cm}\) 74 \(\Arrow{.15cm}\) \textbf{72} & 1m14s &  & 73 \(\Arrow{.15cm}\) \textbf{72} & 1s &  & 77 \(\Arrow{.15cm}\) \textbf{72} & 38s \\ \hline
                evil-myc23x8 & \(n:184\) & \(m:15072\) & \texttt{tc} & 184 & 152 \(\Arrow{.15cm}\) 98 \(\Arrow{.15cm}\) 91 \(\Arrow{.15cm}\) 85 & 25m & 173 & 144 \(\Arrow{.15cm}\) 107 \(\Arrow{.15cm}\) 99 \(\Arrow{.15cm}\) 91 & 25m & 37 & 32 \(\Arrow{.15cm}\) 27 \(\Arrow{.15cm}\) 27 \(\Arrow{.15cm}\) 27 & 25m & 20 & 20 \(\Arrow{.15cm}\) 19 \(\Arrow{.15cm}\) 19 \(\Arrow{.15cm}\) 19 & 25m \\
                          & \(\rho:0.90\) & \(\omega:16\) & \texttt{tc+rs} &  & 152 \(\Arrow{.15cm}\) 80 \(\Arrow{.15cm}\) 19 \(\Arrow{.15cm}\) 19 & 25m &  & 144 \(\Arrow{.15cm}\) 72 \(\Arrow{.15cm}\) \textbf{16} & 2m52s &  & 32 \(\Arrow{.15cm}\) 27 \(\Arrow{.15cm}\) 27 \(\Arrow{.15cm}\) 18 & 25m &  & 20 \(\Arrow{.15cm}\) 19 \(\Arrow{.15cm}\) 19 \(\Arrow{.15cm}\) 19 & 25m \\ \hline
                evil-myc11x17 & \(n:187\) & \(m:16490\) & \texttt{tc} & 187 & 169 \(\Arrow{.15cm}\) 140 \(\Arrow{.15cm}\) 136 \(\Arrow{.15cm}\) 132 & 25m & 181 & 165 \(\Arrow{.15cm}\) 143 \(\Arrow{.15cm}\) 139 \(\Arrow{.15cm}\) 135 & 25m & 63 & 56 \(\Arrow{.15cm}\) 52 \(\Arrow{.15cm}\) 52 \(\Arrow{.15cm}\) 52 & 25m & 40 & 40 \(\Arrow{.15cm}\) 39 \(\Arrow{.15cm}\) 39 \(\Arrow{.15cm}\) 39 & 25m \\
                          & \(\rho:0.95\) & \(\omega:34\) & \texttt{tc+rs} &  & 141 \(\Arrow{.15cm}\) 84 \(\Arrow{.15cm}\) 64 \(\Arrow{.15cm}\) 46 & 25m &  & 134 \(\Arrow{.15cm}\) 87 \(\Arrow{.15cm}\) 65 \(\Arrow{.15cm}\) 44 & 25m &  & 44 \(\Arrow{.15cm}\) 40 \(\Arrow{.15cm}\) 39 \(\Arrow{.15cm}\) \textbf{34} & 20m12s &  & 40 \(\Arrow{.15cm}\) 36 \(\Arrow{.15cm}\) 36 \(\Arrow{.15cm}\) 36 & 25m \\ \hline
                evil-s3m25x8 & \(n:200\) & \(m:18350\) & \texttt{tc} & 200 & 177 \(\Arrow{.15cm}\) 135 \(\Arrow{.15cm}\) 130 \(\Arrow{.15cm}\) 126 & 25m & 188 & 169 \(\Arrow{.15cm}\) 142 \(\Arrow{.15cm}\) 137 \(\Arrow{.15cm}\) 132 & 25m & 66 & 58 \(\Arrow{.15cm}\) 52 \(\Arrow{.15cm}\) 52 \(\Arrow{.15cm}\) 52 & 25m & 40 & 40 \(\Arrow{.15cm}\) 38 \(\Arrow{.15cm}\) 38 \(\Arrow{.15cm}\) 37 & 25m \\
                          & \(\rho:0.92\) & \(\omega:32\) & \texttt{tc+rs} &  & 177 \(\Arrow{.15cm}\) 136 \(\Arrow{.15cm}\) 134 \(\Arrow{.15cm}\) 134 & 25m &  & 169 \(\Arrow{.15cm}\) 129 \(\Arrow{.15cm}\) 129 \(\Arrow{.15cm}\) 129 & 25m &  & 58 \(\Arrow{.15cm}\) 52 \(\Arrow{.15cm}\) 49 \(\Arrow{.15cm}\) 44 & 25m &  & 40 \(\Arrow{.15cm}\) 38 \(\Arrow{.15cm}\) 38 \(\Arrow{.15cm}\) 37 & 25m \\ \hline
                evil-myc5x42 & \(n:210\) & \(m:21404\) & \texttt{tc} & 210 & 197 \(\Arrow{.15cm}\) 181 \(\Arrow{.15cm}\) 180 \(\Arrow{.15cm}\) 179 & 25m & 207 & 194 \(\Arrow{.15cm}\) 183 \(\Arrow{.15cm}\) 181 \(\Arrow{.15cm}\) 180 & 25m & 106 & 99 \(\Arrow{.15cm}\) 94 \(\Arrow{.15cm}\) 94 \(\Arrow{.15cm}\) 94 & 25m & 90 & 90 \(\Arrow{.15cm}\) 87 \(\Arrow{.15cm}\) 87 \(\Arrow{.15cm}\) 85 & 25m \\
                          & \(\rho:0.98\) & \(\omega:84\) & \texttt{tc+rs} &  & 162 \(\Arrow{.15cm}\) 101 \(\Arrow{.15cm}\) 88 \(\Arrow{.15cm}\) \textbf{84} & 8m4s &  & 149 \(\Arrow{.15cm}\) 99 \(\Arrow{.15cm}\) 88 \(\Arrow{.15cm}\) \textbf{84} & 9m30s &  & 102 \(\Arrow{.15cm}\) \textbf{84} & 2s &  & 90 \(\Arrow{.15cm}\) 89 \(\Arrow{.15cm}\) \textbf{84} & 3m8s \\ \hline
                evil-myc11x20 & \(n:220\) & \(m:22960\) & \texttt{tc} & 220 & 201 \(\Arrow{.15cm}\) 169 \(\Arrow{.15cm}\) 165 \(\Arrow{.15cm}\) 162 & 25m & 214 & 197 \(\Arrow{.15cm}\) 176 \(\Arrow{.15cm}\) 172 \(\Arrow{.15cm}\) 168 & 25m & 73 & 66 \(\Arrow{.15cm}\) 61 \(\Arrow{.15cm}\) 61 \(\Arrow{.15cm}\) 61 & 25m & 47 & 47 \(\Arrow{.15cm}\) 47 \(\Arrow{.15cm}\) 47 \(\Arrow{.15cm}\) 46 & 25m \\
                          & \(\rho:0.95\) & \(\omega:40\) & \texttt{tc+rs} &  & 192 \(\Arrow{.15cm}\) 148 \(\Arrow{.15cm}\) 147 \(\Arrow{.15cm}\) 145 & 25m &  & 183 \(\Arrow{.15cm}\) 147 \(\Arrow{.15cm}\) 146 \(\Arrow{.15cm}\) 141 & 25m &  & 60 \(\Arrow{.15cm}\) 51 \(\Arrow{.15cm}\) 50 \(\Arrow{.15cm}\) 48 & 25m &  & 47 \(\Arrow{.15cm}\) 45 \(\Arrow{.15cm}\) 43 \(\Arrow{.15cm}\) 43 & 25m \\ \hline
                evil-myc23x10 & \(n:230\) & \(m:24072\) & \texttt{tc} & 230 & 201 \(\Arrow{.15cm}\) 144 \(\Arrow{.15cm}\) 138 \(\Arrow{.15cm}\) 132 & 25m & 217 & 192 \(\Arrow{.15cm}\) 157 \(\Arrow{.15cm}\) 149 \(\Arrow{.15cm}\) 140 & 25m & 48 & 40 \(\Arrow{.15cm}\) 36 \(\Arrow{.15cm}\) 36 \(\Arrow{.15cm}\) 36 & 25m & 25 & 25 \(\Arrow{.15cm}\) 24 \(\Arrow{.15cm}\) 24 \(\Arrow{.15cm}\) 24 & 25m \\
                          & \(\rho:0.91\) & \(\omega:20\) & \texttt{tc+rs} &  & 201 \(\Arrow{.15cm}\) 146 \(\Arrow{.15cm}\) 126 \(\Arrow{.15cm}\) 119 & 25m &  & 192 \(\Arrow{.15cm}\) 137 \(\Arrow{.15cm}\) 125 \(\Arrow{.15cm}\) 111 & 25m &  & 40 \(\Arrow{.15cm}\) 36 \(\Arrow{.15cm}\) 36 \(\Arrow{.15cm}\) 34 & 25m &  & 25 \(\Arrow{.15cm}\) 24 \(\Arrow{.15cm}\) 24 \(\Arrow{.15cm}\) 24 & 25m \\ \hline
                evil-chv12x20 & \(n:240\) & \(m:27328\) & \texttt{tc} & 240 & 222 \(\Arrow{.15cm}\) 188 \(\Arrow{.15cm}\) 183 \(\Arrow{.15cm}\) 179 & 25m & 233 & 216 \(\Arrow{.15cm}\) 192 \(\Arrow{.15cm}\) 187 \(\Arrow{.15cm}\) 183 & 25m & 74 & 68 \(\Arrow{.15cm}\) 63 \(\Arrow{.15cm}\) 63 \(\Arrow{.15cm}\) 62 & 25m & 49 & 49 \(\Arrow{.15cm}\) 47 \(\Arrow{.15cm}\) 47 \(\Arrow{.15cm}\) 46 & 25m \\
                          & \(\rho:0.95\) & \(\omega:40\) & \texttt{tc+rs} &  & 216 \(\Arrow{.15cm}\) 184 \(\Arrow{.15cm}\) 184 \(\Arrow{.15cm}\) 184 & 25m &  & 210 \(\Arrow{.15cm}\) 186 \(\Arrow{.15cm}\) 186 \(\Arrow{.15cm}\) 186 & 25m &  & 65 \(\Arrow{.15cm}\) 61 \(\Arrow{.15cm}\) 59 \(\Arrow{.15cm}\) 55 & 25m &  & 49 \(\Arrow{.15cm}\) 47 \(\Arrow{.15cm}\) 47 \(\Arrow{.15cm}\) 46 & 25m \\ \hline
                evil-myc5x48 & \(n:240\) & \(m:27962\) & \texttt{tc} & 240 & 226 \(\Arrow{.15cm}\) 207 \(\Arrow{.15cm}\) 205 \(\Arrow{.15cm}\) 204 & 25m & 236 & 223 \(\Arrow{.15cm}\) 210 \(\Arrow{.15cm}\) 207 \(\Arrow{.15cm}\) 205 & 25m & 120 & 114 \(\Arrow{.15cm}\) 108 \(\Arrow{.15cm}\) 108 \(\Arrow{.15cm}\) 108 & 25m & 100 & 100 \(\Arrow{.15cm}\) 98 \(\Arrow{.15cm}\) 97 \(\Arrow{.15cm}\) \textbf{96} & 15m59s \\
                          & \(\rho:0.97\) & \(\omega:96\) & \texttt{tc+rs} &  & 223 \(\Arrow{.15cm}\) 175 \(\Arrow{.15cm}\) 167 \(\Arrow{.15cm}\) 161 & 25m &  & 214 \(\Arrow{.15cm}\) 177 \(\Arrow{.15cm}\) 166 \(\Arrow{.15cm}\) 160 & 25m &  & 119 \(\Arrow{.15cm}\) 97 \(\Arrow{.15cm}\) \textbf{96} & 2m4s &  & 100 \(\Arrow{.15cm}\) 99 \(\Arrow{.15cm}\) \textbf{96} & 2m32s \\ \hline
                evil-s3m25x10 & \(n:250\) & \(m:29075\) & \texttt{tc} & 250 & 227 \(\Arrow{.15cm}\) 182 \(\Arrow{.15cm}\) 177 \(\Arrow{.15cm}\) 172 & 25m & 238 & 219 \(\Arrow{.15cm}\) 190 \(\Arrow{.15cm}\) 186 \(\Arrow{.15cm}\) 176 & 25m & 81 & 73 \(\Arrow{.15cm}\) 66 \(\Arrow{.15cm}\) 66 \(\Arrow{.15cm}\) 66 & 25m & 50 & 50 \(\Arrow{.15cm}\) 49 \(\Arrow{.15cm}\) 48 \(\Arrow{.15cm}\) 47 & 25m \\
                          & \(\rho:0.93\) & \(\omega:40\) & \texttt{tc+rs} &  & 227 \(\Arrow{.15cm}\) 184 \(\Arrow{.15cm}\) 184 \(\Arrow{.15cm}\) 184 & 25m &  & 219 \(\Arrow{.15cm}\) 181 \(\Arrow{.15cm}\) 175 \(\Arrow{.15cm}\) 175 & 25m &  & 73 \(\Arrow{.15cm}\) 66 \(\Arrow{.15cm}\) 65 \(\Arrow{.15cm}\) 62 & 25m &  & 50 \(\Arrow{.15cm}\) 49 \(\Arrow{.15cm}\) 48 \(\Arrow{.15cm}\) 47 & 25m \\ \hline
                evil-myc11x23 & \(n:253\) & \(m:30422\) & \texttt{tc} & 253 & 232 \(\Arrow{.15cm}\) 199 \(\Arrow{.15cm}\) 195 \(\Arrow{.15cm}\) 190 & 25m & 247 & 226 \(\Arrow{.15cm}\) 203 \(\Arrow{.15cm}\) 198 \(\Arrow{.15cm}\) 194 & 25m & 84 & 76 \(\Arrow{.15cm}\) 71 \(\Arrow{.15cm}\) 71 \(\Arrow{.15cm}\) 71 & 25m & 55 & 55 \(\Arrow{.15cm}\) 54 \(\Arrow{.15cm}\) 54 \(\Arrow{.15cm}\) 53 & 25m \\
                          & \(\rho:0.95\) & \(\omega:46\) & \texttt{tc+rs} &  & 224 \(\Arrow{.15cm}\) 194 \(\Arrow{.15cm}\) 194 \(\Arrow{.15cm}\) 194 & 25m &  & 219 \(\Arrow{.15cm}\) 195 \(\Arrow{.15cm}\) 195 \(\Arrow{.15cm}\) 195 & 25m &  & 74 \(\Arrow{.15cm}\) 68 \(\Arrow{.15cm}\) 66 \(\Arrow{.15cm}\) 62 & 25m &  & 55 \(\Arrow{.15cm}\) 53 \(\Arrow{.15cm}\) 53 \(\Arrow{.15cm}\) 52 & 25m \\ \hline
                evil-myc5x60 & \(n:300\) & \(m:43817\) & \texttt{tc} & 300 & 284 \(\Arrow{.15cm}\) 262 \(\Arrow{.15cm}\) 260 \(\Arrow{.15cm}\) 259 & 25m & 296 & 280 \(\Arrow{.15cm}\) 264 \(\Arrow{.15cm}\) 262 \(\Arrow{.15cm}\) 261 & 25m & 149 & 143 \(\Arrow{.15cm}\) 139 \(\Arrow{.15cm}\) 134 \(\Arrow{.15cm}\) 134 & 25m & 122 & 122 \(\Arrow{.15cm}\) 121 \(\Arrow{.15cm}\) \textbf{120} & 1m31s \\
                          & \(\rho:0.98\) & \(\omega:120\) & \texttt{tc+rs} &  & 300 \(\Arrow{.15cm}\) 249 \(\Arrow{.15cm}\) 244 \(\Arrow{.15cm}\) 244 & 25m &  & 296 \(\Arrow{.15cm}\) 246 \(\Arrow{.15cm}\) 238 \(\Arrow{.15cm}\) 238 & 25m &  & 149 \(\Arrow{.15cm}\) 127 \(\Arrow{.15cm}\) 123 \(\Arrow{.15cm}\) 121 & 25m &  & 122 \(\Arrow{.15cm}\) 122 \(\Arrow{.15cm}\) \textbf{120} & 2m22s \\ \hline
                gen200-p0-9-44 & \(n:200\) & \(m:17910\) & \texttt{tc} & 200 & 164 \(\Arrow{.15cm}\) 116 \(\Arrow{.15cm}\) 112 \(\Arrow{.15cm}\) 108 & 25m & 189 & 156 \(\Arrow{.15cm}\) 121 \(\Arrow{.15cm}\) 116 \(\Arrow{.15cm}\) 111 & 25m & 62 & 53 \(\Arrow{.15cm}\) 48 \(\Arrow{.15cm}\) 48 \(\Arrow{.15cm}\) 47 & 25m & \textbf{44} & \textbf{44} & \textless{}1s \\
                          & \(\rho:0.90\) & \(\omega:44\) & \texttt{tc+rs} &  & 165 \(\Arrow{.15cm}\) 113 \(\Arrow{.15cm}\) 107 \(\Arrow{.15cm}\) 102 & 25m &  & 156 \(\Arrow{.15cm}\) 112 \(\Arrow{.15cm}\) 107 \(\Arrow{.15cm}\) 104 & 25m &  & 53 \(\Arrow{.15cm}\) 48 \(\Arrow{.15cm}\) 48 \(\Arrow{.15cm}\) \textbf{44} & 23m15s &  & \textbf{44} & \textless{}1s \\ \hline
                gen200-p0-9-55 & \(n:200\) & \(m:17910\) & \texttt{tc} & 200 & 164 \(\Arrow{.15cm}\) 117 \(\Arrow{.15cm}\) 113 \(\Arrow{.15cm}\) 108 & 25m & 189 & 155 \(\Arrow{.15cm}\) 120 \(\Arrow{.15cm}\) 116 \(\Arrow{.15cm}\) 111 & 25m & 71 & 60 \(\Arrow{.15cm}\) \textbf{55} & 13s & \textbf{55} & \textbf{55} & \textless{}1s \\
                          & \(\rho:0.90\) & \(\omega:55\) & \texttt{tc+rs} &  & 164 \(\Arrow{.15cm}\) 117 \(\Arrow{.15cm}\) 114 \(\Arrow{.15cm}\) 112 & 25m &  & 155 \(\Arrow{.15cm}\) 115 \(\Arrow{.15cm}\) 115 \(\Arrow{.15cm}\) 112 & 25m &  & 60 \(\Arrow{.15cm}\) \textbf{55} & 13s &  & \textbf{55} & \textless{}1s \\ \hline
                hamming8-2 & \(n:256\) & \(m:31616\) & \texttt{tc} & 256 & 247 \(\Arrow{.15cm}\) 225 \(\Arrow{.15cm}\) 222 \(\Arrow{.15cm}\) 220 & 25m & 248 & 242 \(\Arrow{.15cm}\) 225 \(\Arrow{.15cm}\) 222 \(\Arrow{.15cm}\) 221 & 25m & \textbf{128} & \textbf{128} & \textless{}1s & \textbf{128} & \textbf{128} & \textless{}1s \\
                          & \(\rho:0.97\) & \(\omega:128\) & \texttt{tc+rs} &  & 247 \(\Arrow{.15cm}\) 225 \(\Arrow{.15cm}\) 225 \(\Arrow{.15cm}\) 225 & 25m &  & 242 \(\Arrow{.15cm}\) 220 \(\Arrow{.15cm}\) 220 \(\Arrow{.15cm}\) 220 & 25m &  & \textbf{128} & \textless{}1s &  & \textbf{128} & \textless{}1s \\ \hline
                hamming8-4 & \(n:256\) & \(m:20864\) & \texttt{tc} & 256 & 156 \(\Arrow{.15cm}\) 26 \(\Arrow{.15cm}\) \textbf{16} & 4m11s & 164 & 104 \(\Arrow{.15cm}\) 31 \(\Arrow{.15cm}\) 22 \(\Arrow{.15cm}\) \textbf{16} & 13m23s & 24 & 17 \(\Arrow{.15cm}\) \textbf{16} & 6s & \textbf{16} & - \(\Arrow{.15cm}\) \textbf{16} & 46s \\
                          & \(\rho:0.64\) & \(\omega:16\) & \texttt{tc+rs} &  & 157 \(\Arrow{.15cm}\) 32 \(\Arrow{.15cm}\) 23 \(\Arrow{.15cm}\) 17 & 25m &  & 104 \(\Arrow{.15cm}\) 39 \(\Arrow{.15cm}\) 31 \(\Arrow{.15cm}\) 23 & 25m &  & 17 \(\Arrow{.15cm}\) \textbf{16} & 6s &  & - \(\Arrow{.15cm}\) \textbf{16} & 46s \\ \hline
                p-hat300-2 & \(n:300\) & \(m:21928\) & \texttt{tc} & 300 & 90 \(\Arrow{.15cm}\) 39 \(\Arrow{.15cm}\) 35 \(\Arrow{.15cm}\) 32 & 25m & 209 & 76 \(\Arrow{.15cm}\) 42 \(\Arrow{.15cm}\) 37 \(\Arrow{.15cm}\) 34 & 25m & 42 & 29 \(\Arrow{.15cm}\) \textbf{25} & 3s & 26 & - \(\Arrow{.15cm}\) - \(\Arrow{.15cm}\) - \(\Arrow{.15cm}\) \textbf{25} & 9m45s \\
                          & \(\rho:0.49\) & \(\omega:25\) & \texttt{tc+rs} &  & 92 \(\Arrow{.15cm}\) 29 \(\Arrow{.15cm}\) \textbf{25} & 2m20s &  & 77 \(\Arrow{.15cm}\) 31 \(\Arrow{.15cm}\) \textbf{25} & 3m45s &  & 29 \(\Arrow{.15cm}\) \textbf{25} & 3s &  & - \(\Arrow{.15cm}\) - \(\Arrow{.15cm}\) - \(\Arrow{.15cm}\) \textbf{25} & 9m45s \\ \hline
                p-hat300-3 & \(n:300\) & \(m:33390\) & \texttt{tc} & 300 & 180 \(\Arrow{.15cm}\) 93 \(\Arrow{.15cm}\) 86 \(\Arrow{.15cm}\) 79 & 25m & 258 & 158 \(\Arrow{.15cm}\) 96 \(\Arrow{.15cm}\) 89 \(\Arrow{.15cm}\) 83 & 25m & 69 & 58 \(\Arrow{.15cm}\) 41 \(\Arrow{.15cm}\) 40 \(\Arrow{.15cm}\) 40 & 25m & 41 & - \(\Arrow{.15cm}\) - \(\Arrow{.15cm}\) 41 \(\Arrow{.15cm}\) \textbf{36} & 20m27s \\
                          & \(\rho:0.74\) & \(\omega:36\) & \texttt{tc+rs} &  & 180 \(\Arrow{.15cm}\) 98 \(\Arrow{.15cm}\) 76 \(\Arrow{.15cm}\) 65 & 25m &  & 158 \(\Arrow{.15cm}\) 95 \(\Arrow{.15cm}\) 77 \(\Arrow{.15cm}\) 67 & 25m &  & 59 \(\Arrow{.15cm}\) 41 \(\Arrow{.15cm}\) 39 \(\Arrow{.15cm}\) \textbf{36} & 13m54s &  & - \(\Arrow{.15cm}\) - \(\Arrow{.15cm}\) 41 \(\Arrow{.15cm}\) \textbf{36} & 20m26s \\ \hline
                p-hat500-1 & \(n:500\) & \(m:31569\) & \texttt{tc} & 500 & 79 \(\Arrow{.15cm}\) \textbf{9} & 7s & 205 & 39 \(\Arrow{.15cm}\) \textbf{9} & 6s & 32 & 13 \(\Arrow{.15cm}\) \textbf{9} & 3s & - & - & 25m \\
                          & \(\rho:0.25\) & \(\omega:9\) & \texttt{tc+rs} &  & 170 \(\Arrow{.15cm}\) \textbf{9} & 8s &  & 46 \(\Arrow{.15cm}\) \textbf{9} & 7s &  & 22 \(\Arrow{.15cm}\) \textbf{9} & 4s &  & - & 25m \\ \hline
                san200-0-9-1 & \(n:200\) & \(m:17910\) & \texttt{tc} & 200 & 160 \(\Arrow{.15cm}\) 116 \(\Arrow{.15cm}\) 116 \(\Arrow{.15cm}\) 115 & 25m & 189 & 151 \(\Arrow{.15cm}\) 117 \(\Arrow{.15cm}\) 117 \(\Arrow{.15cm}\) 116 & 25m & 73 & \textbf{70} & \textless{}1s & \textbf{70} & \textbf{70} & \textless{}1s \\
                          & \(\rho:0.90\) & \(\omega:70\) & \texttt{tc+rs} &  & 160 \(\Arrow{.15cm}\) 105 \(\Arrow{.15cm}\) 99 \(\Arrow{.15cm}\) 98 & 25m &  & 151 \(\Arrow{.15cm}\) 105 \(\Arrow{.15cm}\) 100 \(\Arrow{.15cm}\) 99 & 25m &  & \textbf{70} & \textless{}1s &  & \textbf{70} & \textless{}1s \\ \hline
                san200-0-9-2 & \(n:200\) & \(m:17910\) & \texttt{tc} & 200 & 167 \(\Arrow{.15cm}\) 116 \(\Arrow{.15cm}\) 112 \(\Arrow{.15cm}\) 108 & 25m & 189 & 158 \(\Arrow{.15cm}\) 121 \(\Arrow{.15cm}\) 116 \(\Arrow{.15cm}\) 112 & 25m & 75 & 63 \(\Arrow{.15cm}\) \textbf{60} & 8s & \textbf{60} & \textbf{60} & \textless{}1s \\
                          & \(\rho:0.90\) & \(\omega:60\) & \texttt{tc+rs} &  & 167 \(\Arrow{.15cm}\) 123 \(\Arrow{.15cm}\) 123 \(\Arrow{.15cm}\) 118 & 25m &  & 158 \(\Arrow{.15cm}\) 117 \(\Arrow{.15cm}\) 117 \(\Arrow{.15cm}\) 116 & 25m &  & 63 \(\Arrow{.15cm}\) \textbf{60} & 8s &  & \textbf{60} & \textless{}1s \\ \hline
                san200-0-9-3 & \(n:200\) & \(m:17910\) & \texttt{tc} & 200 & 167 \(\Arrow{.15cm}\) 116 \(\Arrow{.15cm}\) 112 \(\Arrow{.15cm}\) 107 & 25m & 188 & 157 \(\Arrow{.15cm}\) 122 \(\Arrow{.15cm}\) 117 \(\Arrow{.15cm}\) 112 & 25m & 64 & 54 \(\Arrow{.15cm}\) 48 \(\Arrow{.15cm}\) 47 \(\Arrow{.15cm}\) 47 & 25m & \textbf{44} & - \(\Arrow{.15cm}\) \textbf{44} & 1s \\
                          & \(\rho:0.90\) & \(\omega:44\) & \texttt{tc+rs} &  & 167 \(\Arrow{.15cm}\) 125 \(\Arrow{.15cm}\) 123 \(\Arrow{.15cm}\) 119 & 25m &  & 157 \(\Arrow{.15cm}\) 117 \(\Arrow{.15cm}\) 117 \(\Arrow{.15cm}\) 116 & 25m &  & 54 \(\Arrow{.15cm}\) 48 \(\Arrow{.15cm}\) 47 \(\Arrow{.15cm}\) 45 & 25m &  & - \(\Arrow{.15cm}\) \textbf{44} & 1s \\ \hline
                san400-0-5-1 & \(n:400\) & \(m:39900\) & \texttt{tc} & 400 & 182 \(\Arrow{.15cm}\) 93 \(\Arrow{.15cm}\) 75 \(\Arrow{.15cm}\) 59 & 25m & 226 & 153 \(\Arrow{.15cm}\) 94 \(\Arrow{.15cm}\) 84 \(\Arrow{.15cm}\) 65 & 25m & 21 & \textbf{13} & \textless{}1s & - & - & 25m \\
                          & \(\rho:0.50\) & \(\omega:13\) & \texttt{tc+rs} &  & 183 \(\Arrow{.15cm}\) 93 \(\Arrow{.15cm}\) 37 \(\Arrow{.15cm}\) 16 & 25m &  & 154 \(\Arrow{.15cm}\) 89 \(\Arrow{.15cm}\) 65 \(\Arrow{.15cm}\) 58 & 25m &  & \textbf{13} & \textless{}1s &  & - & 25m \\ \hline
                sanr200-0-9 & \(n:200\) & \(m:17863\) & \texttt{tc} & 200 & 163 \(\Arrow{.15cm}\) 113 \(\Arrow{.15cm}\) 108 \(\Arrow{.15cm}\) 104 & 25m & 189 & 155 \(\Arrow{.15cm}\) 118 \(\Arrow{.15cm}\) 113 \(\Arrow{.15cm}\) 109 & 25m & 74 & 65 \(\Arrow{.15cm}\) 57 \(\Arrow{.15cm}\) 56 \(\Arrow{.15cm}\) 56 & 25m & 49 & 49 \(\Arrow{.15cm}\) 46 \(\Arrow{.15cm}\) 45 \(\Arrow{.15cm}\) \textbf{42} & 23m38s \\
                          & \(\rho:0.90\) & \(\omega:42\) & \texttt{tc+rs} &  & 164 \(\Arrow{.15cm}\) 118 \(\Arrow{.15cm}\) 112 \(\Arrow{.15cm}\) 110 & 25m &  & 155 \(\Arrow{.15cm}\) 115 \(\Arrow{.15cm}\) 114 \(\Arrow{.15cm}\) 112 & 25m &  & 65 \(\Arrow{.15cm}\) 56 \(\Arrow{.15cm}\) 54 \(\Arrow{.15cm}\) 50 & 25m &  & 49 \(\Arrow{.15cm}\) 46 \(\Arrow{.15cm}\) 45 \(\Arrow{.15cm}\) \textbf{42} & 23m36s \\ \hline
                sanr400-0-5 & \(n:400\) & \(m:39984\) & \texttt{tc} & 400 & 176 \(\Arrow{.15cm}\) 35 \(\Arrow{.15cm}\) 19 \(\Arrow{.15cm}\) \textbf{13} & 13m43s & 234 & 107 \(\Arrow{.15cm}\) 35 \(\Arrow{.15cm}\) 22 \(\Arrow{.15cm}\) 15 & 25m & 56 & 46 \(\Arrow{.15cm}\) 16 \(\Arrow{.15cm}\) 14 \(\Arrow{.15cm}\) \textbf{13} & 5m26s & - & - & 25m \\
                          & \(\rho:0.50\) & \(\omega:13\) & \texttt{tc+rs} &  & 177 \(\Arrow{.15cm}\) 35 \(\Arrow{.15cm}\) 20 \(\Arrow{.15cm}\) 18 & 25m &  & 108 \(\Arrow{.15cm}\) 35 \(\Arrow{.15cm}\) 25 \(\Arrow{.15cm}\) 24 & 25m &  & 50 \(\Arrow{.15cm}\) 16 \(\Arrow{.15cm}\) \textbf{13} & 4m19s &  & - & 25m \\ \hline
            \end{tabular}
        \end{adjustbox}
    }
    \caption{Upper-bound progress on medium graph instances with more than \(15000\) and fewer than \(45000\) edges.}
    \label{tab:medium}
\end{table}

%% file: first_minute_progress_plots.tex
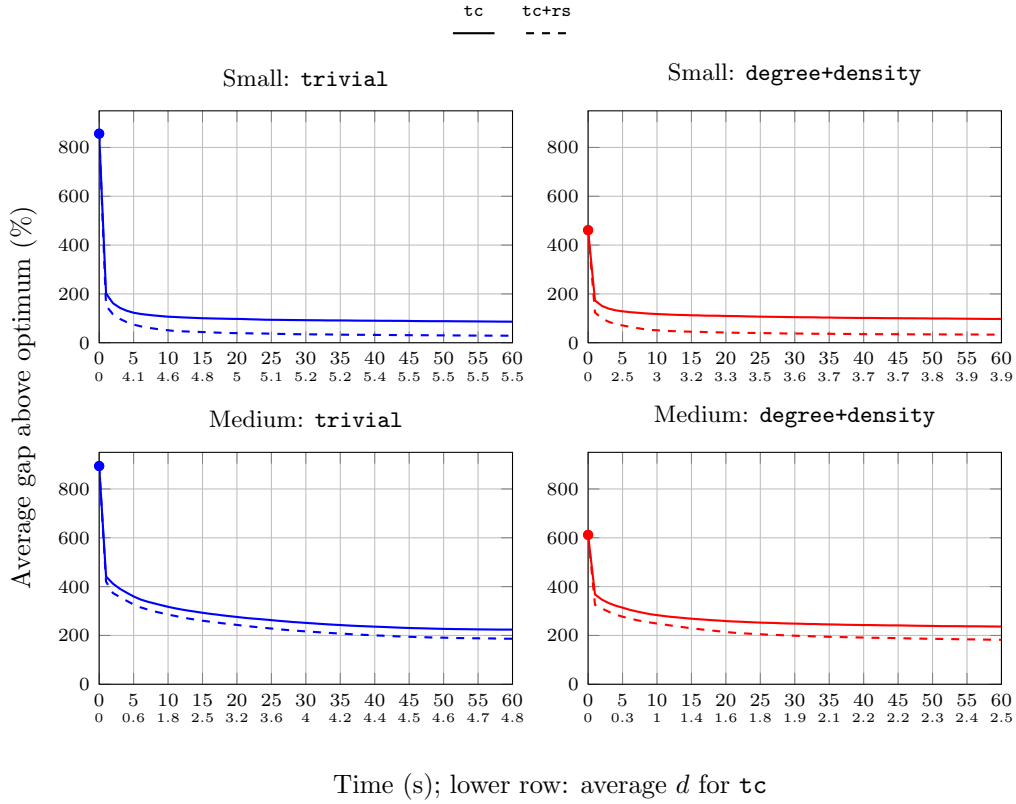
\begin{figure}[ht]
    \centering
    {\scriptsize
        \setlength{\tabcolsep}{5pt}
        \begin{tabular}{c c}
            \texttt{tc} & \texttt{tc+rs} \\
            \tikz[baseline=-0.5ex]\draw[black, solid, line width=0.8pt] (0,0) -- (0.55cm,0); & \tikz[baseline=-0.5ex]\draw[black, dashed, line width=0.8pt] (0,0) -- (0.55cm,0); \\
        \end{tabular}
        \par\vspace{0.4em}
    }
    \begin{tikzpicture}
        \begin{groupplot}[
            group style={group size=2 by 2, horizontal sep=1.0cm, vertical sep=1.45cm},
            width=0.47\textwidth,
            height=0.31\textwidth,
            xmin=0, xmax=60,
            ymin=0,
            ymax=950,
            xtick={0,5,10,15,20,25,30,35,40,45,50,55,60},
            grid=major,
            xlabel={},
            tick label style={font=\scriptsize},
            xticklabel style={align=center},
            label style={font=\small},
            title style={font=\small},
        ]
            \nextgroupplot[
                title={Small: \texttt{trivial}},
                xticklabels={{\shortstack{0\\{\tiny 0}}},{\shortstack{5\\{\tiny 4.1}}},{\shortstack{10\\{\tiny 4.6}}},{\shortstack{15\\{\tiny 4.8}}},{\shortstack{20\\{\tiny 5}}},{\shortstack{25\\{\tiny 5.1}}},{\shortstack{30\\{\tiny 5.2}}},{\shortstack{35\\{\tiny 5.2}}},{\shortstack{40\\{\tiny 5.4}}},{\shortstack{45\\{\tiny 5.5}}},{\shortstack{50\\{\tiny 5.5}}},{\shortstack{55\\{\tiny 5.5}}},{\shortstack{60\\{\tiny 5.5}}}},
            ]
                \addplot[color=blue, solid, mark=none, line width=0.8pt] coordinates {
                    (0,856.45)
                    (1,201.76)
                    (2,162.03)
                    (3,143.13)
                    (4,131.11)
                    (5,122.91)
                    (6,118.04)
                    (7,114.6)
                    (8,111.6)
                    (9,109.12)
                    (10,106.63)
                    (11,105.49)
                    (12,103.95)
                    (13,102.78)
                    (14,101.59)
                    (15,100.42)
                    (16,99.81)
                    (17,99.04)
                    (18,98.59)
                    (19,97.86)
                    (20,97.6)
                    (21,97.04)
                    (22,96.02)
                    (23,95.3)
                    (24,94.33)
                    (25,93.95)
                    (26,93.74)
                    (27,93.39)
                    (28,93.06)
                    (29,92.72)
                    (30,92.47)
                    (31,92.2)
                    (32,92)
                    (33,91.48)
                    (34,91.48)
                    (35,91.34)
                    (36,91.25)
                    (37,90.76)
                    (38,90.45)
                    (39,90.3)
                    (40,90.21)
                    (41,89.94)
                    (42,89.94)
                    (43,89.67)
                    (44,89.67)
                    (45,89.26)
                    (46,88.78)
                    (47,88.63)
                    (48,88.33)
                    (49,88.33)
                    (50,88.33)
                    (51,88.17)
                    (52,87.86)
                    (53,87.75)
                    (54,87.51)
                    (55,87.37)
                    (56,87.02)
                    (57,86.88)
                    (58,86.5)
                    (59,86.5)
                    (60,86.2)
                };
                \addplot[color=blue, only marks, mark=*, mark size=1.7pt, mark options={fill=blue, draw=blue}] coordinates {
                    (0,856.45)
                };
                \addplot[color=blue, dashed, mark=none, line width=0.8pt] coordinates {
                    (0,856.45)
                    (1,151.61)
                    (2,116.73)
                    (3,97.85)
                    (4,84.38)
                    (5,74.32)
                    (6,67.48)
                    (7,61.84)
                    (8,57.83)
                    (9,53.81)
                    (10,51.08)
                    (11,47.95)
                    (12,46.92)
                    (13,45.76)
                    (14,44.98)
                    (15,43.75)
                    (16,43.01)
                    (17,41.5)
                    (18,40.67)
                    (19,39.77)
                    (20,39.28)
                    (21,39.01)
                    (22,38.4)
                    (23,38.03)
                    (24,37.76)
                    (25,37.05)
                    (26,36.23)
                    (27,35.88)
                    (28,35.19)
                    (29,34.9)
                    (30,34.62)
                    (31,34.26)
                    (32,34.26)
                    (33,34.06)
                    (34,33.56)
                    (35,33.03)
                    (36,32.89)
                    (37,32.77)
                    (38,32.54)
                    (39,32.4)
                    (40,32.06)
                    (41,31.99)
                    (42,31.73)
                    (43,31.5)
                    (44,31.23)
                    (45,30.96)
                    (46,30.79)
                    (47,30.79)
                    (48,30.65)
                    (49,30.26)
                    (50,30.14)
                    (51,30)
                    (52,29.86)
                    (53,29.86)
                    (54,29.44)
                    (55,29.3)
                    (56,29.22)
                    (57,29.05)
                    (58,29.05)
                    (59,28.93)
                    (60,28.67)
                };
                \addplot[color=blue, only marks, mark=*, mark size=1.7pt, mark options={fill=blue, draw=blue}] coordinates {
                    (0,856.45)
                };

            \nextgroupplot[
                title={Small: \texttt{degree+density}},
                xticklabels={{\shortstack{0\\{\tiny 0}}},{\shortstack{5\\{\tiny 2.5}}},{\shortstack{10\\{\tiny 3}}},{\shortstack{15\\{\tiny 3.2}}},{\shortstack{20\\{\tiny 3.3}}},{\shortstack{25\\{\tiny 3.5}}},{\shortstack{30\\{\tiny 3.6}}},{\shortstack{35\\{\tiny 3.7}}},{\shortstack{40\\{\tiny 3.7}}},{\shortstack{45\\{\tiny 3.7}}},{\shortstack{50\\{\tiny 3.8}}},{\shortstack{55\\{\tiny 3.9}}},{\shortstack{60\\{\tiny 3.9}}}},
            ]
                \addplot[color=red, solid, mark=none, line width=0.8pt] coordinates {
                    (0,461.23)
                    (1,173.4)
                    (2,150.9)
                    (3,139.65)
                    (4,132.76)
                    (5,128.74)
                    (6,125.47)
                    (7,123.3)
                    (8,121.02)
                    (9,118.99)
                    (10,117.62)
                    (11,116.46)
                    (12,115.09)
                    (13,114.06)
                    (14,113.52)
                    (15,112.54)
                    (16,112.17)
                    (17,110.95)
                    (18,110.82)
                    (19,110.5)
                    (20,109.66)
                    (21,109.2)
                    (22,108.93)
                    (23,107.98)
                    (24,107.48)
                    (25,106.84)
                    (26,106.45)
                    (27,106.11)
                    (28,105.7)
                    (29,105.26)
                    (30,104.81)
                    (31,104.46)
                    (32,104.32)
                    (33,104.1)
                    (34,103.22)
                    (35,103.22)
                    (36,102.54)
                    (37,102.18)
                    (38,101.94)
                    (39,101.64)
                    (40,101.49)
                    (41,101.49)
                    (42,101.2)
                    (43,100.75)
                    (44,100.6)
                    (45,100.6)
                    (46,100.22)
                    (47,100.22)
                    (48,99.94)
                    (49,99.85)
                    (50,99.34)
                    (51,99.34)
                    (52,99.34)
                    (53,98.84)
                    (54,98.53)
                    (55,98.32)
                    (56,98.11)
                    (57,97.8)
                    (58,97.58)
                    (59,97.43)
                    (60,97.29)
                };
                \addplot[color=red, only marks, mark=*, mark size=1.7pt, mark options={fill=red, draw=red}] coordinates {
                    (0,461.23)
                };
                \addplot[color=red, dashed, mark=none, line width=0.8pt] coordinates {
                    (0,461.23)
                    (1,123.87)
                    (2,99.72)
                    (3,85.74)
                    (4,76.55)
                    (5,70.89)
                    (6,65.59)
                    (7,60.76)
                    (8,56.11)
                    (9,52.47)
                    (10,51.09)
                    (11,49.13)
                    (12,47.85)
                    (13,47.09)
                    (14,45.92)
                    (15,45.21)
                    (16,44.6)
                    (17,43.8)
                    (18,42.7)
                    (19,42.58)
                    (20,41.81)
                    (21,41.09)
                    (22,40.66)
                    (23,40.37)
                    (24,40.25)
                    (25,39.47)
                    (26,39.27)
                    (27,39.01)
                    (28,38.29)
                    (29,37.84)
                    (30,37.84)
                    (31,37.6)
                    (32,37.07)
                    (33,36.72)
                    (34,36.72)
                    (35,36.56)
                    (36,36.21)
                    (37,36.06)
                    (38,36.06)
                    (39,35.43)
                    (40,35.23)
                    (41,35.15)
                    (42,35.03)
                    (43,35.03)
                    (44,35.03)
                    (45,35.03)
                    (46,34.65)
                    (47,34.42)
                    (48,34.42)
                    (49,34.42)
                    (50,34.42)
                    (51,34.27)
                    (52,33.94)
                    (53,33.8)
                    (54,33.8)
                    (55,33.8)
                    (56,33.8)
                    (57,33.72)
                    (58,33.32)
                    (59,33.32)
                    (60,33.21)
                };
                \addplot[color=red, only marks, mark=*, mark size=1.7pt, mark options={fill=red, draw=red}] coordinates {
                    (0,461.23)
                };

            \nextgroupplot[
                title={Medium: \texttt{trivial}},
                xticklabels={{\shortstack{0\\{\tiny 0}}},{\shortstack{5\\{\tiny 0.6}}},{\shortstack{10\\{\tiny 1.8}}},{\shortstack{15\\{\tiny 2.5}}},{\shortstack{20\\{\tiny 3.2}}},{\shortstack{25\\{\tiny 3.6}}},{\shortstack{30\\{\tiny 4}}},{\shortstack{35\\{\tiny 4.2}}},{\shortstack{40\\{\tiny 4.4}}},{\shortstack{45\\{\tiny 4.5}}},{\shortstack{50\\{\tiny 4.6}}},{\shortstack{55\\{\tiny 4.7}}},{\shortstack{60\\{\tiny 4.8}}}},
            ]
                \addplot[color=blue, solid, mark=none, line width=0.8pt] coordinates {
                    (0,893.94)
                    (1,441.34)
                    (2,412.1)
                    (3,391.3)
                    (4,374.94)
                    (5,359.71)
                    (6,347.4)
                    (7,338.45)
                    (8,331.21)
                    (9,324.04)
                    (10,317.33)
                    (11,311.27)
                    (12,306.02)
                    (13,301.34)
                    (14,297.18)
                    (15,292.86)
                    (16,288.91)
                    (17,285.65)
                    (18,281.96)
                    (19,278.78)
                    (20,275.64)
                    (21,272.51)
                    (22,269.72)
                    (23,267.64)
                    (24,265.17)
                    (25,262.79)
                    (26,259.94)
                    (27,257.77)
                    (28,255.18)
                    (29,253.42)
                    (30,251.35)
                    (31,249.58)
                    (32,247.52)
                    (33,245.81)
                    (34,244.01)
                    (35,242.62)
                    (36,240.87)
                    (37,239.47)
                    (38,238.26)
                    (39,237.21)
                    (40,236.01)
                    (41,235)
                    (42,233.82)
                    (43,232.63)
                    (44,231.42)
                    (45,230.53)
                    (46,229.66)
                    (47,229.19)
                    (48,228.29)
                    (49,227.54)
                    (50,226.81)
                    (51,226.45)
                    (52,225.91)
                    (53,225.7)
                    (54,225.23)
                    (55,224.77)
                    (56,224.51)
                    (57,224.27)
                    (58,223.94)
                    (59,223.94)
                    (60,223.83)
                };
                \addplot[color=blue, only marks, mark=*, mark size=1.7pt, mark options={fill=blue, draw=blue}] coordinates {
                    (0,893.94)
                };
                \addplot[color=blue, dashed, mark=none, line width=0.8pt] coordinates {
                    (0,893.94)
                    (1,420.55)
                    (2,373.66)
                    (3,357.35)
                    (4,341.89)
                    (5,326.84)
                    (6,315.13)
                    (7,306.49)
                    (8,298.55)
                    (9,292)
                    (10,285.55)
                    (11,279.92)
                    (12,274.35)
                    (13,268.9)
                    (14,264.88)
                    (15,260.24)
                    (16,256.7)
                    (17,253.17)
                    (18,249.54)
                    (19,246.46)
                    (20,243.09)
                    (21,239.76)
                    (22,236.69)
                    (23,233.74)
                    (24,231.13)
                    (25,228.25)
                    (26,225.61)
                    (27,222.91)
                    (28,220.82)
                    (29,218.52)
                    (30,216.49)
                    (31,214.66)
                    (32,212.9)
                    (33,210.99)
                    (34,209.36)
                    (35,207.82)
                    (36,206.09)
                    (37,204.43)
                    (38,203.06)
                    (39,201.61)
                    (40,200.33)
                    (41,199.15)
                    (42,197.61)
                    (43,196.61)
                    (44,195.56)
                    (45,194.49)
                    (46,193.39)
                    (47,192.89)
                    (48,191.83)
                    (49,191.14)
                    (50,190.77)
                    (51,190.2)
                    (52,189.3)
                    (53,188.99)
                    (54,188.51)
                    (55,188.25)
                    (56,187.99)
                    (57,187.69)
                    (58,186.89)
                    (59,186.59)
                    (60,186.5)
                };
                \addplot[color=blue, only marks, mark=*, mark size=1.7pt, mark options={fill=blue, draw=blue}] coordinates {
                    (0,893.94)
                };

            \nextgroupplot[
                title={Medium: \texttt{degree+density}},
                xticklabels={{\shortstack{0\\{\tiny 0}}},{\shortstack{5\\{\tiny 0.3}}},{\shortstack{10\\{\tiny 1}}},{\shortstack{15\\{\tiny 1.4}}},{\shortstack{20\\{\tiny 1.6}}},{\shortstack{25\\{\tiny 1.8}}},{\shortstack{30\\{\tiny 1.9}}},{\shortstack{35\\{\tiny 2.1}}},{\shortstack{40\\{\tiny 2.2}}},{\shortstack{45\\{\tiny 2.2}}},{\shortstack{50\\{\tiny 2.3}}},{\shortstack{55\\{\tiny 2.4}}},{\shortstack{60\\{\tiny 2.5}}}},
            ]
                \addplot[color=red, solid, mark=none, line width=0.8pt] coordinates {
                    (0,611.99)
                    (1,367.9)
                    (2,346.84)
                    (3,332.69)
                    (4,321.97)
                    (5,313.77)
                    (6,305.42)
                    (7,298.55)
                    (8,292.7)
                    (9,287.17)
                    (10,283.17)
                    (11,279.67)
                    (12,276.22)
                    (13,273.81)
                    (14,271)
                    (15,268.57)
                    (16,266.53)
                    (17,264.48)
                    (18,262.49)
                    (19,260.91)
                    (20,258.87)
                    (21,257.28)
                    (22,256.31)
                    (23,254.97)
                    (24,253.74)
                    (25,252.5)
                    (26,251.69)
                    (27,251.14)
                    (28,249.92)
                    (29,248.94)
                    (30,248.49)
                    (31,247.84)
                    (32,246.82)
                    (33,246.65)
                    (34,245.73)
                    (35,245.13)
                    (36,244.7)
                    (37,243.93)
                    (38,243.45)
                    (39,242.97)
                    (40,242.48)
                    (41,242.48)
                    (42,241.95)
                    (43,241.46)
                    (44,241.22)
                    (45,241.22)
                    (46,240.64)
                    (47,240.03)
                    (48,239.52)
                    (49,239.33)
                    (50,238.9)
                    (51,238.21)
                    (52,238.02)
                    (53,238.02)
                    (54,237.88)
                    (55,237.43)
                    (56,237.43)
                    (57,237.28)
                    (58,236.89)
                    (59,236.6)
                    (60,236.27)
                };
                \addplot[color=red, only marks, mark=*, mark size=1.7pt, mark options={fill=red, draw=red}] coordinates {
                    (0,611.99)
                };
                \addplot[color=red, dashed, mark=none, line width=0.8pt] coordinates {
                    (0,611.99)
                    (1,325.57)
                    (2,310.77)
                    (3,296.81)
                    (4,285.7)
                    (5,276.83)
                    (6,269.88)
                    (7,263.16)
                    (8,257.2)
                    (9,252.94)
                    (10,248.73)
                    (11,244.82)
                    (12,240.78)
                    (13,236.83)
                    (14,233.07)
                    (15,229.36)
                    (16,226.11)
                    (17,222.78)
                    (18,220.02)
                    (19,217.16)
                    (20,214.03)
                    (21,211.56)
                    (22,208.92)
                    (23,207.88)
                    (24,206.33)
                    (25,204.91)
                    (26,203.43)
                    (27,202.33)
                    (28,201.26)
                    (29,200.11)
                    (30,198.84)
                    (31,197.92)
                    (32,197.13)
                    (33,196.13)
                    (34,195.4)
                    (35,194.3)
                    (36,193.8)
                    (37,193.25)
                    (38,192.44)
                    (39,191.88)
                    (40,191.3)
                    (41,190.49)
                    (42,190.2)
                    (43,189.7)
                    (44,189.44)
                    (45,188.68)
                    (46,188.18)
                    (47,187.44)
                    (48,187.09)
                    (49,186.52)
                    (50,186.3)
                    (51,185.62)
                    (52,185.27)
                    (53,184.82)
                    (54,184.48)
                    (55,183.93)
                    (56,183.85)
                    (57,183.21)
                    (58,182.79)
                    (59,182.29)
                    (60,182.03)
                };
                \addplot[color=red, only marks, mark=*, mark size=1.7pt, mark options={fill=red, draw=red}] coordinates {
                    (0,611.99)
                };

        \end{groupplot}
        \path (group c1r1.north west) -- node[pos=0.5, xshift=-1.0cm, rotate=90, anchor=center] {Average gap above optimum (\%)} (group c1r2.south west);
        \path (group c1r2.south west) -- node[pos=0.5, below=1.05cm, anchor=north] {Time (s); lower row: average \(d\) for \texttt{tc}} (group c2r2.south east);
    \end{tikzpicture}
    \caption{Average upper-bound gap during the first minute for \texttt{trivial} and \texttt{degree+density} on small and medium instances. Solid and dashed lines correspond to \texttt{tc} and \texttt{tc+rs}, respectively. In each \(x\)-axis tick, the upper entry gives time in seconds and the lower entry gives the average current \(d\) for \texttt{tc}.}
    \label{plot:minute_progress}
\end{figure}

%% file: heatmap_comparison.tex
\begin{figure}[ht]
    \centering
    \begin{tikzpicture}[
        celltext/.style={font=\scriptsize, text=black, fill=white, fill opacity=0.86, text opacity=1, inner sep=1.1pt},
        rowlabel/.style={font=\small, anchor=east},
        collabel/.style={font=\scriptsize, rotate=45, anchor=west},
        heatcell/.style={draw=white, line width=0.45pt, rounded corners=1.7pt},
        ratiocell/.style={draw=white, line width=0.45pt, rounded corners=1.7pt, fill=white},
    ]
        \node[collabel] at (2.12,-0.42) {\([0.0,0.1)\)};
        \node[collabel] at (3.20,-0.42) {\([0.1,0.2)\)};
        \node[collabel] at (4.28,-0.42) {\([0.2,0.3)\)};
        \node[collabel] at (5.36,-0.42) {\([0.3,0.4)\)};
        \node[collabel] at (6.44,-0.42) {\([0.4,0.5)\)};
        \node[collabel] at (7.52,-0.42) {\([0.5,0.6)\)};
        \node[collabel] at (8.60,-0.42) {\([0.6,0.7)\)};
        \node[collabel] at (9.68,-0.42) {\([0.7,0.8)\)};
        \node[collabel] at (10.76,-0.42) {\([0.8,0.9)\)};
        \node[collabel] at (11.84,-0.42) {\([0.9,1.0]\)};
        \node[rowlabel] at (1.84,-1.08) {Small};
        \filldraw[heatcell, fill=blue!90!white] (2.00,-0.72) rectangle (3.08,-1.44);
        \node[celltext] at (2.54,-1.08) {3/3};
        \filldraw[heatcell, fill=blue!90!white] (3.08,-0.72) rectangle (4.16,-1.44);
        \node[celltext] at (3.62,-1.08) {1/1};
        \filldraw[heatcell, fill=blue!90!white] (4.16,-0.72) rectangle (5.24,-1.44);
        \node[celltext] at (4.70,-1.08) {1/1};
        \filldraw[heatcell, fill=blue!90!white] (5.24,-0.72) rectangle (6.32,-1.44);
        \node[celltext] at (5.78,-1.08) {2/2};
        \filldraw[heatcell, fill=blue!90!white] (6.32,-0.72) rectangle (7.40,-1.44);
        \node[celltext] at (6.86,-1.08) {2/2};
        \filldraw[heatcell, fill=blue!90!white] (7.40,-0.72) rectangle (8.48,-1.44);
        \node[celltext] at (7.94,-1.08) {1/1};
        \filldraw[heatcell, fill=blue!90!white] (8.48,-0.72) rectangle (9.56,-1.44);
        \node[celltext] at (9.02,-1.08) {4/4};
        \filldraw[heatcell, fill=blue!67!white] (9.56,-0.72) rectangle (10.64,-1.44);
        \node[celltext] at (10.10,-1.08) {5/7};
        \filldraw[heatcell, fill=blue!40!white] (10.64,-0.72) rectangle (11.72,-1.44);
        \node[celltext] at (11.18,-1.08) {3/8};
        \filldraw[heatcell, fill=blue!61!white] (11.72,-0.72) rectangle (12.80,-1.44);
        \node[celltext] at (12.26,-1.08) {7/11};
        \node[font=\scriptsize, text=black] at (13.32,-1.08) {29/40};
        \node[rowlabel] at (1.84,-1.80) {Medium};
        \filldraw[heatcell, fill=gray!15] (2.00,-1.44) rectangle (3.08,-2.16);
        \node[celltext] at (2.54,-1.80) {0/0};
        \filldraw[heatcell, fill=blue!90!white] (3.08,-1.44) rectangle (4.16,-2.16);
        \node[celltext] at (3.62,-1.80) {1/1};
        \filldraw[heatcell, fill=blue!90!white] (4.16,-1.44) rectangle (5.24,-2.16);
        \node[celltext] at (4.70,-1.80) {1/1};
        \filldraw[heatcell, fill=gray!15] (5.24,-1.44) rectangle (6.32,-2.16);
        \node[celltext] at (5.78,-1.80) {0/0};
        \filldraw[heatcell, fill=blue!90!white] (6.32,-1.44) rectangle (7.40,-2.16);
        \node[celltext] at (6.86,-1.80) {2/2};
        \filldraw[heatcell, fill=blue!90!white] (7.40,-1.44) rectangle (8.48,-2.16);
        \node[celltext] at (7.94,-1.80) {2/2};
        \filldraw[heatcell, fill=blue!90!white] (8.48,-1.44) rectangle (9.56,-2.16);
        \node[celltext] at (9.02,-1.80) {1/1};
        \filldraw[heatcell, fill=blue!90!white] (9.56,-1.44) rectangle (10.64,-2.16);
        \node[celltext] at (10.10,-1.80) {1/1};
        \filldraw[heatcell, fill=blue!10!white] (10.64,-1.44) rectangle (11.72,-2.16);
        \node[celltext] at (11.18,-1.80) {0/4};
        \filldraw[heatcell, fill=blue!25!white] (11.72,-1.44) rectangle (12.80,-2.16);
        \node[celltext] at (12.26,-1.80) {4/21};
        \node[font=\scriptsize, text=black] at (13.32,-1.80) {12/33};
        \node[rowlabel] at (1.84,-2.52) {Combined};
        \filldraw[heatcell, fill=blue!90!white] (2.00,-2.16) rectangle (3.08,-2.88);
        \node[celltext] at (2.54,-2.52) {3/3};
        \filldraw[heatcell, fill=blue!90!white] (3.08,-2.16) rectangle (4.16,-2.88);
        \node[celltext] at (3.62,-2.52) {2/2};
        \filldraw[heatcell, fill=blue!90!white] (4.16,-2.16) rectangle (5.24,-2.88);
        \node[celltext] at (4.70,-2.52) {2/2};
        \filldraw[heatcell, fill=blue!90!white] (5.24,-2.16) rectangle (6.32,-2.88);
        \node[celltext] at (5.78,-2.52) {2/2};
        \filldraw[heatcell, fill=blue!90!white] (6.32,-2.16) rectangle (7.40,-2.88);
        \node[celltext] at (6.86,-2.52) {4/4};
        \filldraw[heatcell, fill=blue!90!white] (7.40,-2.16) rectangle (8.48,-2.88);
        \node[celltext] at (7.94,-2.52) {3/3};
        \filldraw[heatcell, fill=blue!90!white] (8.48,-2.16) rectangle (9.56,-2.88);
        \node[celltext] at (9.02,-2.52) {5/5};
        \filldraw[heatcell, fill=blue!70!white] (9.56,-2.16) rectangle (10.64,-2.88);
        \node[celltext] at (10.10,-2.52) {6/8};
        \filldraw[heatcell, fill=blue!30!white] (10.64,-2.16) rectangle (11.72,-2.88);
        \node[celltext] at (11.18,-2.52) {3/12};
        \filldraw[heatcell, fill=blue!38!white] (11.72,-2.16) rectangle (12.80,-2.88);
        \node[celltext] at (12.26,-2.52) {11/32};
        \node[font=\scriptsize, text=black] at (13.32,-2.52) {41/73};
        \node[rowlabel] at (1.84,-3.24) {Avg. speedup};
        \filldraw[ratiocell] (2.00,-2.88) rectangle (3.08,-3.60);
        \node[celltext] at (2.54,-3.24) {14,456x};
        \filldraw[ratiocell] (3.08,-2.88) rectangle (4.16,-3.60);
        \node[celltext] at (3.62,-3.24) {112,574x};
        \filldraw[ratiocell] (4.16,-2.88) rectangle (5.24,-3.60);
        \node[celltext] at (4.70,-3.24) {\(>\)354x};
        \filldraw[ratiocell] (5.24,-2.88) rectangle (6.32,-3.60);
        \node[celltext] at (5.78,-3.24) {90x};
        \filldraw[ratiocell] (6.32,-2.88) rectangle (7.40,-3.60);
        \node[celltext] at (6.86,-3.24) {150x};
        \filldraw[ratiocell] (7.40,-2.88) rectangle (8.48,-3.60);
        \node[celltext] at (7.94,-3.24) {\(>\)695x};
        \filldraw[ratiocell] (8.48,-2.88) rectangle (9.56,-3.60);
        \node[celltext] at (9.02,-3.24) {6.1x};
        \filldraw[ratiocell] (9.56,-2.88) rectangle (10.64,-3.60);
        \node[celltext] at (10.10,-3.24) {24x};
        \filldraw[ratiocell] (10.64,-2.88) rectangle (11.72,-3.60);
        \node[celltext] at (11.18,-3.24) {9.7x};
        \filldraw[ratiocell] (11.72,-2.88) rectangle (12.80,-3.60);
        \node[celltext] at (12.26,-3.24) {42x};
        \node[rowlabel] at (1.84,-3.96) {Avg. slowdown};
        \filldraw[ratiocell] (2.00,-3.60) rectangle (3.08,-4.32);
        \node[celltext] at (2.54,-3.96) {-};
        \filldraw[ratiocell] (3.08,-3.60) rectangle (4.16,-4.32);
        \node[celltext] at (3.62,-3.96) {-};
        \filldraw[ratiocell] (4.16,-3.60) rectangle (5.24,-4.32);
        \node[celltext] at (4.70,-3.96) {-};
        \filldraw[ratiocell] (5.24,-3.60) rectangle (6.32,-4.32);
        \node[celltext] at (5.78,-3.96) {-};
        \filldraw[ratiocell] (6.32,-3.60) rectangle (7.40,-4.32);
        \node[celltext] at (6.86,-3.96) {-};
        \filldraw[ratiocell] (7.40,-3.60) rectangle (8.48,-4.32);
        \node[celltext] at (7.94,-3.96) {-};
        \filldraw[ratiocell] (8.48,-3.60) rectangle (9.56,-4.32);
        \node[celltext] at (9.02,-3.96) {-};
        \filldraw[ratiocell] (9.56,-3.60) rectangle (10.64,-4.32);
        \node[celltext] at (10.10,-3.96) {1,181x};
        \filldraw[ratiocell] (10.64,-3.60) rectangle (11.72,-4.32);
        \node[celltext] at (11.18,-3.96) {\(>\)756x};
        \filldraw[ratiocell] (11.72,-3.60) rectangle (12.80,-4.32);
        \node[celltext] at (12.26,-3.96) {\(>\)1,299x};
        \filldraw[heatcell, fill=blue!10!white] (2.00,-4.64) rectangle (2.55,-4.84);
        \filldraw[heatcell, fill=blue!90!white] (2.65,-4.64) rectangle (3.20,-4.84);
        \node[anchor=west, font=\scriptsize] at (3.34,-4.74) {larger fraction of wins};
    \end{tikzpicture}
    \caption{Comparison of \texttt{dsatur} with \texttt{tc+rs} against direct computation of the initial \texttt{sdp} bound, grouped by edge-density bucket. The bottom rows report mean speedup and slowdown factors. For speedups, a leading \textgreater{} means that the \texttt{sdp} run did not finish within the time limit, so the reported speedup is only a lower bound. For slowdowns, it means that \texttt{dsatur} with \texttt{tc+rs} did not reach the target bound within the time limit, so the reported slowdown is only a lower bound.}
    \label{plot:heatmap}
\end{figure}
